\documentclass[a4paper,11pt]{amsart}
\usepackage[utf8]{inputenc}
\usepackage[T1]{fontenc}

\usepackage{newpxtext}
\usepackage{newpxmath}

\usepackage{mathtools, verbatim}
\usepackage{amssymb}
\usepackage{amsfonts}
\usepackage{amsthm}
\usepackage{calrsfs}
\usepackage{latexsym}
\usepackage{array}
\usepackage{bbm}
\usepackage{stmaryrd}
\usepackage{url} 
\usepackage[svgnames]{xcolor}
\usepackage{esint}

\usepackage[english,activeacute]{babel}
\usepackage{hyperref}
\usepackage{bm}

\newcommand{\m}[1]{
\ifdefequal{#1}{1}
{\mathbf{#1}} 
{\mathbb{#1}}
}

\newcommand{\q}[1]{\mathcal{#1}}

\renewcommand{\le}{\leqslant}
\renewcommand{\ge}{\geqslant}

\newcommand{\ga}{\gamma}
\newcommand{\Om}{\Omega}
\newcommand{\F}{\Phi}
\newcommand{\R}{\mathbb{R}}
\newcommand{\Z}{\mathbb{Z}}
\newcommand{\M}{\mathcal{M}}
\newcommand{\f}{\varphi}
\newcommand{\loc}{\mathrm{loc}}

\def\RR{\mathbb{R}}
\newcommand{\weak}{\rightharpoonup}

\newcommand{\set}[2]{\left\{ #1 \colon #2 \right\}}
\newcommand{\e}{\varepsilon}

\DeclareMathOperator{\Span}{\textrm{Span}}

\DeclareMathOperator{\Ham}{\textrm{Ham}}
\DeclareMathOperator{\sgn}{\textrm{sgn}}

\newcommand{\be}{\begin{equation}}
\newcommand{\ee}{\end{equation}}

\theoremstyle{plain}
\newtheorem{thm}{Theorem}[section]
\newtheorem*{thm*}{Theorem}
\newtheorem{prop}[thm]{Proposition}
\newtheorem{cor}[thm]{Corollary}
\newtheorem{lem}[thm]{Lemma}

\theoremstyle{definition}
\newtheorem{defi}[thm]{Definition}
\theoremstyle{remark}
\newtheorem{nb}[thm]{Remark}
\theoremstyle{example}
\newtheorem{exa}[thm]{Example}
\newtheorem{claim}[thm]{Claim}

\numberwithin{equation}{section}

\newcommand{\nd}{\noindent}
\newcommand{\ds}{\displaystyle}

\makeatletter
\def\paragraph{\@startsection{paragraph}{4}%
  \z@\z@{-\fontdimen2\font}%
  {\normalfont\bfseries}}
\makeatother

\title{Asymptotic stability of precessing domain walls for the Landau-Lifshitz-Gilbert equation in a nanowire with Dzyaloshinskii-Moriya interaction}
\author{Rapha\"el C\^ote \and Radu Ignat}
\thanks{The authors acknowledge partial support by the ANR projects MAToS  ANR-14-CE25-0009-01 and MOSICOF ANR-21-CE40-0004. }
\date{\today}

\allowdisplaybreaks

\begin{document} 

\begin{abstract}
We consider a ferromagnetic nanowire and we focus on an asymptotic regime where the  Dzya\-lo\-shin\-skii-Mo\-riya interaction is taken into account. 

First we prove a dimension reduction result via $\Gamma$-convergence that determines a limit functional $E$ defined for maps $m:\R\to \m S^2$ in the direction $e_1$ of the nanowire. The energy functional $E$ is invariant under translations in $e_1$ and rotations about the axis $e_1$. We fully classify the critical points of finite energy $E$ when a transition between $-e_1$ and $e_1$ is imposed; these transition layers are called (static) domain walls. 

The evolution of a domain wall by the Landau-Lifshitz-Gilbert equation associated to $E$ under the effect of an applied magnetic field $h(t)e_1$ depending on the time variable $t$ gives rise to the so-called \emph{precessing domain wall}. Our main result proves the asymptotic stability of precessing domain walls for small $h$ in $L^\infty([0, +\infty))$ and small $H^1(\R)$ perturbations of the static domain wall, up to a gauge which is intrinsic to invariances of the functional $E$.
\end{abstract}

\maketitle

\section{Introduction}

\subsection{A reduced model in a ferromagnetic nanowire with Dzyaloshinskii-Moriya interaction (DMI)}

We consider a nanowire modeled by a straight line $\m Re_1 \subset \m R^3$ where $e_1= (1,0,0)$, $e_2=(0,1,0)$, $e_3= (0,0,1)$ is the canonical basis of $\m R^3$ and
magnetisations 
$m = (m_1, m_2, m_3): \m R \to \m S^2$ of this nanowire, with values into the unit sphere $\m S^2\subset \m R^3$,  to which we associate the energy functional
\begin{align} \label{def:energy}
E(m) = \frac{1}{2} \int_\R |\partial_x m|^2 + 2\ga \partial_x m \cdot (e_1\wedge m)+ (1-m_1^2) \, dx,
\end{align}
where  $x$ is the variable in direction $e_1$ of the nanowire and $\gamma \in \m R$ is a given constant with 
\[ \ga^2<1. \]
Here, $\cdot$ and $\wedge$ are the scalar and cross product in $\m R^3$. As $m\in \m S^2$, we often use $1-m_1^2=m_2^2+m_3^2$. This model is obtained by $\Gamma$-convergence in a special regime for a ferromagnetic nanowire with Dzyaloshinskii-Moriya interaction (see Theorem~\ref{thm:gamma} in Section \ref{sec:model} below). In several places it will be convenient to work with the map $m$ in spherical coordinates $(\varphi, \theta)$ on the sphere $\m S^2$ with respect to the $e_1$ axis, that is
\begin{equation} \label{spheric_coord}
m = \begin{pmatrix} 
\cos \theta \\
\sin \theta \cos \varphi \\
\sin \theta \sin \varphi
\end{pmatrix}.
\end{equation}
Note that a continuous map $m:\R\to \m S^2$ avoiding the poles $\{ \pm e_1 \} = \{ (\pm 1,0,0) \}$ admits continuous spherical coordinates $(\varphi, \theta):\R\to \R^2$ in \eqref{spheric_coord}.

\bigskip

Our first task in this paper is to characterise the minimisers of the energy $E$ when a transition from $-e_1$ to $e_1$ is imposed at $\pm \infty$, which we call \emph{(static) domain walls}. The family of such minimisers is invariant under the following transformations: 
\begin{enumerate}
\item[$\bullet$]  translation in space $\tau_y m(x) = m(x-y)$ for $y\in \R$, and
\item[$\bullet$] rotation $\ds R_\phi = \begin{pmatrix} 
1 & 0 & 0 \\
0 & \cos \phi & -\sin \phi \\
0 & \sin \phi & \cos \phi
\end{pmatrix}$ about the axis $e_1$ and angle $\phi\in \R$.
\end{enumerate}
This is due to the invariance of the energy $E(m)$ under these transformations. 
Thus, we are naturally led to define a gauge (or action of the group)
\[ G=(\m R^2, +) \]
 over magnetisations $m: \m R \to \m S^2$ by setting for $g= (y, \phi) \in G$:
\be
\label{grup}
 g.m := \tau_y R_\phi m=R_\phi \tau_y m. \ee 
 Observe that this action is commutative (i.e.,  $g.(\tilde g.m)=(g+\tilde g).m=\tilde g.(g.m)$ for every $g, \tilde g\in G$), $g.e_1 = e_1$ and for any $g = (y,\phi) \in G$, $m,\tilde m: \m R \to \m S^2$,
\begin{gather} 
(g.m) \wedge (g. \tilde m) = g .(m \wedge \tilde m), \qquad (g.m) \cdot (g. \tilde m) =\tau_y (m \cdot \tilde m), \nonumber \\
\partial_y(g.m) = - g. \partial_x m, \qquad \partial_\phi (g. m) = e_1 \wedge g.m=g.(e_1\wedge m). \label{formula4}
\end{gather}
We refer to Appendix \ref{sec:tool} where we gather a few other useful algebraic identities.

\bigskip

Our second and main concern in this paper is to study the evolution of a magnetisation under the effect of an applied  magnetic field 
\[ H_a=h(t) e_1 \]
in the direction $e_1$ and of intensity $h(t)$ that is a \emph{given} scalar continuous function \emph{depending on the time variable $t$}. As the magnetisation moves, it creates the magnetic field $H$: this effective field $H$ is composed by the gradient $\delta E(m)$ of the energy $E$ together with the applied field $H_a$, i.e.,\begin{align} 
\label{def:H}
H &= H(m) =  -\delta E(m) + H_a\\ 
\nonumber
\textrm{with} \quad \quad \delta E(m)&= -\partial_{xx} m - 2 \gamma e_1 \wedge \partial_x m +m_2 e_2+m_3 e_3.
\end{align}
The magnetisation $m=m(t,x)$ evolves according to the Landau-Lifshitz-Gilbert equation:
\begin{equation} \label{ll} \tag{LLG}
\partial_t m = m \wedge H(m) - \alpha m \wedge (m \wedge H(m)),
\end{equation}
where $\alpha >0$ is a given damping coefficient. 
We are especially interested here in the study of the flow of \eqref{ll} near (static) domains walls. More precisely, the evolution of (static) domain walls by the equation \eqref{ll} gives rise to the so-called \emph{precessing domain walls} and we want to prove the asymptotic stability of these precessing domain walls for the applied magnetic field $H_a=h(t) e_1$ with $h$ small in $L^\infty([0, +\infty))$ and perturbation of the initial data (given by the static domain walls) that is small in $H^1(\R)$.

\subsection{Description of static and precessing domain walls}

Note that every configuration $m:\R\to \m S^2$ of finite energy $E(m)<\infty$ admits limits belonging to $\{\pm e_1\}$ as $x\to \pm \infty$ (see e.g. Lemma \ref{lem:infinit}). In the following, we focus on configurations such that
\[ m(\pm\infty)=(\pm1,0,0), \textrm{ i.e., } \lim_{x\to \pm \infty} m(x)=\pm e_1. \]
Our first result gives a complete classification of (static) domain walls: they actually all derive from two explicit domain walls $w_*^\pm$ (which, in spherical coordinates, corresponds to taking the opposite angle of $\theta_*$) under the gauge $G$ defined in \eqref{grup}. The precise statement is the following.
\begin{thm}[Static domain walls]
\label{thm_DM}
For $\ga^2<1$, every finite energy critical point $m$ of $E$ in \eqref{def:energy} connecting $\pm e_1$, i.e., 
\begin{align} \label{eq:dm_critical}
m \wedge \delta E(m) = 0 \quad \textrm{with} \quad E(m)<\infty \quad \textrm{and} \quad m(\pm\infty)=(\pm1,0,0),
\end{align}
has the form 
\[ m=g.w^\pm_* \quad \textrm{ for some }\quad g= (y,\phi) \in G, \]
where 
$w^\pm_*$ are given in spherical coordinates\footnote{This is unambiguous because $w^\pm_*$ does not touch $\pm e_1$ for any $x \in \m R$. In the canonical basis $(e_1, e_2, e_3)$, \eqref{formul1} can be rewritten as
\[ w^\pm_*=\tanh \bigg(\sqrt{1-\ga^2}x\bigg) e_1\pm \frac1{\cosh(\sqrt{1-\ga^2}x)}\bigg(\cos(\ga x)e_2-\sin(\ga x) e_3\bigg). \]} by $(\varphi_*,\pm \theta_*)$, i.e., 
\be
\label{formul1}
w^\pm_* = \begin{pmatrix}  \cos \theta_* \\ \pm \sin \theta_* \cos \varphi_* \\ \pm \sin \theta_* \sin \varphi_* \end{pmatrix}  \quad
\textrm{with } \, \,   \varphi_*(x)=-\ga x, \, \, \theta_*(x) = 2\arctan (e^{-\sqrt{1-\ga^2}x}), \, \,  x\in \R.
\ee 
\end{thm}

\bigskip

Equivalently, $\theta_*:\R\to (0,\pi)$ defined in the above theorem solves the first order ODE 
\begin{equation}
\label{eq:theta*}
\partial_x \theta_* = -\sqrt{1-\ga^2} \sin \theta_*, \quad \theta_*(-\infty)=\pi, \, \theta_*(+\infty)=0.
\end{equation}
Differentiating \eqref{formul1}, by the above ODE in $\theta_*$, we deduce that $w^\pm_*$ satisfies the system of first order ODEs (see \eqref{345} below):
\[ \partial_x w^\pm_*=\sqrt{1-\ga^2} w^\pm_* \wedge(e_1\wedge w^\pm_*)-\ga e_1\wedge w^\pm_*. \]
The case $\ga=0$ (i.e., absence of DMI) corresponds to (in-plane) static domain walls where a rotation in $\theta_*$ of $180^\circ$ takes place along the nanowire axis $e_1$; these transitions are called Bloch walls (see e.g. \cite{CL06a, KK_these}). We highlight that the novelty of Theorem \ref{thm_DM} consists in treating the more general case $\gamma^2<1$ of the Dzya\-lo\-shin\-skii-Mo\-riya interaction for critical points of finite energy $E$. If $\gamma\neq 0$, next to the rotation in $\theta_*$, the optimal transition layer carries out a rotation in $\f_*$ in the plane $(e_2, e_3)$ transversal to the nanowire axis $e_1$.

\medskip

The evolution of a (static) domain wall under the \eqref{ll} flow for the time-dependent applied magnetic field $H_a=h(t) e_1$ is given by the precessing domain wall:

\begin{cor}[Precessing domain walls] \label{cor:precessing_dm}
Let $\gamma^2<1$, $\alpha\in \R$ and $h\in L^1_{\loc} ([0, +\infty), \m R)$. Define the $W^{1,1}_\loc$ gauge $g_* = (y_*, \phi_*): t\in [0,+\infty) \to G$ by the initial conditions $y_*(0) =0$, $\phi_*(0) =0$ and the derivative $\dot g_* = (\dot y_*, \dot \phi_*)$ (in the $L^1_{\loc}$ sense) given by
\be
\label{eq_y_phi}
\forall t \ge 0, \quad \dot y_*=-\frac{\alpha h(t)}{\sqrt{1-\ga^2}}, \quad \dot \phi_* =\big(-1+\frac{\alpha \ga}{\sqrt{1-\ga^2}}\big)h(t).
\ee
Then the maps $(t,x)\in [0,+\infty) \times \R \mapsto g_*(t).w^\pm_*(x)\in \m S^2$ having spherical coordinates\footnote{The precessing domain wall represents the time-dependent translation of the static wall $w^\pm_*$ whose centre $y_*$ evolves by \eqref{eq_y_phi} that is combined with a time-dependent precession about the nanowire axis $e_1$ encoded in the rotation of angle $\phi_*$ in the plane $(e_2, e_3)$ evolving by \eqref{eq_y_phi}. Thus, this evolution is fundamentally different than the more common travelling waves (e.g., the Walker wall \cite{SW74}).}
\be
\label{preces}
\left( \varphi_*(x-y_*(t))+\phi_*(t),\pm \theta_*(x-y_*(t)) \right),
\ee
with $\varphi_*$ and $\theta_*$ given in \eqref{formul1}
are two solutions to \eqref{ll} with initial data $w^\pm_*$ at $t=0$. These solutions are called precessing domain walls.
\end{cor}

Note that in this statement, we do not require $\alpha >0$, although this condition is the physically relevant one, but we will require it in our stability result (see Theorem \ref{th:stab} below). In the case $\gamma=0$, the precessing domain walls were reported in \cite{GRS10}, whose linear asymptotic stability was proved in \cite{GGRS11} (the nonlinear asymptotic stability is verified only numerically in \cite{GGRS11} in the case $\gamma=0$). Our aim is to prove rigorously the nonlinear asymptotic stability of precessing domain walls in the more general case $\gamma^2<1$ when the Dzya\-lo\-shin\-skii-Mo\-riya interaction is taken into account. 

Note that in our stability statement, we will require more regularity on $h$ than in Corollary \ref{cor:precessing_dm}, namely that $h \in L^\infty([0, +\infty), \m R)$ (and small in that space): obviously in that case, the corresponding gauge $g_*$ is Lipschitz continuous (and of class $\q C^1$ if $h$ is assumed to be continuous). Provided that $h\in L^\infty([0, \infty))$ and $\alpha>0$, by Theorem \ref{th:lwp} below, the precessing domain walls \eqref{preces} are the unique solutions to \eqref{ll} with initial data $w_*^\pm$.

\subsection{Asymptotic stability under the Landau-Lifshitz-Gilbert flow}

We denote $H^s$ (and $L^p$) for the Sobolev space $H^s(\m R, \m R^3)$ with $s \ge 0$ (and the Lebesgue space $L^p(\m R, \m R^3)$ with $p\in [1, \infty]$, respectively). We also denote $\dot H^s$ for the homogeneous Sobolev space  whose seminorm is given via Fourier transform:
\begin{equation} 
\label{def:H^s}
\| m \|_{\dot H^s}^2 := \frac{1}{2\pi} \int_\R |\hat m (\xi)|^2 |\xi|^{2s} d\xi, \quad \text{where} \quad \hat m(\xi) = \int_{\m R} e^{-ix\xi} m(x)\, dx.
\end{equation}
(In particular, $\| m \|_{\dot H^2} = \| \partial_{xx} m \|_{L^2}$). In the following, we work in the spaces $\q H^s$ (for $s \ge 1$) modelled on the inhomogeneous Sobolev spaces $H^s$ (whose norm is given by $\|\cdot\|_{H^s}=\|\cdot\|_{L^2}+\|\cdot \|_{\dot H^s}$)
but adapted to the geometry of the target manifold $\m S^2$. More precisely, we define for $s \ge 1$:
\begin{align} 
\label{def:qH^s}  \q H^s &:= \{ m=(m_1, m_2, m_3) \in \q C(\m R, \m S^2) :  \| m \|_{\q H^s} < +\infty \}\\ 
\nonumber &\textrm{with}\quad \| m \|_{\q H^s}:= \| m_2 \|_{L^2} + \| m_3 \|_{L^2} + \| m \|_{\dot H^s}.
\end{align}
(Note that $m\in \q H^s$ implies that $m_1$ is not in $L^2$.)
In particular, $\q H^1$ corresponds to the set of finite energy configurations $E(m) < +\infty$ (see Lemmas \ref{lem:Hs_H1} and \ref{lem:infinit} in Appendix~\ref{append}) in which case, the energy gradient $\delta E(m)\in H^{-1}$. Also if $m, \tilde m \in \q H^1$ with $m(\pm \infty)=\tilde m(\pm \infty)$, then $m - \tilde m \in H^1$. Moreover, if $w_*$ is one of the domain walls $w_*^\pm$ in \eqref{formul1}, then every configuration $m\in \q H^1$ with $\|m-w_*\|_{\q H^1}$ small enough stays close to $w_*$ in $H^1$, i.e., 
$\|m-w_*\|_{H^1}\lesssim \|m-w_*\|_{\q H^1}$
(we refer to Proposition \ref{prop:conv_H_H1} and Lemma  \ref{lem:H_H1_2} for proofs and more details).

 \begin{nb} \label{rem23} Note that all the derivates of $w_*^\pm$ of order $k\ge 1$ are exponentially localised, so that
$w_*^\pm\in {\q H}^k$ for all $k\ge 1$. In particular,  for any $g\in G$, $g. w_*^\pm-w_*^\pm\in H^1$ and more generally if  $w\in w_*^\pm + H^1$ then $g. w\in w_*^\pm+ H^1$ for any $g\in G$, an observation we will use on several occasions (see \eqref{here-a} below).
\end{nb}

Our main result is the asymptotic stability of precessing domain walls in $H^1$ for small applied magnetic field $h$ and under small perturbations of the initial data, up to an adequate gauge choice. It means that for small $h \in L^\infty([0,+\infty))$, if the flow \eqref{ll} starts with an initial data $m_0$ close in $\q H^1$ to a (static) domain wall $w_*^\pm$, then the (unique) solution $m(t,x)$ of \eqref{ll} stays close in the $H^1$ topology at all times $t>0$ to the precessing domain wall  $g_*(t). w_*^\pm$ (given in Corollary \ref{cor:precessing_dm}), up to a fixed gauge $g_\infty \in G$. This gauge freedom  cannot be avoided due to the invariance of the equation, and it is the only degree of freedom in the problem.

\begin{thm}[Asymptotic stability of precessing domain walls in $H^1$] \label{th:stab}
Let $\alpha >0$ and $\ga \in (-1,1)$. There exist $\delta_0 >0$, $\sigma >0$ and $C>0$ such that the following holds: if $w_*\in \{w_*^\pm\}$, $h \in L^\infty([0,+\infty), \m R)$  and the initial data $m_0 \in \q H^1$ satisfy
\begin{align} \label{def:e0}
\| m_0 - w_* \|_{\q H^1} < \delta_0, \quad \| h \|_{L^\infty([0,+\infty))} < \delta_0, 
\end{align}
then there exist a unique solution $m \in \q C([0,+\infty), \q H^1)$ of \eqref{ll} with the initial data $m_0$, defined globally for forward times $t>0$,  and a Lipschitz  gauge
$g= (y, \phi): [0,+\infty) \to G$ such that 
\begin{align} \label{eq:decay}
\forall t \ge 0, \quad |\dot g(t) - \dot g_*(t)| + \| m(t) - g(t). w_* \|_{H^1} \le C  e^{-\sigma t} \| m_0 - w_* \|_{\q H^1},
\end{align}
where $g_*$ is defined in \eqref{eq_y_phi}. 
In particular, there exists a gauge $g_\infty \in G$ such that $ |g_\infty| \le C \| m_0 - w_* \|_{\q H^1}$ and 
\begin{align} \label{est:decay2}
\forall t \ge 0, \quad \| m(t) - (g_\infty +g_*(t)). w_* \|_{H^1} \le C  e^{-\sigma t} \| m_0 - w_* \|_{\q H^1}. 
\end{align}
\end{thm}

\begin{nb}
The asymptotic stability of the precessing domain walls is expected to fail for large applied fields $|h|$, as mentioned in the formal paper \cite{GGRS11}. The heuristic explanation is the following: 
 for a constant applied field $h(t)=h_0 > 1-\gamma^2$ (resp. $h_0 < - (1-\gamma^2)$), the static constant solution $m = -e_1$ (resp. $m = +e_1$) is linearly unstable, i.e., the second variation $D^2 E(m)$ is negative in some direction orthogonal to $e_1$. As the precessing domain walls are exponentially localised in space (i.e., in $x$ variable) at fixed time, they are nearly constant away from the domain wall; therefore, it is expected that the precessing domain walls are similarly unstable for $|h_0|> 1-\gamma^2$ (see also \cite{Jiz11} for a rigourous proof). Numerical results are also provided in \cite{GGRS11} to support this expectation.
\end{nb}

\subsection{Comments}

Earlier stability results (for $H^2$ perturbation) were obtained by Carbou and Labb\'e \cite{CL06a} for static domain walls (called Bloch walls), in the absence of applied field. Jizzini \cite{Jiz11} establishes  $H^1$ stability under a constant applied field (up to the optimal size), without proving any exponential decay, see also \cite{Car10,Car14}. These cited works are done in the absence of DMI, i.e., $\gamma=0$. Their approach is rather different from ours as it is focused on obtaining a Lyapunov type argument, by studying the linearised operator around the travelling wave (associated to the Bloch wall under the (LLG) flow) and putting in evidence a spectral gap. We also refer to Takasao \cite{Tak11} where a similar method was used to prove asymptotic $H^1$ stability of a travelling wave called the Walker wall, which appears in a different context (the transition layers connect the transversal directions $-e_2$ and $e_2$ to the nanowire and the anisotropy penalises the $m_1$ component) under a constant applied field and in absence of DMI.

\bigskip

Our method is significantly different from \cite{CL06a, Tak11}: for the proof of Theorem \ref{th:stab}, we essentially rely on an adequate expansion of the energy dissipation identity (see \eqref{eq:en_dissip2} and Proposition \ref{prop:en_expansion}), and on the positivity of the damping coefficient $\alpha >0$, as do all the stability results quoted above. (In case of no damping, i.e., $\alpha=0$, the equation will have the flavour of a Schr\"odinger map equation, for which stability questions of domain walls are  widely open.) 
Doing so, we are able to consider applied magnetic fields $h(t)$ depending on time which are not necessarily continuous: this can turn out to be very relevant in control questions, where the intensity of the applied field is a natural control (for this matter, we refer to the work of  Carbou, Labb\'e and Trelat \cite{CLT08}). On this topic, let us emphasise that the regularity assumption $h \in L^\infty$ makes a lot of sense (for bang-bang controls in particular) and this is why we rather state and prove our result under this condition, and not under slightly more regular assumption such as $h$ continuous and bounded (which would yield $\q C^1$ regularity of $m$ and $g$).  

We are also able to treat more general energy functionals, and specifically the Dzya\-lo\-shin\-skii-Moriya interaction\footnote{One idea in treating the DMI at the stationary level is to consider the rotated magnetization $M=R_{\varphi_*} m$ with $\varphi_*$ given in \eqref{formul1} yielding an energy density for $M$ without DMI, see Step 2 in the proof of Theorem~\ref{thm_DM}. One could also use this rotated magnetisation at the dynamical level, however our method based on modulation is rather indifferent on DMI, and the computations do not get more technical due to DMI.} -- which is physically relevant -- and a more involved geometric context (we consider \emph{precessing} domain walls, with the additional action of rotations next to translations); even though one might be able to implement a Lyapunov argument in this context, the computations should turn out to be noticeably trickier than those we perform.

\bigskip

We mentioned above that \cite{GGRS11}  proved linear stability of precessing domain walls (without the Dzyaloshinskii-Moriya interaction), and we took some inspiration from their computation. But at the linear level, the geometric considerations can be mostly avoided. To be able to complete the stability proof at the nonlinear level, the evolution equation \eqref{ll} has to be written under a suitable gauge transform, corresponding to the choice of an orthogonal frame. In particular, we develop a modulation theory in the context of functions taking values into a manifold (see Lemma \ref{lem:mod} and Proposition \ref{lem:decomp}), where the regularity issues are to be taken care with caution. Another important difficulty in making the arguments rigorous (when compared to \cite{GGRS11}) is that the energy $E$ is coercive around a domain wall $w_*$ only at level $H^1$ of regularity but the dissipative effects are accessible only at level $H^2$; hence, the functional setting and bootstrap are to be carefully chosen. We also made a point of presenting the results in the energy space $\q H^1$, and we refer to Propositions \ref{prop:conv_H_H1} and Lemmas \ref{lem:H_H1_2} which state how $\q H^1$ is related to $w_*+ H^1$.

\bigskip

On the bright side, our arguments are amenable to space localisation: this an important improvement with respect to \cite{CL06a, Tak11}, and actually a great motivation for this work. Our method should indeed be effective to understand more complicated configurations, for example the dynamics of several domain walls in interaction. We conjecture that, at least for sufficiently small applied field $h$, generic magnetisations behave as a combination of far-off walls. As critical points of the energy $E$ with several walls do not exist (see Remark \ref{rem:several}), they either collide or they separate at infinite distance; in view of the possible relative speeds, it seems that at most two walls are possible. This will be the purpose of forthcoming works.

\bigskip

The paper is organised as follows. In Section 2, we present the derivation of the reduced micromagnetic model in  \eqref{def:energy} by proving a dimension reduction result via $\Gamma$-convergence in a certain regime for a ferromagnetic nanowire. In Section 3, we prove the results on static domain walls, namely Theorem \ref{thm_DM} as well as the structure of the precessing domain walls in Corollary \ref{cor:precessing_dm}. In Section 4, we prove modulation results at the stationary level in Lemma \ref{lem:mod} and at the dynamical level in Proposition \ref{lem:decomp}; these results together with the energy coercivity and energy dissipation estimates in Proposition \ref{prop:en_expansion} are the main ingredients in proving the asymptotic stability result in Theorem~\ref{th:stab}. In  Appendix, we give a few useful identities regarding the gauge $g \in G$, and various properties of the space $\q H^1$, related with the energy $E$ and  the Sobolev space $H^1$.

\section{Dimension reduction via \texorpdfstring{$\Gamma$}{Gamma}-convergence}
\label{sec:model}

Let $\Om\subset \m R^3$ be a cylinder modelling
a nanowire of axis $(-L, L)$ in direction $e_1$ and cross-section given by the disk $B_R$ centred at the origin of radius $R$  in the plane $(e_2, e_3)$. For a magnetisation $m:\Om\to \m S^2$, the following 3D micromagnetic energy 
functional is defined: 
\[  \m E(m)= \int_\Om d^2 |\nabla m|^2 +Q\F(m)+ DMI(m) \, dX+\int_{\m R^3} |\nabla U|^2\, dX \]
where $X=(X_1, X_2,X_3)\in \R^3$ is the space variable, $d>0$ is the exchange length, $Q>0$ is the quality factor associated to the uniaxial anisotropy of the form
\[ \F(m)=1-m_1^2=m_2^2+m_3^2 \]
favouring the axes $\pm e_1$. Here, $U:\R^3\to \R$ is the stray-field potential solving the static Maxwell equation
\[ \Delta U=\nabla\cdot (m \m {1}_\Om)\quad \textrm{in} \quad \R^3, \] 
where one thinks of $m = m \m{1}_\Om$ as being extended by $0$ outside $\Omega$. 
We take into account the Dzyaloshinskii-Moriya interaction (DMI) that has the energy density
\[ DMI(m)=D:(\nabla m \wedge m)=\sum_{k=1}^3 D_k\cdot (\partial_k m\wedge m) \]
where the tensor $D=(D_1,D_2, D_3)\in \R^{3\times 3}$ depends on the ferromagnetic material, $\wedge$ is the exterior product in $\R^3$ and $:$ stands for the Frobenius inner product of tensors.

The variational problem to describe minimisers for $\m E$ is nonlocal (due to the stray-field $\nabla U$) and nonconvex (due to the constraint $|m|=1$). We will focus on a regime of very thin and hard nanowire where the nonlocal contribution of the stray-field energy becomes negligible.

\subsection{Asymptotic regime} 
We denote $a\ll b$ (or $a=o(b)$) if $\frac a b\to 0$, resp. $a\lesssim b$ (or $a=O(b)$) if there is a constant $C>0$ such that $a\le C b$, and $a\sim b$ if $a\lesssim b$ and $b\lesssim a$. We also denote $o_a(1)$ a quantity that tends to $0$ as $a\to 0$.
Denoting the length scale 
\[ \rho:=\frac d{\sqrt{Q}}, \] we assume the following regime of parameters between the length $L$ of the nanowire, the thickness $R$ of the cross-section, the exchange length $d$, the  Dzyaloshinskii-Moriya tensor $D$ and $\rho$:
\be
\label{regime}
\begin{cases}
&\ds  R\ll \rho \ll L,  \, \frac{\rho L}{d^2}\ll 1, \, \frac{\rho}{d^2}D_{1,1}\sim 1,\\
& \frac{\rho}{d^2}(|D_{1,2}|+|D_{1,3}|+|D_{2,1}|+|D_{3,1}|)\ll 1,\\ 
&\frac{\sqrt{\rho L}}{d^2}(|D_{2,2}|+|D_{2,3}|+|D_{3,2}|+|D_{3,3}|)\ll 1,
\end{cases}
\ee
where $D_k=(D_{k,1}, D_{k,2}, D_{k,3})$ for $k=1,2,3$.
The asymptotic analysis is carried out after the parameter
\[ \eta:=\frac R \rho\to 0, \]
 assuming that all the parameters of the system are functions depending on $\eta$. In the following, we assume more than $\frac{\rho}{d^2}D_{1,1}\sim 1$; namely, 
\be 
\label{reg2}
\frac{\rho}{d^2}D_{1,1}\to -2\gamma \quad \text{as} \quad \eta\to 0, \ee
for a constant $\gamma\in (-1,1)$. In particular, we have
\[ R\ll \frac{d}{\sqrt Q}\ll L\ll d\sqrt Q, \quad 1\ll Q, \]
 which means that the material is hard (i.e., the quality factor is large), the nanowire is very thin and of short length (since $d$ is of order of nanometers). As we prove in Theorem \ref{thm:gamma} below, the stray-field contribution vanishes asymptotically in the regime \eqref{regime} \& \eqref{reg2}.

\subsection{Rescaling} 
We rescale as follows: 
\begin{gather*}
\tilde X=(X_1/\rho, X_2/R, X_3/R), \quad \tilde \Om=(-L/\rho, L/\rho)\times B_1, \quad \tilde m(\tilde X)=m(X), 
\quad \tilde U(\tilde X)=\frac1R U(X), \\
\tilde E_\eta(\tilde m):=\frac{\rho}{d^2R^2} \m E(m).
\end{gather*}
Therefore, \emph{denoting $'$ for the quantities depending on the last two variables} (i.e., $\tilde X'=
(\tilde X_2, \tilde X_3)$, $\tilde \nabla'=(\tilde \partial_{2}, \tilde \partial_{3})$, $\tilde m'=(\tilde m_2, \tilde m_3)$, $D'=(D_2, D_3)\in \R^{3\times 2}$ etc \dots), we have
\begin{align}
\label{en*}
\tilde E_\eta(\tilde m)&=\int_{\tilde \Om} |\tilde \partial_1 \tilde m|^2+|\frac 1{\eta}\tilde \nabla' \tilde m|^2
+ \F(\tilde m)+\frac\rho{d^2}\bigg(D_1\cdot \tilde \partial_1 \tilde m\wedge \tilde m+
D': (\frac 1 \eta \tilde \nabla' \tilde m)\wedge \tilde m\bigg)\, d\tilde X\\
\nonumber
&\quad \quad +\frac{R^2}{d^2}\int_{\R^3} |\tilde \partial_1 \tilde U|^2+|\frac1 \eta\tilde \nabla' \tilde U|^2\, d\tilde X
\end{align}
where $\eta^2 \tilde \partial_{11} \tilde U+\tilde \Delta' \tilde U=\eta \tilde \partial_1 (\tilde m_1 {\m 1_{\tilde \Om}})+\tilde \nabla'\cdot (\tilde m' {\m 1_{\tilde \Om}})$ in $\R^3$.
For simplicity of notation, \emph{we skip the $\tilde{\cdot}$ in the following.}

\subsection{\texorpdfstring{$\Gamma$}{Gamma}-convergence}
For every $\eta>0$ small, we focus on configurations defined on the infinite cylinder 
\[ \Sigma:=\R\times B_1 \]
that are constant and equal to $\pm e_1$ outside the rescaled nanowire $\Sigma_\eta:=(-L/\rho, L/\rho)\times B_1,$ i.e., 
\[ \M_\eta=\left\{m:\Sigma\to \m S^2\, :\, \|m\|_{\q H^1(\Sigma)}<\infty, \, m'=(m_2, m_3)=0 \, \textrm{ if } \, |x_1|\ge L/\rho\right\} \]
where $\M_\eta$ is endowed with the seminorm 
\[ \|m\|_{\q H^1(\Sigma)}:=\|\nabla m\|_{L^2(\Sigma)}+\|m'\|_{L^2(\Sigma)}. \]
As $m\in \M_\eta$ is constant outside of  $\Sigma_\eta$, we can rewrite the energy \eqref{en*} as
\begin{align}
\label{defEe}
E_\eta(m) & =\int_{\Sigma} |\partial_1 m|^2+|\frac 1{\eta}\nabla' m|^2
+ \F(m)+\frac\rho{d^2}\bigg(D_1\cdot \partial_1  m\wedge m+
D': (\frac 1 \eta  \nabla'  m)\wedge  m\bigg)\, dX \\
\nonumber
& \qquad +\frac{R^2}{d^2}\int_{\R^3} |\partial_1  U|^2+|\frac1 \eta \nabla'  U|^2\, dX,
\end{align}
where the stray-field potential $U:\R^3\to \R$
is the unique solution in $\dot H^1(\R^3)$ of
\[ \partial_{11} U+ \frac1{\eta^2}\Delta'  U=\frac1\eta  (\partial_1, \frac1 \eta \nabla')\cdot (m {\m 1_{\Sigma_\eta}}) \quad \text{in} \quad \R^3, \] 
i.e. (in consequence of the Lax-Milgram theorem applied in the space $\dot H^1(\R^3)$ under the assumption 
$m\in L^2(\Sigma_\eta)$),
\[ \int_{\R^3} (\partial_1, \frac1 \eta \nabla')U \cdot (\partial_1, \frac1 \eta \nabla')\zeta\, dX=
\frac1 \eta\int_{\Sigma_\eta} m\cdot (\partial_1, \frac1 \eta \nabla')\zeta\, dX, \quad \text{for every } \zeta\in \q C^\infty_c(\R^3). \]
As $m\in L^2(\Sigma_\eta)$, due to the density of $\q C^\infty_c(\R^3)$ in $\dot H^1(\R^3)$, the above equality is satisfied for all $\zeta \in \dot H^1(\R^3)$. In particular, for $\zeta=U$ it yields
\be
\label{stray_l2}
\int_{\R^3} \big|(\partial_1, \frac1 \eta \nabla')U\big|^2\, dX\le \frac1{\eta^2} 
\int_{\Sigma_\eta} |m|^2\, dX.
\ee
Our result shows that in the regime \eqref{regime}\& \eqref{reg2}, $E_\eta$ $\Gamma$-converges to the limit energy
\be
\label{defE0}
 E_0(m)=\int_{\Sigma} |\partial_1 m|^2+
\F(m) -2\gamma e_1 \cdot \partial_1  m\wedge m\, dX \ee
that is defined for configurations depending only on $X_1$: 
\[ \M_0=\left\{m:\Sigma\to \m S^2\, :\,  m=m(X_1), \|m\|_{\q H^1(\Sigma)}<\infty \right\}, \]
in particular, $m'$ has vanishing limit as $X_1\to \pm \infty$ (as $m'\in H^1$). Note that for $m$ depending only on the variable $X_1$, this limit model is exactly the one presented in the introduction (up to a multiplicative constant for the energy $E$ in \eqref{def:energy}). The $\Gamma$-convergence \footnote{One can consider that $E_\eta$ is extended by $+\infty$ in the set $\q H^1(\Sigma)\setminus \M_\eta$, as well as $E_0$ is extended by $+\infty$ in the set  $\q H^1(\Sigma)\setminus \M_0$. } is carried out in the $\q H^1(\Sigma)$ weak topology, i.e., we say that $m_\eta\rightharpoonup m_0$ weakly in $\q H^1(\Sigma)$-topology if and only if $\nabla m_\eta\rightharpoonup \nabla m_0$ and $m'_\eta\rightharpoonup m'_0$ weakly in $L^2(\Sigma)$.

\begin{thm}[$\Gamma$-convergence]
\label{thm:gamma}
Let $E_\eta$ be given in \eqref{defEe} and $E_0$ in \eqref{defE0}. 
In the regime \eqref{regime}\& \eqref{reg2}, $E_\eta\stackrel{\Gamma}{\longrightarrow} E_0$ as $\eta\to 0$ in the weak $\q H^1(\Sigma)$ topology, i.e.,

\begin{enumerate}
\item Compactness: If $m_\eta\in \M_\eta$ such that $\limsup_{\eta\to 0} E_\eta(m_\eta)<\infty$, then for a subsequence, $m_\eta\rightharpoonup m_0$ weakly in $\q H^1(\Sigma)$ for a limit $m_0\in \M_0$.

\item Lower bound: If $m_\eta\in \M_\eta$ and $m_0\in \M_0$ with $m_\eta\rightharpoonup m_0$ weakly in $\q H^1(\Sigma)$ as $\eta\to 0$, then 
\[ \liminf_{\eta\to 0} E_\eta(m_\eta)\ge E_0(m_0). \]

\item Upper bound: If $m_0\in \M_0$, then there exists a family $m_\eta\in \M_\eta$ such that $m_\eta\to m_0$ strongly in 
$\q H^1(\Sigma)$ as $\eta\to 0$ and 
\[ \lim_{\eta\to 0} E_\eta(m_\eta)=E_0(m_0). \]
\end{enumerate}
\end{thm}

The key step in the proof of Theorem \ref{thm:gamma} consists in showing the coercivity of $E_\eta$ over $\M_\eta$:

\begin{lem}[Coercivity]
\label{lem:coercive}
In the regime \eqref{regime} \& \eqref{reg2}, 
there exist  two constants $\eta_\gamma>0$ and $C_\gamma>0$ such that for every $0<\eta\le \eta_\gamma$, 
\[  \forall m \in \M_\eta, \quad E_\eta(m)\ge C_\gamma\bigg(-1+\int_{\Sigma} |\partial_1 m|^2+|\frac 1{\eta}\nabla' m|^2+|m'|^2\, dX\bigg). \]
\end{lem}

\begin{proof}[Proof of Lemma \ref{lem:coercive}]
The idea is to absorb the DMI term  into the exchange and anisotropy term  in $E_\eta$ (the stray-field term is non-negative and doesn't play any role here). As $\frac{\rho}{d^2}D_{1,1}\to -2\gamma$ as $\eta\to 0$ and $|\gamma|<1$, there exists $\tilde C_\gamma>0$ such that for small $\eta>0$:
\[ |\partial_1 m'|^2+\frac{\rho}{d^2}D_{1,1} e_1\cdot \partial_1 m\wedge m+ |m'|^2\ge
\tilde C_\gamma (|\partial_1 m'|^2+|m'|^2). \]
Similarly, since $\frac{\rho}{d^2}(|D_{2,1}|+|D_{3,1}|)\to 0$ as $\eta\to 0$, we have
\begin{align*}
\frac{\rho}{d^2}\big|D_{2,1} e_1\cdot (\frac1\eta \partial_2 m)\wedge m\big| & \le o_\eta(1)(|\frac1\eta \partial_2 m'|^2+|m'|^2) \\
\frac{\rho}{d^2}\big|D_{3,1} e_1\cdot (\frac1\eta \partial_3 m)\wedge m\big| & \le o_\eta(1)(|\frac1\eta \partial_3 m'|^2+|m'|^2).
\end{align*}
In order to treat the term $e_2\cdot \partial_1 m\wedge m=\partial_1 m_3 m_1-\partial_1 m_1 m_3$, integration by parts and the fact that $m\in \M_\eta$ (in particular, $m_3=0$ away from a compact set) yield
\[ \int_\Sigma e_2\cdot \partial_1 m\wedge m\, dX=-2\int_\Sigma \partial_1 m_1 m_3\, dX; \]
therefore, one uses that $\frac{\rho}{d^2}|D_{1,2}|\ll 1$ yielding
\[ \bigg| \frac{\rho}{d^2} D_{1,2} \int_\Sigma e_2\cdot \partial_1 m\wedge m \, dX\bigg|\le o_\eta(1)
\int_\Sigma |\partial_1 m|^2+m_3^2 \, dX. \]
Similarly, $\frac{\rho}{d^2}|D_{1,3}|\ll 1$ yields
\[ \bigg| \frac{\rho}{d^2} D_{1,3} \int_\Sigma e_3\cdot \partial_1 m\wedge m \, dX\bigg|\le o_\eta(1)
\int_\Sigma |\partial_1 m|^2+m_2^2 \, dX. \]
The last terms we treat as follows: since $\frac{\sqrt{\rho L}}{d^2}(|D_{2,2}|+|D_{2,3}|+|D_{3,2}|+|D_{3,3}|)\to 0$, $m_3=0$ outside $\Sigma_\eta$ and $|\Sigma_\eta|=2\pi L/\rho$, we have
\[ \bigg| \frac{\rho}{d^2} D_{2,2} \int_\Sigma e_2\cdot (\frac1 \eta\partial_2 m)\wedge m \, dX\bigg|\le \frac{\rho}{d^2} |D_{2,2}| \int_{\Sigma_\eta} \big| \frac1 \eta\partial_2 m \big|\, dX=o_\eta(1) \|\frac1 \eta \partial_2 m\|_{L^2(\Sigma)} \]
and the same type of estimates hold for the terms in $D_{2,3}$, $D_{3, 2}$ and $D_{3,3}$.
Summing up, we deduce the existence of a constant $\tilde c_\gamma>0$ such that
\[ E_\eta(m)\ge \tilde c_\gamma \int_{\Sigma} |\partial_1 m|^2+|\frac 1{\eta}\nabla' m|^2+|m'|^2\, dX
-o_\eta(1) \|\frac1 \eta \nabla' m\|_{L^2(\Sigma)} \]
so the conclusion follows immediately. 
\end{proof}

\bigskip

\begin{proof}[Proof of Theorem \ref{thm:gamma}]
We divide the proof in several steps:

\medskip

\nd \emph {Step 1. Compactness.} Assume that $m_\eta\in \M_\eta$ with $\limsup_{\eta\to 0} E_\eta(m_\eta)<\infty$. By Lemma \ref{lem:coercive}, 
\be
\label{coerci}
\int_{\Sigma} |\partial_1 m_\eta|^2+|\frac 1{\eta}\nabla' m_\eta|^2+|m'_\eta|^2\, dX\le C
\ee
for some $C>0$ independent of $\eta$. Therefore, for a subsequence still denoted $\eta\to 0$, there is a limit $m_0\in \dot{H}^1(\Sigma)$ such that $m_\eta\rightharpoonup m_0$ weakly in $\dot{H}^1(\Sigma)$, strongly in $L^2_{\loc}(\Sigma)$ and a.e. in $\Sigma$, and $m'_\eta\rightharpoonup m'_0$ weakly in $L^2(\Sigma)$. Moreover, $|m_0|=1$ in $\Sigma$ (due to the a.e. convergence) and $\nabla'm_0=0$ in $\Sigma$ (because $\nabla' m_\eta\to 0$ in $L^2(\Sigma)$ as $\eta\to 0$), thus, $m_0$ depends only on $X_1$.  We conclude that
$m_0\in \M_0$. 

\medskip

\nd \emph{Step 2. Lower bound.} Let $m_\eta\in \M_\eta$ and $m_0\in \M_0$ with $m_\eta\rightharpoonup m_0$ weakly in $\q H^1(\Sigma)$ as $\eta\to 0$. For a subsequence, we may assume that $E_\eta(m_\eta)\to \liminf_{\eta\to 0} E_\eta(m_\eta)$ as $\eta\to 0$ and 
$\liminf_{\eta\to 0} E_\eta(m_\eta)<\infty$ (otherwise, the desired inequality is trivially satisfied). 
By Lemma \ref{lem:coercive}, we know that \eqref{coerci} holds true. 
We note that
\[ |\partial_1 m'_\eta|^2 -2\gamma e_1\cdot \partial_1 m_\eta\wedge m_\eta+\gamma^2 |m'_\eta|^2=(\partial_1 m_{\eta,2}-\gamma m_{\eta,3})^2+(\partial_1 m_{\eta,3}+\gamma m_{\eta,2})^2. \]
Using the weak lower semicontinuity of the $L^2$-norm, since $\nabla m_\eta\rightharpoonup \nabla m_0$ and $m'_\eta\rightharpoonup m'_0$ weakly in $L^2(\Sigma)$, we obtain
\begin{align*}
\liminf_{\eta\to 0} &\int_\Sigma |\partial_1 m'_\eta|^2-2\gamma e_1\cdot \partial_1 m_\eta\wedge m_\eta+\gamma^2 |m'_\eta|^2\, dX\\
&\ge \int_\Sigma |\partial_1 m'_0|^2-2\gamma e_1\cdot \partial_1 m_0\wedge m_0+\gamma^2 |m'_0|^2\, dX
\end{align*}
as well as
\[ \liminf_{\eta\to 0} \int_\Sigma |\partial_1 m_{\eta,1}|^2+(1-\gamma^2)|m'_\eta|^2\, dX\ge \int_\Sigma |\partial_1 m_{0,1}|^2+(1-\gamma^2)|m'_0|^2\, dX \] 
because $\gamma^2<1$.
The remaining terms in $E_\eta$ are either non-negative, or involve the DMI. For DMI, we use \eqref{coerci} and the proof of Lemma \ref{lem:coercive} to deduce that the terms involving $D_{1,2}$, $D_{1,3}$, $D_2$ and $D_3$ converge to $0$ as $\eta\to 0$
in the regime \eqref{regime} \& \eqref{reg2}. Finally, as $\frac{\rho D_{1,1}}{d^2}\to -2\gamma$, we get again by \eqref{coerci}:
\[ \bigg|\big(\frac{\rho}{d^2}D_{1,1}+2\gamma\big)  \int_\Sigma e_1\cdot \partial_1 m_\eta\wedge m_\eta \, dX\bigg|=o_\eta(1)  \int_\Sigma |\partial_1 m'_\eta|^2+|m_\eta'|^2\, dX\to 0 \quad \textrm{as} \quad \eta\to 0. \]
Summing up, we obtain the desired lower bound.

\medskip

\nd \emph{Step 3. Upper bound.} Let $x$ be the first coordinate $x=X_1$ and $m=m(x)\in \M_0$, in particular, $\|m\|_{{\mathcal H}^1(\Sigma)}<\infty$. For small $\e>0$, we show that for every $0<\eta\le \eta_\e$ there exists $m_\eta\in \M_\eta$ such that
$\|m-m_\eta\|_{{\mathcal H}^1(\Sigma)}=o_\e(1)$ and $\big|E_\eta(m_\eta)-E_0(m)\big|=o_\e(1)$. For this, fix some $\e\in (0, \frac 1{100})$. Note that in the regime $L/\rho\to \infty$ as $\eta\to 0$, due to Lemma \ref{lem:infinit}, we may assume that for two limits $a^\pm\in \{\pm 1\}$, $m_1\to a^\pm$ and $m'\to (0,0)$ as $x\to \pm \infty$, so that
\be
\label{assum1}
\int_{\R\setminus [-\frac L \rho+1, \frac L \rho-1]} |\partial_1 m|^2+|m'|^2\, dx\le \e, \, \textrm{ and }\, |m'(x)|\le \e \textrm{ for } |x|\ge \frac L \rho-1 
\ee
and 
\be
\label{con_m1} |m_1-a^\pm|\le \e^2 \textrm{ for } \pm x\ge \frac L \rho-1.
\ee
We choose $m_\eta=m_\eta(x):x\in\R\to \m S^2$ such that $m_\eta:=m$ on $(-\frac L \rho+1, \frac L \rho-1)$, $m_\eta(x)=(a^\pm,0,0)$ if $\pm x\ge \frac L \rho$ and $m_\eta$ is the (unique) continuous function in $\R$ extended by linear interpolation in the spherical coordinates $\theta$ (keeping the other spherical coordinate $\f$ fixed) inside $[-\frac L \rho, \frac L \rho]\setminus [-\frac L \rho+1, \frac L \rho-1]$. In other words, $m_\eta([\frac L \rho-1, \frac L \rho])$ (resp. $m_\eta([-\frac L \rho, -\frac L \rho+1])$) is the geodesic on $\m S^2$ between 
$m(\frac L \rho-1)$ and $a^+$ (resp. $m(-\frac L \rho+1)$ and $a^-$). Due to \eqref{con_m1}, the slope of the spherical coordinate $\theta_\eta$ of $m_\eta$ is of order $o_\e(1)$ inside $[-\frac L \rho, \frac L \rho]\setminus [-\frac L \rho+1, \frac L \rho-1]$, so
\be
\label{interm_en}
\int_{[-\frac L \rho, \frac L \rho]\setminus [-\frac L \rho+1, \frac L \rho-1]} |\partial_1 m_\eta|^2+|m'_\eta|^2\, dx=o_\e(1).
\ee
Therefore, using \eqref{assum1} and \eqref{interm_en}, we deduce that $\|m-m_\eta\|_{{\mathcal H}^1(\Sigma)}=o_\e(1)$ and
\[ \int_{\Sigma} |\partial_1 m_\eta|^2+|\frac 1{\eta}\nabla' m_\eta|^2 + |m'_\eta|^2\, dX=\int_{\Sigma} |\partial_1 m|^2+|m'|^2\, dX+o_\e(1). \]
As $m_\eta$ depends only on $x=X_1$, the only nonzero DMI term in $E_\eta(m_\eta)$ involves $D_1$. The above estimates and the regime \eqref{regime}\& \eqref{reg2} imposed on $D_1=(D_{1,1}, D_{1,2}, D_{1,3})$ yield
\[ \frac\rho{d^2}\int_\Sigma D_1\cdot \partial_1  m_\eta\wedge m_\eta\, dX=-2\gamma \int_\Sigma e_1\cdot \partial_1 m
\wedge m\, dX+o_\e(1)+o_\eta(1). \]
We prove that the stray-field energy associated to $m_\eta$ in $E_\eta(m_\eta)$ vanishes as $\eta\to 0$. As $\Sigma_\eta=(-\frac L\rho, \frac L\rho)\times B_1$, by \eqref{stray_l2}, we have for the stray-field potential $U_\eta$ associated to $m_\eta$:
\[ \int_{\R^3} \big|(\partial_1, \frac1 \eta \nabla')U_\eta\big|^2\, dX\le \frac1{\eta^2} 
\int_{\Sigma_\eta} |m_\eta|^2\, dX\le \frac{2\pi L}{\rho \eta^2}. \]
As $\frac{R^2}{d^2} \frac{L}{\rho \eta^2}=\frac{\rho L}{d^2}\to 0$ as $\eta\to 0$, summing up the above estimates, we conclude that $\big|E_\eta(m_\eta)-E_0(m)\big|=o_\e(1)+o_\eta(1)$ which is in fact $o_\e(1)$ for $\eta$ small enough.
\end{proof}

$\Gamma$-convergence results are also known in some other asymptotic regimes of nanowire where the stray-field contribution is still present in the $\Gamma$-limit functional, though as a local term (see e.g. \cite{SlaSon, KK_these, Kuh09, CL06b}).

\subsection{Convergence of minimisers}

As a classical consequence of the $\Gamma$-convergence me\-thod, we deduce in the following the convergence of minimisers of the $3D$ micromagnetic energy $E_\eta$ in \eqref{defEe} to minimisers of the reduced energy $E_0$ in \eqref{defE0} when a transition from $-e_1$ to $e_1$ is imposed for the configurations in $\M_\eta$ and $\M_0$, respectively. The non\-stan\-dard part of the result is to insure the compactness of this boundary condition which is based on the compactness result proved in \cite[Lemma 1]{DIO} for sign-changing configurations.

\begin{cor}
In the regime \eqref{regime}\& \eqref{reg2}, let $m_\eta$ be a minimiser of $E_\eta$ over the set $\{m\in \M_\eta\, :\, m(x_1, \cdot)=\pm e_1 \textrm{ for } \pm x_1\ge L/\rho\}$. Then for a subsequence $\eta\to 0$, there exists $x_{1,\eta}\in (- L/\rho, L/\rho)$ such that $m_\eta(\cdot+x_{1,\eta} e_1)
\rightharpoonup m_0$ weakly in $\q H^1(\Sigma)$ where $m_0\in \M_0$ is a minimiser of $E_0$ over the set $\{m\in \M_0\, :\, m(\pm \infty)=\pm e_1\}$ and $e_1\cdot m_0(0)=0$.
\end{cor}

\begin{proof}
We divide the proof in two steps:

\medskip

\nd \emph {Step 1. Existence of minimisers of $E_\eta$ connecting $-e_1$ to $e_1$.}  For fixed $\eta>0$, the existence of a minimiser $m_\eta$ of $E_\eta$ over the set $\{m\in \M_\eta\, :\, m(x_1, \cdot)=\pm e_1 \textrm{ for } \pm x_1\ge L/\rho\}$ is a consequence of the direct method in the calculus of variations. Indeed, by Lemma \ref{lem:coercive} applied at fixed $\eta>0$, the energy $E_\eta$ is coercive with respect to the $H^1(\Sigma_\eta)$-norm, therefore any minimising sequence for the energy $E_\eta$ satisfying the boundary condition $\pm e_1$ at $x_1=\pm L/\rho$ converges weakly in 
$H^1(\Sigma_\eta)$ (for a subsequence) to a minimiser $m_\eta$ of $E_\eta$ with the same boundary condition (by the trace embedding $H^1(\Sigma_\eta) \hookrightarrow H^{1/2}(\partial \Sigma_\eta)$). 

\medskip

\nd \emph {Step 2. Convergence of minimisers $m_\eta$ as $\eta\to 0$.} For a sequence $\eta\to 0$ and minimisers $m_\eta$ as above, we denote the average of $m_\eta$ over the cross-section $\{x_1\}\times B_1$ at each $x_1$:
\[ \bar m_\eta(x_1)=\frac1{\pi}\int_{B_1} m_\eta(x_1, x_2, x_3)\, dx_2dx_3, \quad x_1\in \R. \]
By the upper bound in Theorem \ref{thm:gamma}, any fixed $1D$ transition on $\mathbb{S}^2$ between $\pm e_1$ on the interval $(-1,1)$ (extended continuously by a constant on $(-L/\rho, L/\rho)\setminus (-1,1)$) leads to a uniform bound for $\{E_\eta(m_\eta)\}_{\eta\to 0}$ in the regime \eqref{regime}\& \eqref{reg2}. By the coercivity in Lemma \ref{lem:coercive}, we deduce that $\{m_\eta\}$ as well as $\{\bar m_\eta\}$ are bounded in $\dot{H}^1(\Sigma)$. Since $e_1\cdot \bar m_\eta(\pm \infty)=\pm 1$, by the compactness result \cite[Lemma 1]{DIO}, we know that there exists a zero $x_{1,\eta}\in (- L/\rho, L/\rho)$ of $e_1\cdot \bar m_\eta$ such that for a subsequence $\eta\to 0$, $e_1\cdot \bar m_\eta(\cdot+x_{1,\eta}) \rightharpoonup u$ weakly in $\dot{H}^1(\R)$ and locally uniformly in $\R$ with $u(0)=0$ and
\[ \limsup_{x_1\to -\infty} u(x_1)\le 0\quad \textrm{and}\quad \liminf_{x_1\to +\infty} u(x_1)\ge 0. \]
Passing eventually to a further subsequence, we also have by Theorem \ref{thm:gamma} that $m_\eta(\cdot+x_{1,\eta} e_1)\rightharpoonup m_0$ weakly in $\q H^1(\Sigma)$ for a limit $m_0\in \M_0$. By the trace theorem applied at the cross-section $\{x_1\}\times B_1$ at each $x_1\in \R$, we deduce that $e_1\cdot \bar m_0=e_1\cdot m_0=u$ in $\R$ (as $m_0$ depends only on $x_1$). By Lemma \ref{lem:infinit}, we know that $e_1\cdot m_0(\pm \infty)\in \{\pm 1\}$; by the sign constraint on $u$ at $\pm \infty$, we obtain that $e_1\cdot m_0(\pm \infty)=\pm 1$, i.e., $m_0(\pm \infty)=\pm e_1$. We conclude by Theorem \ref{thm:gamma} that $m_0$ is a minimiser of $E_0$ over the set $\{m\in \M_0\, :\, m(\pm \infty)=\pm e_1\}$.
\end{proof}

\section{Domain walls. Proof of Theorem \ref{thm_DM} and Corollary \ref{cor:precessing_dm}}

The gradient of the energy $E$ at a map $m$ is given by  
\[ \delta E(m) = - \partial_{xx} m +2\ga (\partial_x m_3 e_2-\partial_x m_2 e_3)+m_2 e_2+m_3 e_3.\]
Therefore, imposing the constraint that $m$ takes values into $\m S^2$, such a critical point $m$ of $E$ is collinear to the gradient $\delta E(m)$, and so $m$ satisfies the Euler-Lagrange equation \eqref{eq:dm_critical}:
\[ m\wedge \delta E(m)=0. \] 

\smallskip

\nd {\it Spherical coordinates when $m$ avoids the poles}. Sometimes we use the following frame adapted to a $\m S^2$-valued magnetisation $m$. For that, assume $m$ can be written as in \eqref{spheric_coord} using spherical coordinates $(\varphi, \theta)$. We will denote
\be
\label{np}
n=\partial_\theta m=-\frac1{\sin \theta}m\wedge (e_1\wedge m) \quad \textrm{ and } \quad p=\frac1{\sin \theta} \partial_\varphi m=\frac1{\sin \theta} e_1\wedge m=m\wedge n,
\ee
so that $(m,n,p)$ is an orthonormal frame in $\R^3$. However, the writing \eqref{spheric_coord} in spherical coordinates can be used only in an open interval $I$ where $m$ does not touch the poles $\pm e_1$. Typically, if $m \in H^1_{\text{loc}}(I, \m S^2)$ then $m$ is continuous in $I$ and we assume 
\[ m(I)\subset \m S^2\setminus \{\pm e_1\} \]
 (in particular $m_1(I)\in (-1,1)$). Since $\arccos:(-1,1)\to (0, \pi)$ is a diffeomorphism, there exists a unique continuous map $\theta:I\to (0, \pi)$ such that 
\[ m_1=\cos \theta \quad \text{in } I, \quad \theta \in H^1_{\text{loc}}(I) \quad \text{and} \quad \partial_x \theta=-\frac{\partial_x m_1}{\sqrt{1-m_1^2}} \quad \text{a.e. in } I. \]
Furthermore, as $\sin \theta>0$ in $I$, the map $\ds \frac1{\sin \theta}(m_2, m_3):I\to \m S^1$ is continuous in the interval $I$; therefore, there exists a continuous lifting $\varphi:I\to \m R$ (unique up to an additive $2\pi \Z$ constant) such that \eqref{spheric_coord} holds true in $I$, 
\[ \varphi \in H^1_{\loc}(I) \quad \text{and} \quad \partial_x \varphi=\frac1{1-m_1^2}
\partial_x m \cdot (e_1\wedge m) \quad \text{a.e. in } I. \]
In particular, the density of the energy $E$ writes a.e. in $I$:
\begin{equation} \label{energy_spherical:bounds}
|\partial_x m|^2 + 2\ga \partial_x m \cdot (e_1\wedge m)+ (1-m_1^2)=(\partial_x \theta)^2+\sin^2 \theta \big((\partial_x \varphi + \gamma)^2+1-\ga^2\big).
\end{equation}
Note that if in addition $m\in H^2_{\loc}(I, \m S^2)$, then $\theta, \varphi \in H^2_{\loc}(I)$.
In spherical coordinates $(\varphi, \theta)$, under the same constraint that $m(I)\subset \m S^2\setminus \{\pm e_1\}$ in an interval $I$, the Euler-Lagrange equation can be written in $I$ as 
\be
\label{number12}
\begin{cases}
\partial_x \big(\sin^2 \theta (\partial_x \varphi + \gamma) \big)=0 \\
\partial_{xx} \theta- \sin \theta \cos \theta \big((\partial_x \varphi + \gamma)^2+1-\ga^2\big)=0.
\end{cases}
\ee
However, in the proof of  Theorem \ref{thm_DM} below, we cannot take for granted that a solution to \eqref{eq:dm_critical} does not touch the poles. We instead take a different path which completely avoids such a discussion. The above writing is nonetheless useful later, e.g., in the proof of Corollary \ref{cor:precessing_dm}.

\begin{proof}[Proof of  Theorem \ref{thm_DM}]

We divide the proof in several steps:

\medskip

\nd {\it Step 1: Applying a special rotation $R_\phi$ to $m$.} We start by writing the Euler-Lagrange equation\footnote{Due to the Lagrange multiplier $\lambda(x) m$, we can replace $m_2e_2+m_3e_3$ by $-m_1 e_1$ in \eqref{ode_critic}, and we derive the equation for $\lambda(x)$ by doing the (pointwise) scalar product of \eqref{ode_critic} with $m$.}
\begin{align}
\label{ode_critic}
&-\partial_{xx} m+2\ga (\partial_x m_3 e_2-\partial_x m_2 e_3)-m_1 e_1=\lambda(x) m,\\ 
\nonumber
\textrm{with } \quad &\lambda(x)=|\partial_x m|^2+2\ga\partial_x m \cdot (e_1\wedge m)-m_1^2.
\end{align}
As $E(m)<\infty$, in particular, $m\in H^1_{\loc}(\m R, \m S^2)$ is continuous and $\lambda\in L^1_{\loc}(\m R)$, we deduce by a standard bootstrap argument that $m$ and $\lambda$ are smooth in $\R$.
We set 
\[ \phi(x):=-\gamma x, \quad M:=R_{-\phi}m. \]
In particular, $M\in \q C^\infty$.

\medskip

\nd {\it Step 2: Computing the energy $E(m)$ and the equation \eqref{ode_critic} in terms of $M$.} As $m=R_{\phi}M$, $\partial_x \phi=-\ga$ and $\partial_{xx}\phi=0$, we compute 
\begin{align*}
 \partial_x m& =R_\phi \partial_x M+\partial_x(R_\phi)M=R_\phi \big(\partial_x M-\ga e_1\wedge M\big), \\
\partial_x m\cdot (e_1\wedge m)& =(\partial_x M-\ga e_1\wedge M)\cdot (e_1\wedge M)=\partial_x M \cdot (e_1\wedge M) -\ga |e_1\wedge M|^2, \\
\lambda(x) & = |\partial_x M|^2-\ga^2 |e_1\wedge M|^2-M_1^2, \\
\partial_{xx} m & = R_\phi \big(\partial_{xx} M-2\ga e_1\wedge \partial_x M-\ga^2 M\big)+\ga^2 M_1 e_1.
\end{align*}
Therefore, we deduce that
\[ E(m) = \frac 1 2 \int_{\m R} |\partial_x M|^2+(1-\ga^2) |e_1\wedge M|^2 dx, \]
\be
\label{EL_in_M}
-\partial_{xx} M-(1-\ga^2)M_1e_1=\Lambda(x) M, \quad \Lambda(x)=|\partial_x M|^2-(1-\ga^2)M_1^2.
\ee

\medskip

\nd {\it Step 3: We prove that $M$ is a planar transition between $\pm e_1$; that is, there exists a rotation $R_{\tilde \phi}$ such that $R_{\tilde \phi} M$ takes values into the horizontal circle $\m S^1:=\m S^1\times \{0\}\subset \m S^2$. } 
First, we show that 
\be
\label{cte=0}
M_2 \partial_x M_3-M_3\partial_x M_2=0 \quad \textrm{ in } \quad \m R.
\ee
Indeed, since $M\in \q C^\infty$, we compute
\begin{align*}
\partial_x (M_2 \partial_x M_3-M_3\partial_x M_2)&=M_2 \partial_{xx} M_3-M_3\partial_{xx} M_2\\
&\stackrel{\eqref{EL_in_M}}{=}-M_2(\Lambda(x)M_3)+M_3(\Lambda(x)M_2)=0,
\end{align*}
yielding $M_2 \partial_x M_3-M_3\partial_x M_2$ is constant in $\m R$. Since $|\partial_x M|^2, 1-M_1^2=M_2^2+M_3^2 \in L^1(\m R)$ (because $\ga^2<1$ and $E(m)<\infty$), there exists a sequence $x_n\to \infty$ such that 
\[ M_2(x_n), \partial_x M_2(x_n), M_3(x_n), \partial_x M_3(x_n) \to 0 \quad \textrm{ as }  n\to \infty \]
which proves \eqref{cte=0}. This implies that the two vector fields $(M_2, \partial_x M_2)$ and  $(M_3, \partial_x M_3)$ in $\m R^2$ 
are collinear in every point $x$, the collinearity factor could depend a-priori on $x$. However, by \eqref{ode_critic}, these two vector fields solve the same first order linear ODE system in $(u,v)$, i.e., 
\[ \partial_x u=v, \, \partial_x v=-\Lambda(x) u; \]
therefore, by uniqueness in the Cauchy-Lipschitz theorem, the collinearity factor is constant (in $x$), i.e., there exists $\beta\in \m R$ such that
$M_2=\beta M_3$ (or $M_3=\beta M_2$, respectively) in $\m R$. Choosing $\tilde \phi$ such that 
$\cot \tilde \phi=\beta$ (or $\tan \tilde \phi=\beta$, respectively), 
we conclude that the 3rd component of $R_{-\tilde \phi} M$ (in the direction $e_3$) vanishes\footnote{Such an in-plane transition $M$ is called Bloch wall of $180^\circ$. A different type of Bloch wall is studied in \cite{IgnMer}.}.

\medskip

\nd {\it Step 4: Uniqueness (up to translations) of finite energy critical points $M$ in \eqref{EL_in_M} that take values in $\m S^1$ with $M(\pm \infty)=\pm e_1$.} First, every finite energy map $M:\m R\to \m S^1$ (in particular, $M$ is continuous), writes as $M=(\cos \theta, \sin \theta, 0)$ for a continuous lifting $\theta:\m R\to \m R$ (that is unique up to an additive  constant in $2\pi \m Z$); moreover, we compute
\[ E(m) = \frac 12 \int_{\m R} |\partial_x M|^2+(1-\ga^2) |e_1\wedge M|^2\, dx= \frac 12 \int_{\m R} |\partial_x \theta|^2+ (1-\ga^2)\sin^2 \theta\, dx. \]
For such a critical point $M=(\cos \theta, \sin \theta, 0)$, \eqref{EL_in_M} writes as  
\be
\label{crit_theta}
-\partial_{xx} \theta+(1-\ga^2)\sin \theta \cos \theta=0 \quad \textrm{ in } \m R.
\ee
Moreover, as $\theta$ is continuous, it follows that  $\theta\in \q C^\infty$. Multiplying the equation by $\partial_x \theta$, we deduce that 
$-(\partial_{x} \theta)^2+(1-\ga^2)\sin^2 \theta$ is constant in $\m R$. In fact, this constant is zero: indeed, as  $E(m)<\infty$, there exists a sequence $x_n\to \infty$ such that $\partial_x \theta(x_n), \sin \theta(x_n)\to 0$ as $n\to \infty$. In other words, the trajectory $\{X(x)=\big(\theta(x), \partial_x \theta(x)\big)\}_{x\in \m R}$ is included in the zero set of the Hamiltonian
\[ \Ham(X_1, X_2)= \frac 1 2 \big(X_2^2-(1-\ga^2)\sin^2(X_1)\big), \quad X=(X_1, X_2)\in \m R^2. \]
We denote the zero set of $\Ham$ as $Z^-\cup Z^0\cup Z^+$ with 
\[ Z^\pm=\set{(X_1, X_2)}{\pm X_2>0, \, \Ham(X_1, X_2)=0} \] 
and 
\[ Z^0=\set{(X_1, 0)}{\Ham(X_1, 0)=0}=\pi \m Z\times \{0\}. \]
It is readily seen that
any connected component of $Z^+$ and $Z^-$ ends at two consecutive points
of $Z^0$.
Obviously, any point in $Z^0$ is a stationary solution of the dynamical system 
\[ \partial_x X=V(X) \quad \text{with} \quad  V(X) = \left( \frac{\partial \Ham}{\partial X_2},-\frac{\partial \Ham}{\partial X_1} \right)=(X_2, (1-\ga^2)\sin X_1\cos X_1). \]
Therefore, by uniqueness in the Cauchy-Lipschitz theorem for the above system, the trajectory $\{X(x)=(\theta(x), \partial_x\theta(x))\}_{x\in \R}$ begins and ends at two consecutive points in $Z^0$; in particular, the total rotation of a critical point $\theta$ is given by $\ds \int_\R \partial_x \theta \, dx=\pm \pi$. (We refer to \cite{IgUniq} for the uniqueness of domain walls in a generalised model leading to a weighted pendulum equation.)

\medskip

\nd {\it Step 5: Conclusion.} Up to a translation $\tau_y$, we can assume in \eqref{crit_theta} that $M_1(0)=0$ (because $M_1(\pm \infty)=\pm 1$ and $M_1$ is continuous). Furthermore, as $M_1=\cos \theta$, we can assume that $\theta(0)\in \{\pm \frac \pi 2\}$ (up to an additive constant in $2\pi \m Z$).

\begin{itemize}

\item if $\theta(0)=-\frac\pi 2$, then the structure of the zero set of $\Ham$ combined with the boundary condition implies that
$\partial_x \theta>0$ (so, the trajectory $X(x)\in Z^+$ for all $x\in \R$) and $\partial_x \theta=- \sqrt{1-\ga^2}\sin \theta$ with $\theta(-\infty)=-\pi$ and $\theta(\infty)=0$. The uniqueness in the Cauchy-Lipschitz theorem implies $\theta=-\theta_*$ given in \eqref{formul1}. In this case, the critical point $m$ writes as $m=g.w_*^-$ with $g=(y, \tilde \phi)$ with $\tilde \phi$ defined in Step 3.

\item if $\theta(0)=\frac\pi 2$, by the same argument it follows that $\partial_x \theta<0$ (so, the trajectory $X(x)\in Z^-$ for all $x$) and $\partial_x \theta=- \sqrt{1-\ga^2} \sin \theta$ with $\theta(-\infty)=\pi$ and $\theta(\infty)=0$, i.e., $\theta=\theta_*$. In this case, the critical point $m$ writes as $m=g.w_*^+$ with $g=(y, \tilde \phi)$.
\end{itemize}

It remains to prove \eqref{345} below. \label{pagi} If $n^\pm_*$ and $p^\pm_*$ are the quantities defined in \eqref{np} for  $m=w^\pm_*$ with the spherical coordinates $(\varphi_*,\pm \theta_*)$, then we have
\be
\label{345}
\partial_x w^\pm_*=\partial_x (\pm \theta_*) n^\pm_*+\sin(\pm \theta_*) \partial_x \varphi_* p^\pm_*=\sqrt{1-\ga^2} w^\pm_* \wedge(e_1\wedge w^\pm_*)-\ga e_1\wedge w^\pm_*. \quad \quad \qedhere
\ee
\end{proof}

\begin{nb} \label{rem:several} The above proof shows in particular the nonexistence of planar critical points $M=(\cos \theta, \sin \theta, 0)$ in \eqref{crit_theta} with a total angle transition $|\theta(+\infty)-\theta(-\infty)|>\pi$. This leads to the nonexistence of critical points in \eqref{ode_critic} of the form $R_\phi M$ with values in $\m S^2$ with a total angle transition $\theta$ larger than $\pi$. We highlight that such planar critical points do exist if a nonlocal term is added in \eqref{crit_theta}; for example, the existence of the so-called N\'eel walls of angle higher than $\pi$ is proved by Ignat-Moser \cite{IM_JDE, IM_Adv}. For $\Gamma$-convergence results on N\'eel walls at the first and second order, we refer to \cite{Ign, IM_ARMA}. 
\end{nb}

\bigskip

We conclude this section with the proof of Corollary \ref{cor:precessing_dm}:

\begin{proof}[Proof of Corollary \ref{cor:precessing_dm}] 

First let us notice that as $E$ is invariant with respect to the action of the group $G$, 
\begin{align} \label{foot56}
\forall g \in G, m \in \q H^1, \quad  \delta E(g.m)=g.\delta E(m) \quad \text{and} \quad H(g.m) = g.H(m).
\end{align}
Indeed, for fixed $m \in \q H^1$ and $g \in G$,  and for a perturbation map $h \to 0$ in $H^1$, due to the invariance of $E$ with respect to $G$, there hold $E(g.(m +h)) = E(m+ h)$, which yield by Taylor expansion:
\begin{align*}
E(g.m) +  ( \delta E(g.m), g.h )_{H^{-1},H^1}+ o(\|h\|_{H^1})  & = E(m) + ( \delta E(m) , h )_{H^{-1},H^1} + o(\|h\|_{H^1}).
\end{align*}
Hence $ ( \delta E(g.m), g.h )_{H^{-1},H^1} = ( \delta E(m) , h )_{H^{-1},H^1} = ( g. \delta E(m), g.h )_{H^{-1},H^1}$: as this is true for any direction $h\in H^1$, we conclude $\delta E(g.m) =  g. \delta E(m)$. The equality involving $H$ follows immediately.
 
 \bigskip

We now divide the proof in two steps.

\medskip

\nd {\it Step 1.} First, we explain the origin of the equation \eqref{eq_y_phi}. We argue on a formal level, assuming that $m$ has smooth spherical coordinates $(\varphi, \theta)$. We start by writing the gradient $\delta E(m)$ defined in \eqref{def:H} of the energy $E$ in the coordinates \eqref{spheric_coord}. More precisely, in the orthonormal related basis $(m,n,p)$ defined in \eqref{np}, we write
\be
\label{star1} -\delta E(m)=-\rho_0 m+\rho_1 n+\rho_2 p. \ee
By \eqref{energy_spherical:bounds}, as $n=\partial_\theta m$ and $\ds p=\frac{1}{\sin \theta}\partial_\varphi m$, 
the above components $\rho_0$, $\rho_1$ and $\rho_2$ are given by\footnote{Note that $\rho_1$ and $\rho_2$ correspond to the gradient $\delta E(m)$ in spherical variables $\theta$ and $\varphi$ in \eqref{number12} (up to a multiplicative constant). Also, $\rho_0$ is the energy density in $E(m)$ in \eqref{energy_spherical:bounds} due to the quadratic form of $E(m)$.}
\begin{align}
\label{rho*}
\rho_0&=m\cdot\delta E(m)=(\partial_x \theta)^2+ \sin^2 \theta \bigg( (\partial_x \varphi+\ga)^2+1-\ga^2\bigg),\\
\nonumber \rho_1&=-n\cdot\delta E(m)=\partial_{xx} \theta- \sin \theta \cos \theta \bigg( (\partial_x \varphi+\ga)^2+1-\ga^2\bigg),\\
\nonumber 
\rho_2&=-p\cdot\delta E(m)=\sin \theta \partial_{xx} \varphi+2\cos \theta \partial_x \theta (\partial_x \varphi+\ga).
\end{align}
Suppose now that $m$ satisfies \eqref{ll}. As $m$ takes values on the sphere,
\[ \partial_t m=\partial_t \theta n+\sin \theta \partial_t \varphi p. \]
Using \eqref{np}, \eqref{star1} and that $(m,n,p)$ is orthonormal, it follows that \eqref{ll} is equivalent to
\be
\label{intermed}
\begin{cases}
\partial_t \theta = \alpha \rho_1-\rho_2-\alpha h(t)\sin \theta,\\
\sin \theta \partial_t \varphi = \rho_1+\alpha \rho_2-h(t)\sin \theta.
\end{cases}
\ee
We are looking for precessing solutions $m(t,x)$ that are of the form
\be
\label{star2} g_*(t) . w_*^\pm(x), \ee
where $g_* = (y_*,\phi_*0$ is to be determined. They are smooth in $\R_x$ for every fixed time $t$ and can be written in spherical coordinates as in \eqref{preces} (as they avoid the poles $\pm e_1$):
\[ \bigg(\varphi_*(x-y_*(t))+\phi_*(t),\pm \theta_*(x-y_*(t))\bigg), \]
so that the above computations are justified. Now, from the equation \eqref{eq:dm_critical} on $w_*^\pm$, and \eqref{foot56},  $-\delta E(m) = -g. \delta E(w_*^\pm)$ is colinear to $g.w_*^\pm = m$, so that 
\[ \rho_1 = \rho_2 =0. \]
Therefore, by \eqref{intermed} and \eqref{formul1}, the equation \eqref{ll} for solutions \eqref{star2} is equivalent to the evolution of the center $y_*(t)$ and of the rotation angle $\phi_*(t)$ given in \eqref{eq_y_phi}.

\medskip

\nd {\it Step 2.} Second, we check directly (i.e. without using the spherical coordinates \eqref{spheric_coord}) that the map $m_*(t,x)=g_*(t) . w_*^\pm(x)$ (corresponding to \eqref{preces}) is a solution to \eqref{ll}.
Indeed, setting $w_* \in \{w^\pm_*\}$, since $w_*$ is a smooth critical point of $E$, i.e., $w_*\wedge \delta E(w_*)=0$, we compute by \eqref{def:H}:
\begin{align*}
w_* \wedge H(w_*) & = w_* \wedge (h(t)e_1) = - h(t) e_1 \wedge w_*, \\
w_* \wedge (w_* \wedge H(w_*)) & = -h(t) w_* \wedge (e_1 \wedge w_*). 
\end{align*}

Then \eqref{foot56}  yields $H(m_*) = H(g_*(t). w_*) = g_*(t). H(w_*)$. Thus, by the formulas in Appendix \ref{sec:tool}, we obtain:
\begin{align*}
m_* \wedge H(m_*) & = g_*(t). (w_* \wedge H(w_*)) = - h(t) e_1 \wedge m_*, \\
m_* \wedge (m_* \wedge H(m_*)) & = g_*(t). (w_* \wedge (w_* \wedge H(w_*)))=-h(t) m_* \wedge (e_1 \wedge m_*).
\end{align*}
On the other side, differentiating $m_*$ in $t$ and recalling \eqref{formula4}, we get 
\be
\label{2star} \partial_t m_* = - \dot y_* g_*(t). \partial_x w_* + \dot \phi_* g_*(t). e_1 \wedge w_*. \ee
Combining with \eqref{345} and \eqref{eq_y_phi}, we deduce that
\be
\label{eq-funda}
 - \dot y_* \partial_x w_* + \dot \phi_* e_1 \wedge w_*=\alpha h(t) w_* \wedge (e_1 \wedge w_*)-h(t) e_1 \wedge w_*.
\ee
By the above computations, we infer that $m_*$ solves \eqref{ll}.
\end{proof}

For later purposes, we gather below some useful computations on $w_*$.

\begin{lem} \label{lem:comput_w*}
The following holds true: $|\partial_x w_*^\pm|  = | e_1 \wedge w_*^\pm| = \sin \theta_*$ and 
\[ \| \partial_x w_*^\pm \|_{L^2}^2 =  \| e_1 \wedge w_*^\pm \|_{L^2}^2 = \int_{\R}  \sin^2 \theta_* dx = \frac{2}{\sqrt{1-\gamma^2}}. \]
Also, 
\[ \int_{\R}  (e_1 \wedge w_*^\pm) \cdot \partial_x w_*^\pm\, dx  = - \frac{2 \gamma}{\sqrt{1-\gamma^2}} \quad \text{and} \quad  \int_{\R} w^\pm_*  \wedge (  w^\pm_* \wedge e_1) \cdot  \partial_x w^\pm_* dx =-2. \]
Finally, there exists a constant $c_\gamma>0$ such that
\[ \forall g=(y, \phi)\in G, \quad \|-y\partial_x w_*^{\pm}+\phi e_1\wedge w_*^{\pm}\|_{L^2}\ge c_\gamma |g|. \]
In particular, $\partial_x w_*^{\pm}$ and $e_1\wedge w_*^{\pm}$ are linearly independent maps in $L^2$.
\end{lem}

\begin{proof}
By \eqref{formul1}, we have
\[ |e_1 \wedge w_*^\pm|^2 = \sin^2 \theta_*. \]
Therefore, 
\[ \| e_1 \wedge w_*^\pm \|_{L^2}^2 =\int_{\R}  \sin^2 \theta_* dx = - \frac{1}{\sqrt{1-\gamma^2}} \int_{\m R} \partial_x \theta_* \sin \theta_* \, dx= \frac{2}{\sqrt{1-\gamma^2}}. \]
For $\partial_x w_*^\pm$, Appendix \ref{sec:tool} yields
\[ w_*^\pm \wedge (e_1 \wedge w_*^\pm) = e_1 - \cos \theta_* w_*^\pm, \]
so that using \eqref{345} and orthogonality,
\[ | \partial_x w_*^\pm|^2 = (1-\gamma^2) (1 + \cos^2 \theta_* - 2 \cos \theta_* e_1 \cdot w_*^\pm) + \gamma^2 |e_1 \wedge w_*^\pm|^2 = \sin^2 \theta_*. \]
Then, using again \eqref{345}, we get
\[ (e_1 \wedge w_*^\pm) \cdot \partial_x w_* ^\pm= - \gamma |e_1 \wedge w_*^\pm|^2=-\gamma \sin^2 \theta_*, \]
and the desired identity follows by integration. Also, as $w^{\pm}_*\cdot \partial_x w^{\pm}_*=0$, we have
\[ \int_{\R} w^{\pm}_*  \wedge (  w^{\pm}_* \wedge e_1) \cdot  \partial_x w^{\pm}_* dx =  \int_{\R} \bigg( \cos \theta_* w_*^\pm - e_1\bigg)\cdot \partial_x w^{\pm}_*  dx = - \int_{\R}  \partial_x( \cos \theta_*) dx=-2. \]
For the last statement, using the above identities, we compute
\[ \|-y\partial_x w_*^{\pm}+\phi e_1\wedge w_*^{\pm}\|^2_{L^2}=(y^2+\phi^2+2\gamma y\phi)
\frac{2}{\sqrt{1-\gamma^2}}\ge  \sqrt{1-\gamma^2} (y^2+\phi^2). \qedhere \]
\end{proof}

\section{Asymptotic stability}

The aim of this section is to prove Theorem \ref{th:stab}. Recall that $\q H^s\subset \q H^1$ for $s \ge 1$ and $\q H^1$ is the space of finite energy configurations $E(m)<\infty$ (see Lemmas \ref{lem:Hs_H1} and \ref{lem:infinit}). As here the integrals are on $\m R$, we always denote $\ds \int:=\int_{\R}$. An important fact consists in the locally well-posedness in time of our equation \eqref{ll}, in the Hadamard sense, in $\q H^s$ for $s \ge 1$. More precisely, we have the following result:

\begin{thm}[Local well-posedness in $\q H^s$] \label{th:lwp} Let $\alpha >0$, $\gamma \in (-1,1)$ and $h \in L^\infty([0,+\infty), \m R)$. Assume $s \ge 1$ and $m_0 \in \q H^s$. Then there exist a maximal time $T_+= T_+(m_0) \in (0, +\infty]$ and a unique solution $m \in \q C([0, T_+), \q H^s)$ to \eqref{ll} with initial data $m_0$. 

Moreover,
\begin{enumerate}
\item if $T_+ <+\infty$, then $\| m(t) \|_{\q H^1} \to +\infty$ as $t \uparrow T_+$;
\item for $T<T_+$ (with $T_+$ finite or infinite), the map $\tilde m_0\in \q H^s \to \tilde m\in \q C([0,T], \q H^s)$ is continuous in a small $\q H^s$ neighbourhood of $m_0$ (for every initial data $\tilde m_0$ in that neighborhood, the maximal time of the corresponding solution $\tilde m$ satisfies $T_+(\tilde m_0) > T$);
\item if $s \ge 2$, one has the energy dissipation identity\footnote{Observe that in the energy dissipation identity \eqref{eq:en_dissip2}, all the terms are rightfully integrable because $m(t)\in \q H^s$ with $s\ge 2$, so $E(m(t))<\infty$, $m(t) \wedge e_1 \in L^2$ and $\delta E(m(t)) \in L^2$ for all $t<T_+$.}: $t \mapsto E(m(t))$ is a locally Lipschitz function in $[0,T_+)$ (even $\q C^1$ provided $h$ is continuous) and for all $t \in [0,T_+)$,
\begin{align}
\frac{d}{dt} E(m) = - \alpha \int ( |\delta E(m)|^2 - |m \cdot \delta E(m)|^2 )\, dx + \alpha h(t) \int (m \wedge e_1) \cdot (m \wedge \delta E(m)) \, dx. \label{eq:en_dissip2}
\end{align}
\end{enumerate}
\end{thm}

\begin{proof}
We only give a sketch on how the proof should be conducted because a full proof would be lengthy (but not too technical), and is not the primary purpose of this paper. As we need a well-posedness result in the sense of Hadamard, we cannot resort to weak solutions as in Alouges and Soyeur \cite{AS92}. We instead follow the idea of Tsutsumi \cite{Tsu08}, which is to perform a stereographic projection $\Pi_{p_0}$ around a point $p_0 \in \m S^2$ from which $m_0$ remains away. Such point $p_0$ exists because as $m_0$ admits limits at $\pm \infty$ 
(see Lemma~\ref{lem:infinit}), $m_0$ remains in a neighbourhood of the poles $\{\pm e_1\}$ outside $x \in [-A,A]$, while on the compact $[-A,A]$, $m_0$ has finite $H^1$ norm, so that its image $m_0([-A,A])$ cannot cover any two-dimensional cap on $\m S^2$ (see Lemma \ref{lem:bv_S2} in the Appendix). $\m S^2\setminus \{p_0\}$ is mapped into the plane $\m R^2\sim \m C$ under the stereographic projection $\Pi_{p_0}$, and \eqref{ll} is transformed into a quasilinear dissipative Schr\"odinger equation for $u=\Pi_{p_0} m$ of the type
\[ i \partial_t u + (1 - i \alpha) \partial_{xx} u = 2 (1 - i \alpha) \frac{\bar u |\partial_x u|^2}{1+|u|^2} + \text{lower order terms}, \]
with $\bar u$ the complex conjugate of $u$ (see e.g. \cite{LN84} for the exact computation). The dissipative linear semigroup of this equation enjoys smoothing properties similar to those of the heat equation, so that the above formulation is amenable to a Banach fixed point argument in $H^1$, or in $H^s$ for $s \ge 1$. The first conclusions of Theorem \ref{th:lwp} follow. We also refer to \cite{GdL17} (for a the construction of mild solutions for data with small BMO norm) and the reference therein for further details\footnote{We underline that the Cauchy theory for the Laudau-Lifshitz-Gilbert equation is not completely satisfactory: the  equation is $\dot H^{1/2}$ critical for its Schr\"odinger parts, so one could expect a result for such regularity. Theorem \ref{th:lwp} is not meant to be optimal as $m_0 \in \q H^s$ with $s \ge 1$, but it is enough for our purposes. Much sharper results exist in the slightly different (and harder) context of Schr\"odinger maps in critical spaces, see  \cite{BIK07,BIKT11} (even though they are a priori estimates on \emph{smooth} solutions).}.

\smallskip

\noindent {\it Proof of the energy dissipation \eqref{eq:en_dissip2}}.
Here,  $m\in \q H^s$, $s\ge 2$, and so $\delta E(m) =  -\partial_{xx} m - 2 \gamma e_1 \wedge \partial_x m + m_2 e_2 + m_3 e_3 \in L^2$ (by Lemma \ref{lem:Hs_H1}) and so, $\partial_t m\in L^2$ (by \eqref{ll}) yielding
\begin{align*}
\frac{d}{dt} E(m) & = \int \partial_t m \cdot \delta E(m) dx \\
& = \int (m \wedge H(m)) \cdot \delta E(m) dx - \alpha \int (m \wedge ( m \wedge H(m))) \cdot \delta E(m) dx = I + II. 
\end{align*}
First we consider the precession term $I$. Using $H(m) = -\delta E(m) + h e_1$ with $h=h(t)$, we have
\[ I = h \int (m \wedge e_1) \cdot \delta E(m) \, dx. \]
As $m\in \q H^s$, $s\ge 2$ (in particular, $m(\pm \infty)\in\{\pm e_1\}$), integration by parts yields
\[ \int (m \wedge e_1) \cdot \partial_{xx} m \, dx = 0. \]
Also integration by parts and skew-symmetry of $\wedge$ lead to
\[ \int (m \wedge e_1) \cdot (e_1 \wedge \partial_x m) dx = - \int (\partial_x m \wedge e_1) \cdot (e_1 \wedge m) \, dx =0. \]
Since $(m \wedge e_1) \cdot (m_2e_2+m_3e_3)=(m \wedge e_1) \cdot (m-m_1e_1)=0$, by the above computations, we deduce that $I=0$: the precession term has no contribution.

\smallskip

We now consider the damping term $II$, and compute using Appendix \ref{sec:tool}:
\begin{align*}
\MoveEqLeft m \wedge ( m \wedge H(m)) \cdot \delta E(m) =  m \wedge ( m \wedge (-\delta E(m)) \cdot \delta E(m) + h   m \wedge ( m \wedge e_1) \cdot \delta E(m) \\
& =  |m \wedge \delta E(m)|^2 - h (m \wedge e_1) \cdot (m \wedge \delta E(m)) \\
& = (|\delta E(m)|^2 - |m \cdot \delta E(m)|^2) - h (m \wedge e_1) \cdot (m \wedge \delta E(m)).
\end{align*}
Integrating in space yields \eqref{eq:en_dissip2}.
\end{proof}

\bigskip

We will first prove Theorem \ref{th:stab} for initial data $m_0 \in \q H^2$ which is close to $w_*$ in $H^1$, and then in Section \ref{sec:44}, we use a limiting argument to relax the regularity  $m_0 \in \q H^2$ to  $m_0 \in \q H^1$. Hence, until Section~ \ref{sec:4.3} included (as recalled in the statement of the results),  we consider 
\begin{align}
m_0  \in \q H^2
\end{align}
and denote $m \in \q C([0,T), \q H^2)$ the solution to \eqref{ll} with initial data $m_0$.

\subsection{Modulation and decomposition of the magnetisation}

The first ingredient in proving Theorem \ref{th:stab} is a modulation result in Lemma \ref{lem:mod} below, which is a vector-valued version of the scalar context presented in \cite{Wei86} or \cite{MM01a}. This modulation is performed with respect to the gauge group $G$, so that the orthogonality conditions (see \eqref{eq:ortho} below) are induced by the map $g\in G\mapsto g.w_*$ with $w_*\in \{w_*^\pm\}$ and describe the $L^2$-orthogonality conditions  to the tangent plane\footnote{By Lemma \ref{lem:comput_w*}, $\partial_x w_*$ and $e_1 \wedge w_*$ are $L^2$ linearly independent. However, they are not orthogonal in $L^2$.} (up to a gauge) $\Span \{\partial_x w_*, e_1 \wedge w_*\}$ of the target space of the map $g\mapsto g.w_*$ (due to \eqref{2star}). Lemma \ref{lem:mod} is a result at the stationary level. Then, in Proposition \ref{lem:decomp} below, we prove  the dynamical properties of this modulation.

\begin{lem}[Modulation] \label{lem:mod} Let $w_*\in \{w_*^\pm\}$ and  set  the space
\[ \q M=\{w\in \q C(\m R, \m S^2)\, :\, w-w_* \in H^1(\R, \R^3)\} \]
endowed with the $H^1$ distance.
There exist $\delta_1>0$ and $C_1  \ge 1$ such that the following holds. For all $w \in \q M$ such that
\[ \| w - w_* \|_{H^1} \le \delta_1, \]
there exists a gauge $g = (y,\varphi) \in G$ such that if we consider the map $ \e\in 
H^1(\R,  \R^3)$ defined by 
\[ \e=(-g).w-w_*, \quad \text{i.e., } \quad w = g.(w_* + \e), \]
then $2 \e \cdot w_* + |\e|^2 =0$ for all $x \in \m R$,  and the orthogonality constraints hold
\begin{align} \label{eq:ortho}
\int \e \cdot \partial_x w_* \, dx =0, \quad \int \e \cdot (e_1 \wedge w_*)\,  dx =0;
\end{align}
furthermore, 
\begin{align} \label{est:g_w-w*}
|g| + \| \e \|_{H^1} \le C_1 \| w - w_* \|_{H^1},
\end{align}
 $g$ is the unique gauge with the above properties and the map $w \in \q M \mapsto g\in G$ is of class $\q C^\infty$ in a neighbourhood of $w_*$.
\end{lem}

\begin{nb}
\label{rem:23}
Observe that $\e=(-g).w-w_*$ is as smooth as $w$, and the result holds in $H^s$ for any $s \ge 1$, that is, if $w$ is a small perturbation of $w_*$ in $\q M\cap \dot{H}^s$, then $\| \e \|_{H^s} \le C_s \| w - w_* \|_{H^s}$.
\end{nb}

\begin{nb}
Observe that $ \Span \{\partial_x w_*, e_1 \wedge w_*\}=\Span \{n_*, p_*\}$, where the orthonormal frame $(w_*,n_*,p_*)$ is defined by \eqref{np} (and written explicitly below in \eqref{eq:w*_spheric}). Then, due to \eqref{345} the orthogonality conditions \eqref{eq:ortho} are equivalent to 
\[ \int \e \cdot \sin \theta_* n_* \, dx =\int \e \cdot \sin \theta_* p_* \, dx=0. \]
As $\sin \theta_*$ is smooth and exponentially localised, one can also perform the modulation with respect to $\{ \sin \theta_* n_*, \sin \theta_* p_*\}$ instead of $\{\partial_x w_*, e_1 \wedge w_*\}$.
\end{nb}

\begin{proof}

{\it Step 1. Existence of $g$.} Observe that $H^1(\m R)$ endowed with its natural topology is a Banach algebra. Then
\[ K :\begin{array}[t]{r@{\ }c@{\ }l} 
 H^1 (\m R,\m R^3) & \to & H^1(\m R) \\
 \e & \mapsto & 2 \e \cdot w_* + |\e|^2
 \end{array}
 \]
is well defined and $K$ is a continuous polynomial in $\e$, in particular, it is a $\q C^\infty$ map (see Remark \ref{rem23}). Define 
\[ \tilde {\q M} = K^{-1}(\{ 0 \}) = \{ \e \in H^1(\m R, \m R^3): 2 \e \cdot w_* + |\e|^2 =0 \}, \]
and let $\e_0 \in \tilde{\q M}$. As $\nabla K(\e_0). \eta = 2 \eta\cdot (w_* + \e_0)$ and $|w_*+\e_0|=1$, we deduce that for any $f \in H^1(\m R)$, $\nabla K(\e_0).\big(\frac12 f (w_* + \e_0)\big) = f$ so that $\nabla K(\e_0)$ is surjective. As a consequence, $\tilde{\q M}$ is a $\q C^\infty$ Banach submanifold of $H^1(\m R, \m R^3)$, and so $\q M = w_* + \tilde{\q M}$ is a $\q C^\infty$ Banach submanifold of the Banach affine space $w_*+ H^1(\m R, \m R^3)$. We now claim that the function
\begin{gather*}
F: \begin{array}[t]{r@{\ }c@{\ }l} \q M \times G & \to &  \m R^2  \\
(w,g) & \mapsto & \ds  (F_1(w,g), F_2(w,g)): = \left( \int \e \cdot \partial_x w_* dx,  \int \e \cdot (e_1 \wedge w_*) dx \right),
\end{array} \\
\text{where} \quad \e := (-g). w-w_*.
\end{gather*}
is $\q C^\infty$ between Banach manifolds. Indeed, the map $\q M \times G \to \tilde {\q M}$, $(w,g) \mapsto \e$ is well defined  and of class $\q C^\infty$ (the set $G=(\m R^2, +)$ is always endowed with the standard topology). Since $\partial_x w_*, e_1 \wedge w_* \in L^2$ (in fact, these two maps are smooth and exponentially localised by \eqref{formul1}), by composition, $F$ is $\q C^\infty$. Writing $g=(y, \varphi)$, we compute $\partial_y \e = (-g).\partial_x w$, $\partial_\varphi \e = - e_1 \wedge (-g).w$, so that
\begin{align}
\partial_{g} F(w_*,(0,0)) & = \begin{pmatrix}
\partial_y F_1 & \partial_\varphi F_1 \\
\partial_y F_2 & \partial_\varphi F_2
\end{pmatrix} = \begin{pmatrix} 
 \| \partial_x w_* \|_{L^2}^2 & - \ds \int (e_1 \wedge w_* ) \cdot \partial_x w_* dx \\
 \ds \int \partial_x w_* \cdot (e_1 \wedge w_*) dx & -\| e_1 \wedge w_* \|_{L^2}^2 \end{pmatrix} \nonumber \\
& = \| \sin \theta_* \|_{L^2}^2
\begin{pmatrix} 
1 & \gamma \\
- \gamma & -1
\end{pmatrix} = \frac 2 {\sqrt{1-\gamma^2}} \begin{pmatrix} 
1 & + \gamma \\
- \gamma & -1
\end{pmatrix}, 
\label{eq:d_gF}
\end{align}
where we used Lemma \ref{lem:comput_w*}. As $\gamma^2 <1$, it follows that $\partial_g F(w_*, (0,0))$ is invertible. Since $F(w_*, (0,0))=(0,0)$, the implicit function theorem (on Banach manifolds) applies and yields a unique $\q C^\infty$ map $\Psi: w \mapsto g$ defined on a neighbourhood of $w_*$ in $\q M$. The orthogonality conditions \eqref{eq:ortho} are encoded in $F$ and the pointwise equality $2 \e \cdot w_* + |\e|^2 =0$ is encoded in $\tilde {\q M}$. The regularity statement at the end of the Lemma is simply the regularity of $\Psi$.

\medskip

{\it Step 2. Estimates on $g$ and $\e$.} Up to reducing the neighbourhood of $w_*$, we can always assume  $\Psi$ is Lipschitz in that neighbourhood and $|\Psi(w)|\le 1$ for every $w$ in that neighborhood. Then the bound on $g$ of \eqref{est:g_w-w*} follows from Lipschitz continuity of  $\Psi$ and the fact $F(w_*,(0,0))=(0,0)$ (so that $\Psi(w_*) =(0,0)$). Finally, we have for some $C>0$:
\begin{align}
\nonumber \|\e\|_{H^1} & \le \|(-g).(w-w_*)\|_{H^1} +  \|(-g).w_*-w_* \|_{H^1} \\
 \nonumber & \le  \|w-w_*\|_{H^1} + \| R_{-\phi} .\tau_{-y} w_*- \tau_{-y} w_*\|_{H^1} + \| \tau_{-y} w_* - w_* \|_{H^1} \\
 \label{here-a} & \le \|w-w_*\|_{H^1} + C |\phi| \| e_1 \wedge w_* \|_{H^1} +  C |y| \| \partial_x w_* \|_{H^1},
 \end{align}
 and from there and the estimate on $g=(y, \phi)$, the bound on $\e$ in \eqref{est:g_w-w*} follows. If in addition $w\in \dot{H}^s$ for $s\ge 1$, then a similar computation yields $\|\e\|_{H^s}\le  \|w-w_*\|_{H^s} + C_s |g|$  $\le \tilde C_s \|w-w_*\|_{H^s}$ as $|g|\le C \|w-w_*\|_{H^1} \le C \|w-w_*\|_{H^s}$.
\end{proof}

Using the above modulation lemma at the stationary level, we can dynamically decompose in Proposition \ref{lem:decomp} below a magnetisation $m$ on a time interval where it remains close to the domain wall family $G .w_* = \{ g.w_* : g \in G \}$ (where $w_* \in \{ w^\pm_* \}$). Before we proceed, let us introduce some notation which will be useful to perform those computations.

\begin{defi}
\label{def:O}
We denote $W^{2,\infty}(\m R)$ the set of bounded Lipschitz functions in $\m R$ whose derivative is also Lipschitz.\footnote{The $W^{2,\infty}(\m R)$ regularity of coefficients $\beta$ is needed in Claim \ref{cl:O_IPP} below, in order to perform two integrations by parts and bound them in $L^\infty$.} If $k \ge 0$ and $\ell \ge 1$, for a given (possibly) vector-valued function $f=(f_j)_{1\le j \le J}$, we use the notation 
\[ u = O_k^\ell(f) \]
for a (possibly vector valued) function $u$ if (each component of) $u$ is an homogenous polynomial of total degree $\ell$ in (the components of) $f$ and their derivatives up to order $k$ such that the total number of derivatives in each term is at most $k$, and the coefficients in this polynomial are $W^{2,\infty}(\m R)$ functions. That is, $u$ is the sum
of terms of the form
\begin{gather*} 
\beta \prod_{j=1}^J \prod_{\kappa=0}^{\kappa_j} (\partial_x^{\kappa} f_j)^{\ell_{j,\kappa}} \ \ \text{where} \ \ \ell_{j,\kappa}\ge 0, \  \ \sum_{j=1}^J \sum_{\kappa =0}^{\kappa_j} \ell_{j,\kappa} =\ell, \ \ \sum_{j=1}^J \sum_{\kappa =0}^{\kappa_j} \ell_{j,\kappa} \kappa \le k \ \ \text{and} \ \  \beta  \in W^{2,\infty}(\m R).
\end{gather*}
\end{defi} 
In practice, we will only use $1\le \ell \le 6$ and $k \le 4$. 

\begin{exa}
If $k=0$, then $u=\sum \beta \prod_j f_j^{\ell_j}$ with $\sum_j \ell_j=\ell\ge 1$ and coefficients $\beta\in W^{2,\infty}(\m R)$. If $k=1$, then
$u=\sum \beta (\partial_x f_{s})^{l_s}\prod_{j} f_j^{\ell_j}$ with $l_s\in \{0,1\}$, 
$\sum_{j} \ell_j=\ell-l_s$ and coefficients $\beta\in W^{2,\infty}(\m R)$.
\end{exa}

Here are a few direct consequences of Definition \ref{def:O}:

\begin{claim} \label{cl:Okl_1}
(i) $O_k^\ell(f) \cdot O_{\tilde k}^{\tilde \ell}(\tilde f) = O^{\ell+\tilde \ell}_{k+\tilde k}(f, \tilde f)$. \\
(ii) If $\tilde k \ge k$, then $O^\ell_k(f) = O^\ell_{\tilde k}(f)$. \\
iii) If $\beta\in W^{2,\infty}(\m R)$, then $\beta O^\ell_k(f) = O^\ell_k(f)$. \\
(iv) If $u = O^\ell_k(f)$ and all the coefficients $\beta$ in the polynomial $u$ belong to $W^{3,\infty}(\m R)$, then $\partial_x u = O^\ell_{k+1}(f)$.
\end{claim}

Definition \ref{def:O} is made to express pointwise bounds that will turn into Sobolev bounds after integration \emph{with linear dependance in the highest term} provided $k\ge 2$, as follows:

\begin{claim} \label{cl:Okl_2}
Assume that $u=  O_k^\ell(f)$ with $\ell\ge 1$. If $k \ge 2$ and $f\in H^k$, then there holds
\begin{align} \label{est:O_H2}
\| u \|_{L^2} \lesssim \| f \|_{H^{k-1}}^{\ell-1} \| f \|_{H^k}. 
\end{align}
If $k \in \{0,1\}$ and $f\in H^1$, then
\begin{align} \label{est:O_H1}
\| u \|_{L^2}  \lesssim  \| f \|_{H^1}^\ell.
\end{align}
The implicit constants depend on the $L^\infty$ norm of the coefficients $\beta$ in the polynomial $u$.
\end{claim}

\begin{proof}
In the case $k=1$, this is a direct application of the fact that $H^1(\R) \hookrightarrow L^\infty(\R)$. Indeed, derivatives occurs at most once on one factor of each term: the derivative is estimated in $L^2$, and the other $\ell-1$ factors are estimated in $L^\infty$ (which is controlled by the $H^1$ norm). For the terms without any derivative, one factor is estimated in $L^2$ and the others in $L^\infty$. Thus, \eqref{est:O_H1} is proved for $k=1$.

\medskip

If $k=0$, then we use $O_0^\ell(f)=O_1^\ell(f)$ and the above result for $k=1$.

\medskip

For $k \ge 2$, the statement is slightly more subtle. First, consider the terms where a derivative of order $k$ occurs. This term has the form $\beta \partial_x^k f_{1} 
\prod_{j} f_j^{\ell_j}$ with 
$\sum_{j} \ell_j=\ell-1$ and $\beta\in W^{2,\infty}(\m R)$; the factor $\partial_x^k f_{1}$  is estimated in $L^2$, and all the other factors (appearing without derivative) are estimated in $L^\infty$ that is controlled by the $H^{k-1}$-norm (as $k \ge 2$), yielding a final bound as in   \eqref{est:O_H2}. Second, we treat the other terms as follows, in two categories.

\smallskip

 For the terms where a derivative of order $k-1$ occurs, the $k-1$ derivative can appear at most twice (because otherwise, having three times such derivatives, it would mean $3(k-1) > k$ as $k \ge 2$ which contradicts the number of derivatives in Definition \ref{def:O}). 
 
 (1) For terms where two derivatives of order $k-1$ occur, we write them in the form $v=\beta \partial_x^{k-1} f_1 \partial_x^{k-1} f_2\cdot \prod_{j} (\partial_x^{\kappa_j} f_j(x))^{\ell_j}$ with $\kappa_j \le k-2$ and $\sum_{j} \ell_j=\ell-2$ where we allow that $f_1$ may coincide with $f_2$. 
Then, using the Gagliardo-Nirenberg inequality for some $h\in H^1(\R)$, we have
\begin{align} \label{est:GN}
\| h \|_{L^4}^2 \le \| h \|_{L^2} \| h \|_{L^\infty} \le \sqrt{2} \| h \|_{L^2}^{3/2} \| \partial_x h \|^{1/2}_{L^2},
\end{align} 
which combined with Cauchy-Schwarz, it yields
\begin{align*}
\| v \|_{L^2}^2 & \le\|\beta\|^2_{L^\infty}  \| \partial_x^{k-1} f_1 \|_{L^4}^2  \| \partial_x^{k-1} f_2 \|_{L^4}^2
\prod_{j} \|\partial_x^{\kappa_j} f_j \|_{L^\infty}^{2\ell_j}   \\
& \lesssim \prod_{s=1,2} \| f_s \|^{1/2}_{H^k} \| f_s \|_{H^{k-1}}^{3/2} \cdot \prod_{j} \| f_j \|_{H^{k-1}}^{2 \ell_j}  \lesssim  \prod_{s=1,2} \| f_s \|_{H^k} \| f_s \|_{H^{k-1}} \cdot
\prod_{j} \| f_j \|_{H^{k-1}}^{2 \ell_j}
\end{align*}
where the implicit constant depends only on $\|\beta\|_{L^\infty}$. This again gives the estimate \eqref{est:O_H2}.

\smallskip

{(2)} For the other terms, there is at most one derivative of order at $k-1$, the other are of order at most $k-2$. Then for such terms, the factor with highest order of derivative is estimated in $L^2$ and all the others in $L^\infty$ which is controlled by the 
$H^{k-1}$-norm, yielding \eqref{est:O_H2}.
\end{proof}

Also, we need the following type of estimates; the choice of considering homogeneous polynomials with coefficients in $W^{2, \infty}$ in Definition \ref{def:O} is in particular motivated by the possibility to perform twice the integration by parts in the following claim.

\begin{claim} \label{cl:O_IPP}
If $f, v \in H^1(\R)$, then\footnote{Since for $k\in \{0,1\}$, $O^\ell_k(f)=O^\ell_2(f)$, these estimates hold also for $O^\ell_0(f)$ and $O^\ell_1(f)$ for every $\ell\ge 1$.} for $\ell \ge 2$: 
\[ \left| \int O^\ell_2(f) dx \right| \lesssim \| f\|_{H^1}^\ell, \quad \left| \int O^1_{2}(f) v dx \right| \lesssim \| f \|_{H^1} \| v \|_{H^1}, \]
where the implicit constants depend on the $W^{1, \infty}$ norm of the coefficients in $O^\ell_2$.

If $f \in H^2(\R)$, then for $\ell \ge 2$:
\begin{align*}
\left| \int O^\ell_3(f) dx \right|  \lesssim \| f\|_{H^1}^{\ell-1}  \| f \|_{H^2}, \quad
\left| \int O^\ell_4(f) dx \right|  \lesssim \| f\|_{H^1}^{\ell-2}  \| f \|_{H^2}^2,
\end{align*}
where the implicit constants depend on the $W^{2, \infty}$ norm of the coefficients in $O^\ell_k(f)$. \end{claim} 

\begin{nb}
One should be aware that in $O^\ell_2(f)$, terms of the type $\partial^2_{x} f_1 \cdot f_2$ may occur, so that the notation $\int O^\ell_2(f) dx$ should be understood in the $H^{-1}-H^1$ duality.
Similarly, some terms in $\int O^\ell_4(f) dx$ should be understood in the $H^{-k}-H^k$ duality for $k=1$ or $2$.
\end{nb}

\begin{proof}
We prove only the last estimate, all the other estimates follow by the same argument (combined with Claim \ref{cl:Okl_2} or the arguments in its proof). Also, by a density argument, we can furthermore assume without loss of generality $f \in \q C_c^\infty(\m R)$. After at most two integrations by parts in the integral $\ds \int O^\ell_4(f) dx$, the coefficients of the new polynomials in $f$ and its derivatives can be estimated in $L^\infty$, and the worst terms  appearing are of the form:
\[ \int \partial_{xx} f_1 \partial_{xx} f_2 p_{\ell-2} dx, \quad \int \partial_{xx} f_1 \partial_{x} f_2  \partial_{x} f_3 p_{\ell-3} dx, \quad \text{or} \quad \int \partial_{x} f_1 \partial_{x} f_2  \partial_{x} f_3 \partial_{x} f_4 p_{\ell-4} dx, \]
where $p_l$ is a product of the $(f_j)_j$ of degree $l$ without derivatives in $f_j$ and $f_1$ may coincide with $f_2$, $f_3$ or $f_4$ in the above terms etc.; the other terms have less derivatives and are estimated by the same argument. Using the embedding $H^1(\R) \hookrightarrow L^\infty(\R)$ and \eqref{est:GN}, we get
\begin{align*}
\left| \int \partial_{xx} f_1 \partial_{xx} f_2 p_{\ell-2} dx \right| & \lesssim \| f \|_{L^\infty}^{\ell-2} \| f \|_{H^2}^2 \le  \| f \|_{H^1}^{\ell-2} \| f \|_{H^2}^2, \\
\left|  \int \partial_{xx} f_1 \partial_{x} f_2  \partial_{x} f_3 p_{\ell-3} dx \right| & \le \| f \|_{H^2} \| \partial_x f \|_{L^4}^2 \| f \|_{L^\infty}^{\ell-3} \lesssim  \| f \|_{H^1}^{\ell-3/2} \| f \|_{H^2}^{3/2}, \\
\left| \int \partial_{x} f_1 \partial_{x} f_2  \partial_{x} f_3 \partial_{x} f_4 p_{\ell-4} dx \right| & \lesssim \| \partial_x f \|_{L^4}^4 \| f \|_{L^\infty}^{\ell-4} \le \| f \|_{H^1}^{\ell-1} \| f \|_{H^2}. \qedhere
 \end{align*}

\end{proof}

\medskip

Now we state the modulation result for the time-dependent magnetisation:

\begin{prop}[Decomposition of the magnetisation] \label{lem:decomp} Let $w_* \in \{ w^\pm_* \}$. 
There exist two constants $C_2 >0$ and $\delta_2^0 >0$ such that the following holds. If $0 <\delta_2 < \delta_2^0$, $T>0$, $h \in L^\infty([0,T])$ and  $m \in \q C([0,T], \q H^2)$ is a solution to \eqref{ll} such that
\begin{equation} \label{est:m_w_delta}
\forall t \in [0,T], \quad \inf_{\zeta \in G} \| m(t) - \zeta . w_* \|_{H^1} < \delta_2,
\end{equation}
then there exist a Lipschitz gauge $g = (y, \phi): [0,T] \to G$ and a continuous map  $\e: [0,T] \to H^2$ 
such that
\begin{gather} 
m(t,x) =  g(t). (w_*(x) + \e(t,x))\in \m S^2, \label{def:dec_m_e}
\end{gather}
with the orthogonality conditions, i.e.,  for all $t \in [0,T]$, $\e(t)$ satisfies \eqref{eq:ortho} and 
\begin{align}
\forall x \in \m R, \quad |\e(t,x)|^2 + 2 \e(t,x) \cdot w_*(x) & =0. \label{intoS2}
\end{align}
Moreover, $\partial_t \e \in L^\infty([0,T], L^2)$ and 
\begin{gather} \label{est:g_e}
\forall t \in [0, T], \quad | \dot g (t)- \dot g_*(t)| \le C_2 (1+|h(t)|) \| \e (t) \|_{H^1}, \quad \| \e(t) \|_{H^1} \le C_2 \delta_2, 
\end{gather}
where $g_*=(y_*, \phi_*)$ is given in \eqref{eq_y_phi}.  Moreover, the gauge map $g$ is unique up to a constant in $\{0\}\times 2\pi \Z$.
\end{prop}

We emphasise that the bounds in \eqref{est:g_e} (in particular $C_2$) do not depend on the length $T>0$ of the interval $[0,T]$.

\begin{proof} 
Let $\delta_2^0>0$ be chosen later and $0 < \delta_2<\delta^0_2$.

\medskip

\emph{Step 1. Construction of $g$ and $\e$}. First, observe that one can choose a smooth gauge for a map $m$ satisfying \eqref{est:m_w_delta}, i.e.,  for some universal constant $C_2'$ (not depending on $m$, $T$ or $\delta_2$), there exists a smooth gauge $g_0 : [0,T] \to G$ such that
\begin{align}
\label{def:tilde_g}
\forall t \in [0,T], \quad  \| m(t) - g_0(t).w_* \|_{H^1} \le C_2' \delta_2.
\end{align}
Indeed, as $m: [0,T] \to w_* + H^1$ is continuous (by \eqref{est:m_w_delta} and Remark \ref{rem23}), and so uniformly continuous, there 
exists a finite number of times $0 = t_0 < t_1 < \cdots < t_N = T$ such that for $k =0,\dots, N-1$,
\[ \forall t \in [t_k,t_{k+1}], \quad \| m(t) - m(t_k) \|_{H^1} \le \delta_2. \] 
Also, in view of \eqref{est:m_w_delta}, for all $k =0, \dots, N$, there exists $\zeta_k\in G$ such that $\| m(t_k) - \zeta_k.w_* \|_{H^1} \le \delta_2$. In particular, $\| \zeta_k.w_* - \zeta_{k+1}.w_*  \|_{H^1} \le 3 \delta_2$. To prove \eqref{def:tilde_g}, we need the following claim:

\begin{claim} \label{claim:beauty} 
There exist $\delta'>0$ and $C'>0$ such that the following holds. If $g, \tilde g\in G$ such that $\| g.w_* - \tilde g.w_* \|_{H^1} \le \delta'$, then there exists $n \in \m Z$ such that 
\begin{equation}
\label{ineg:beauty}
| g -\big (\tilde g+ (0,2\pi n)\big)| \le C' \| g.w_* - \tilde g.w_* \|_{H^1}.
\end{equation}
Moreover, if $g, \tilde g:[0,T]\to G$ are continuous gauge functions such that $\| g(t).w_* - \tilde g(t).w_* \|_{H^1} \le \delta'$ for every $t\in [0,T]$ then \eqref{ineg:beauty} holds for every $t\in [0,T]$ for a constant $n\in \Z$ independent of $t$.
\end{claim}

\begin{nb}
By \eqref{here-a}, $\| g.w_* - \tilde g.w_* \|_{H^1}\le C|g-\tilde g|$ for every $g, \tilde g\in G$ where $C=C(\|w_*\|_{\q H^2})$. Inequality \eqref{ineg:beauty} somehow provides an inverse inequality (up to an additive constant in $\{ 0 \} \times 2\pi \m Z$); note that $\big( \tilde g+ (0,2\pi n) \big).w_*=\tilde g.w_*$).
\end{nb}

\begin{proof}[Proof of claim] 
Indeed, assume by contradiction that the claim fails. As $ \| g.w_* - \tilde g.w_* \|_{H^1} =  \| (g-\tilde g).w_* - w_* \|_{H^1}$, there would exist a sequence $(g_j)_j = (y_j,\phi_j)_j$ in $G$ such that $\| g_j . w_* - w_* \|_{H^1} \to 0$ and for all $j$, $\phi_j \in [-\pi,\pi]$  and 
\begin{equation} \label{cl:G_absurd} |g_j| \ge j \| g_j . w_* - w_* \|_{H^1}. 
\end{equation}
Up to a subsequence, we can further assume that $\phi_j \to \phi_\infty \in [-\pi,\pi]$ and $y_j \to y_\infty \in \m R$ (because otherwise, $\| \tau_{y_j} R_{\phi_\infty} w_* - w_* \|_{\dot H^1} \to 2 \| w_* \|_{\dot H^1} >0$ as $|y_j| \to +\infty$ which contradicts the assumption $\| g_j. w_* - w_* \|_{\dot H^1} \to 0$ as $j\to +\infty$). As $\| g_j . w_* - w_* \|_{H^1} \to 0$ and $g_j\to (y_\infty, \phi_\infty)$, it follows that $(y_\infty, \phi_\infty).w_* = w_*$; in particular, for the first component, we have $\tau_{y_\infty} \cos \theta_*=\cos \theta_*$ yielding $y_\infty=0$ (as $\cos \theta_*$ is not periodic in $\R$) and for the other component, $\cos(\phi_\infty+\f_*)=\cos \f_*$ and $\sin(\phi_\infty+\f_*)=\sin \f_*$ in $\R$ yielding $\phi_\infty =0$ (as $\phi_\infty \in [-\pi,\pi]$). Now as $g =(y,\phi) \to (0,0)$, by Taylor's expansion and \eqref{2star},
\[ g.w_* - w_* = - y \partial_x w_* + \phi e_1 \wedge w_* + O(|g|^2), \]
so that by Lemma \ref{lem:comput_w*},
\[ \|  g.w_* - w_* \|_{L^2} \ge c_\gamma |g| + O(|g|^2) \ge c_\gamma |g|/2 \quad \textrm{as} \, g\to (0,0). \]
This is a contradiction with \eqref{cl:G_absurd} (as $g_j\to (0,0)$). 

It remains to consider the case of continuous gauge functions $g, \tilde g:[0,T]\to G$. From the above, we dispose of a function $n(t) \in \m Z$ such that \eqref{ineg:beauty} holds at time $t$. By decreasing $\delta'$ if needed, we can assume that $\delta' C'\le \frac 1 5$. By uniform continuity of $g$ and $\tilde g$, there exists $\eta>0$ such that $|t-t'|<\eta$ implies 
$|g(t)-g(t')|, |\tilde g(t)-\tilde g(t')|<\frac1{5}$. Then for such $t,t'$,
\begin{align*}
|n(t)-n(t')|& \le  | g(t) -\big (\tilde g(t)+ (0,2\pi n(t))\big)|+ | g(t') -\big (\tilde g(t')+ (0,2\pi n(t'))\big)|\\
&\quad \quad +|\tilde g(t') -\tilde g(t)|+|g(t') - g(t)|<\frac4{5}.
\end{align*}
This shows that $n$ is locally constant, thus constant in $[0,T]$.
\end{proof}

\medskip

{\it Proof of \eqref{def:tilde_g}}. Using the above claim, provided that $3\delta^0_2 \le \delta'$ (this is a first choice of $\delta^0_2$), we can further change $\zeta_k$ (by adding $(0,2\pi n_k)$ for some well chosen $n_k \in \m Z$) so that for all $k=0, \dots, N-1$,
\[ |\zeta_k - \zeta_{k+1}| \le C' \| \zeta_k.w_* - \zeta_{k+1}.w_*  \|_{H^1} \le 3 C' \delta_2. \]
We now consider $g_1:[0,T]\to G$ affine on each segment $[t_k,t_{k+1}]$ such that $g_1(t_k) = \zeta_k$ for $k = 0,\dots, N$; then we choose some smooth function $g_0:[0,T]\to G$  such that $\| g_1 - g_0 \|_{\q C([0,T])} \le \delta_2$. For all $t \in [t_k,t_{k+1}]$, as $g_1$ is affine on this interval,
\[ \| g_1(t_k). w_* - g_1(t). w_* \|_{H^1} \le C \| w_* \|_{\q H^2} |g_1(t_k) - g_1(t)| \le  C \| w_* \|_{\q H^2} |g_1(t_k) - g_1(t_{k+1})| \le C_2'' \delta_2. \]
Then for any $t \in [0,T]$, denoting $k$ such that $t \in [t_k,t_{k+1}]$, one has
\begin{align*}
\| m(t) - g_0(t). w_* \|_{H^1} & \le \| m(t) - m(t_k) \|_{H^1} + \| m(t_k) - g_1(t_k) .w_* \|_{H^1}  \\
& \qquad + \|  g_1(t_k) .w_* - g_1(t). w_* \|_{H^1}  + \|  g_1(t) .w_* - g_0(t). w_* \|_{H^1} \\
& \le \delta_2 + \delta_2 + C_2'' \delta_2 +  C\delta_2 \le C_2' \delta_2,
\end{align*}
which proves \eqref{def:tilde_g}.

\medskip

We now go back to the construction of $g$ and $\e$ in \eqref{def:dec_m_e}. We can assume that $C_2' \delta^0_2 \le \delta_1$ with $\delta_1$ given in Lemma \ref{lem:mod} by possibly decreasing $\delta^0_2$. As for all $t \in [0,T]$, $\| (-g_0(t)). m(t) - w_* \|_{H^1} \le \delta_1$, Lemma \ref{lem:mod} ensures the existence of a constant $C_1>0$ and of a gauge $g_2(t) \in G$ and a map $\e(t) \in H^1$ such that \eqref{intoS2} and \eqref{eq:ortho} hold together with
\[  (-g_0(t)). m(t) = g_2(t). (w_* + \e(t)), \quad |g_2(t)| + \| \e(t) \|_{H^1} \le C_1 \| (-g_0(t)). m(t) -  w_* \|_{H^1} \le C_1C_2' \delta_2. \]
Then $g = g_0 + g_2$ is the desired gauge modulation, in particular, $\e$ satisfies the estimate in \eqref{est:g_e} for $C_2\ge C_1C_2'$. Moreover, recalling the smooth implicit function $\Psi$ constructed in the proof of Lemma~\ref{lem:mod}, we know that 
\[ g_2(t)=\Psi\big((-g_0(t)). m(t)\big) \]
is continuous in $t$ (as $t\mapsto (-g_0(t)). m(t)$ belongs to $\q C([0,T],\q H^2)$). Thus, $g$ is continuous in $t$ and $\e = (-g). m -w_* \in \q C([0,T],H^2(\m R, \m R^3))$ (as mentioned in Remark \ref{rem:23}). 

\medskip

It remains to prove that $g$ and $\e$ are Lipschitz in $t$ and $g$ satisfies the first bound in \eqref{est:g_e}. However, even if \eqref{ll} yields $\partial_t m \in L^\infty([0,T],L^2)$, this is not enough  \emph{a priori} to infer that $\dot g \in L^\infty$ and $\partial_t \e \in L^\infty([0,T],L^2)$ (because the implicit function $\Psi$ in the proof of Lemma~\ref{lem:mod} is smooth for the $H^1$ topology, but smoothness in the $L^2$ topology would require additional arguments). This is nevertheless the case if one furthermore assume more regularity on $m$, i.e., $m \in  \q C([0,T],\q H^3)$ is a solution to \eqref{ll} in which case one sees by inspection that $\delta E(m) \in \q C([0,T],H^1)$ and so $\partial_t m \in L^\infty([0,T],H^1)$. This means that $t \mapsto m(t)$ is Lipschitz with values in $\q M$ and from there, as $\Psi$ is Lipschitz on $\q M$, $g$ is also a Lipschitz function: $\dot g \in L^\infty([0,T])$. It follows that $\partial_t \e \in L^\infty([0,T],H^1)$. This allows to perform computations of the equations for $g$ and $\e$, and to obtain their Lipschitz regularity together with the first estimate in \eqref{est:g_e} in Steps 2-3 below. As these bounds will not depend on the norm of $\| m \|_{\q C([0,T],\q H^3)}$, a classical approximation argument in Step~4 will show that the bounds \eqref{est:g_e} still hold under the assumption $m \in  \q C([0,T],\q H^2)$. 

\medskip

For the next steps 2 and 3, we therefore assume that $m \in  \q C([0,T],\q H^3)$.

\medskip

\emph{Step 2. Preliminary equation on $\partial_t \e$}. To prove the first statement of \eqref{est:g_e}, we start by estimating the terms in $\e$ in \eqref{ll}. We emphasise that all the equalities below are between functions in $L^\infty([0,T],L^2)$. With the notation $g=g(t)=(y(t), \phi(t))$, we have by the formulas in Appendix \ref{sec:tool}, the linearity of $\delta E(\cdot)$ and  \eqref{foot56}: 
\begin{align} 
\label{eq:dt_m_decomp} \partial_t m & = \partial_t (g. (w_* + \e)) = - \dot y g. (\partial_x w_* + \partial_x \e) + \dot \phi  e_1 \wedge g.(w_* +\e) + g . \partial_t \e, \\
\nonumber H(m) & = - \delta E(g.w_*) - \delta E(g.\e) + h e_1=  - g. \delta E(w_*)  - g. \delta E(\e) + h e_1.
\end{align}
We now expand at order $0$ in $\e$ (that is, terms which are at least linear in $\e$ are treated in $O$), tracking only the dependence in $h$ (so that the implicit polynomial in the $O$ notation below do not depend on $h$). 
Observe that $\delta E(\e) = O^1_2(\e)$, so that as $w_*$ is smooth and $w_*\wedge \delta E(w_*)=0$, Claim \ref{cl:Okl_1} yields 
\begin{align}
\nonumber m \wedge H(m) & =g. \left( (w_* + \e) \wedge ( - \delta E(w_*)  + O^1_2(\e) + h e_1) \right) \\
\nonumber & = g. \big(h (w_*+\e) \wedge e_1 + O^1_0(\e) + O^1_2(\e) +  O^1_0(\e)  O^1_2(\e)\big) \\
\label{use34} & = g. \big(h w_* \wedge e_1+hO^1_0(\e) + O^1_2(\e) + O^2_2(\e)\big), \\
\nonumber m \wedge (m \wedge H(m)) & = g. \left( (w_* + \e) \wedge (h  w_* \wedge e_1 +hO^1_0(\e)+ O^1_2(\e) + O^2_2(\e)) \right) \\
 & = g.\big(h w_*  \wedge (  w_* \wedge e_1) + h(O^1_0(\e)+O^2_0(\e))+O^1_2(\e) + O^2_2(\e) + O^3_2(\e)\big). 
\label{use35} 
\end{align}
Therefore, by \eqref{eq:dt_m_decomp}, \eqref{use34} and \eqref{use35}, the equation of $\e$ at order $0$ is:
\begin{align}
\nonumber \partial_t \e & = (-g). \partial_t m +  \dot y (\partial_x w_* + \partial_x \e) - \dot \phi  e_1 \wedge (w_* +\e) \\
\nonumber & = (-g). \big(m \wedge H(m) - \alpha m \wedge (m \wedge H(m))\big) +  \dot y (\partial_x w_* + \partial_x \e) - \dot \phi  e_1 \wedge (w_* +\e) \\
\nonumber & = h w_* \wedge e_1  - \alpha h w_*  \wedge (  w_* \wedge e_1) +  \dot y \partial_x w_* -  \dot \phi  e_1 \wedge w_* \\
\nonumber & \qquad  + \dot g \cdot O^1_1(\e) +h(O^1_0(\e)+O^2_0(\e))+ O^1_2(\e) + O^2_2(\e) + O^3_2(\e)\\
\label{eq:e_order_0} & = (\dot y-\dot y_*) \partial_x w_* -  (\dot \phi-\dot \phi_*)  e_1 \wedge w_* \\
\nonumber & \qquad+ \dot g \cdot O^1_1(\e) +h(O^1_0(\e)+O^2_0(\e))+ O^1_2(\e) + O^2_2(\e) + O^3_2(\e),
\end{align}
where we used \eqref{eq-funda}.

\bigskip

\emph{Step 3. Equation on $\dot g$}. We use here the orthogonality conditions \eqref{eq:ortho}: differentiating them with respect to time, as $\partial_t \e \in L^\infty([0,T], H^1)$, we get
\begin{align} \label{eq:d_ortho}
\int \partial_t \e \cdot \partial_x w_*\, dx = 0 \quad \text{and}  \int \partial_t \e \cdot (e_1 \wedge w_*) \, dx =0. 
\end{align}
Combined with Lemma \ref{lem:comput_w*}, we obtain the following equalities between functions (of time) in $L^\infty([0,T])$: using \eqref{eq:e_order_0}, the first relation in \eqref{eq:d_ortho} gives
\begin{align*}
 \frac{2}{\sqrt{1-\gamma^2}}   (\dot y-\dot y_*) +  \frac{2\gamma}{\sqrt{1-\gamma^2}} (\dot \phi-\dot \phi_*)  = O( |\dot g| + 1 +  \| \e \|_{H^1}^2) \| \e \|_{H^1}+ O(|h|( \| \e \|_{H^1}+ \| \e \|_{H^1}^2)),
\end{align*}
where we used Claims \ref{cl:Okl_1}, \ref{cl:Okl_2} and \ref{cl:O_IPP} (and that $\partial_x w_*$ is smooth and localised, and so belongs to $H^1$). As $\dot g_*=O(|h|)$ by \eqref{eq_y_phi}, this rewrites as
\[  (\dot y- \dot y_*)+ \gamma (\dot \phi-\dot \phi_*)  +O( \| \e \|_{H^1}) |\dot g-\dot g_*|  =   O( 1 +  \| \e \|_{H^1}^2) \| \e \|_{H^1}(1+ |h|). \]
Similarly, the second relation in \eqref{eq:d_ortho} together with \eqref{eq:e_order_0} yields
\begin{align*}
 \frac{2\gamma}{\sqrt{1-\gamma^2}} ( \dot y-\dot y_*) + \frac{2}{\sqrt{1-\gamma^2}} (\dot \phi-\dot \phi_*)= O( |\dot g| + 1 +  \| \e \|_{H^1}^2) \| \e \|_{H^1}+O(|h|( \| \e \|_{H^1}+ \| \e \|_{H^1}^2));
\end{align*}
furthermore, 
\[  \gamma  (\dot y-\dot y_*) + (\dot \phi-\dot \phi_*) +O( \| \e \|_{H^1}) |\dot g-\dot g_*|=  O( 1 +  \| \e \|_{H^1}^2) \| \e \|_{H^1}(1+ |h|). \] 
As $ \| \e \|_{H^1} \le C \delta_2 \le 1$ (by possibly decreasing $\delta^0_2$), we equivalently get
\[ M_\e (\dot g-\dot g_*) =  O((1+ |h|)\| \e \|_{H^1}) \]
where $M_\e = \begin{pmatrix}
1 & \gamma \\
\gamma & 1
\end{pmatrix} + O(\| \e \|_{H^1})$ is invertible as soon as $\| \e \|_{H^1} \le C_2 \delta_2$ is smaller than a certain constant depending only on $\gamma$, with inverse
\[ M_\e^{-1} = \frac{1}{1-\gamma^2} \begin{pmatrix}
1 & -\gamma \\
- \gamma & 1
\end{pmatrix} + O(\| \e \|_{H^1}). \]
We therefore get $\dot g=\dot g_*+O((1+|h|) \| \e \|_{H^1})$ which is the $\dot g$ estimate in \eqref{est:g_e}. The proof is complete in the case when $m \in \q C([0,T],\q H^3)$.  

\bigskip

\emph{Step 4. Limiting argument}. We treat now the general case $m \in \q C([0,T],\q H^2)$ satisfying \eqref{est:m_w_delta}. There exists a sequence $(m^{n}_{0})_n$ in $\q H^3$ such that $m^{n}_{0} \to m(0)$ in $\q H^2$. Denote $m^{n}$ the solution to \eqref{ll} given by Theorem \ref{th:lwp} with initial data $m^{n}(0) =m^{n}_{0}$. By the continuity of the flow in Theorem \ref{th:lwp}, $m^{n}$ is defined on $[0,T]$ at least for large $n$, and
\be
\label{67} 
\| m^{n} - m \|_{\q C([0,T], \q H^2)} \to 0 \quad \text{as} \quad n \to +\infty. 
\ee
In particular, for large enough $n$, $m^{n}$ also satisfies  \eqref{est:m_w_delta} (recall Proposition \ref{prop:conv_H_H1}: since the first component $m_1(t)$ is not constant by  \eqref{est:m_w_delta}, then $\| m^{n}(t) - m(t) \|_{H^1}$ is small provided $\| m^{n}(t) - m(t) \|_{\q H^1}$ is small). We can hence apply Steps 1, 2 and 3 to $m^{n}$: for large enough $n$, there exists a  decomposition $m^{n} = g^{n}.(w_* + \e^{n})$ with the Lipschitz gauge function $g^n=(y^n, \phi^n):[0,T]\to G$ such that 
\be 
\label{88}\forall t \in [0,T], \quad | \dot g^{n}(t) - \dot g_*(t) | \le  C_2(1+ |h(t)|) \| \e^{n} (t) \|_{H^1}, \quad 
\| \e^{n}(t) \|_{H^1} \le C_2 \delta_2. 
\ee
We can assume without loss of generality that for all $n$, $\phi^n(0) \in [-\pi,\pi]$. 
 As $m^{n}(0) \to m(0)$ in $H^1$ and $\| \e^n(0) \|_{H^1}\le C_2\delta_2$, we have for large enough $N$ and $n \ge N$, $\|g^n(0).w_* - g^N(0).w_*\|_{H^1}\le 4C_2 \delta_2$. Using Claim \ref{claim:beauty} (up to further decreasing $\delta_2^0$), we infer $|g^n(0) - g^N(0)| \le C_2'' \delta_2$ (as $\phi^n(0), \phi^N(0)  \in [-\pi,\pi]$). In particular, $(g^n(0))_n$ is bounded, and together with \eqref{88}, we  conclude that $(g^n)_n$ is bounded in $W^{1, \infty}([0,T])$. Hence, up to a subsequence, Ascoli's theorem yields that $g^n \to g$ in $\q C([0,T])$ and $\dot g^n\to \dot g$ weakly-$*$ in $L^\infty([0,T])$ to a Lipschitz limit gauge function $g\in W^{1, \infty}([0,T], G)$. This leads by \eqref{67} to 
\[ \e^n = (-g^n).m^n - w_* \to (-g).m - w_*=: \e \quad \text{in} \quad \q C([0,T],H^2). \]
Furthermore, as $\dot g^n - \dot g_* = O((1+ |h|)\| \e^n \|_{H^1})$, we infer that $\dot g - \dot g_* = 
O((1+ |h|)\| \e \|_{H^1})$; also, the properties of $\e^n$ in \eqref{intoS2} and the orthogonality conditions \eqref{eq:ortho} for every $t\in [0,T]$ transfer to the limit $\e$ as $n\to \infty$ as well as the estimate $\|\e(t)\|_{H^1}\le C_2 \delta_2$. Finally, by inspection of equation \eqref{ll}, we see that \eqref{67} yields $\partial_t m^n \to \partial_t m$ in $L^\infty([0,T],L^2)$, and  in the relation \eqref{eq:dt_m_decomp} 
\[ \partial_t \e^n = (-g^n). \partial_t m^n + \dot y^n (\partial_x w_* + \partial_x \e^n) - \dot \phi^n  e_1 \wedge (w_* +\e^n), \]
all terms of the right hand side have strong or weak-$*$ limits in $L^\infty([0,T],L^2)$, so that $\partial_t \e^n$ has a weak-$*$ limit in in $L^\infty([0,T],L^2)$, as well, which is necessarily $\partial_t \e$ (by uniqueness of limits in the distributional sense). Therefore, we conclude $\partial_t \e \in L^\infty([0,T],L^2)$.

\bigskip

\emph{Step 5. Uniqueness}. Let $\tilde g = (\tilde y, \tilde \phi)$ be another such gauge (associated with error $\tilde \e$) for $m$. By \eqref{def:dec_m_e}, for all $t \in [0,T]$, $\| g(t). w_* - \tilde g(t). w_* \|_{H^1} \le \| \e (t) \|_{H^1} + \| \tilde \e(t) \|_{H^1} \le 2C_2 \delta_2 \le \delta'$, so that Claim \ref{claim:beauty} applies and there exists 
 $n \in \m Z$ (independent of $t$) such that $|g(t) - \tilde g(t) + (0,2\pi n)| \le 2C'C_2 \delta_2$.
By lowering $\delta_2^0$ further, \eqref{est:g_e} applied to $\e(t)$ and $\tilde \e(t)$ yields $(-g(t)). m(t)$ and $(-\tilde g(t) - (0,2\pi n)).m(t)$ fit the hypothesis for $w$ in Lemma \ref{lem:mod}, and in view of the orthogonality condition \eqref{eq:ortho} satisfied by both $\e(t)$ and $\tilde \e(t)$, the uniqueness of the gauge in Lemma \ref{lem:mod} yields for all $t \in [0,T]$, $g(t) = \tilde g(t) + (0,2\pi n)$ and so $\e(t) = \tilde \e(t)$.
\end{proof}

\subsection{Coercivity of the energy}

The second ingredient in the proof of Theorem \ref{th:stab} is a coercivity property of the energy and estimates on the energy dissipation \eqref{eq:en_dissip2}.
For the coercivity property, we expand the energy around $w_*\in \{w_*^\pm\}$. To express it more conveniently, we work in the $(w_*,n_*,p_*)$ basis related to $w_*$ given in \eqref{np}. We can actually give the coordinates explicitly:
\begin{equation} \label{eq:w*_spheric}
w_* =  \begin{pmatrix}
\cos \theta_*\\
\sin \theta_* \cos \varphi_* \\
\sin \theta_* \sin \varphi_*
\end{pmatrix},
\quad n_* =  \begin{pmatrix}
-\sin \theta_*\\
\cos \theta_* \cos \varphi_* \\
\cos \theta_* \sin \varphi_*
\end{pmatrix}, \quad
p_* =  \begin{pmatrix}
0 \\
- \sin \varphi_* \\
\cos \varphi_*
\end{pmatrix},
\end{equation}
with $\theta_*$ and $\varphi_*$ given in \eqref{formul1}. 
Observe that $(w_*,n_*,p_*)$ is a direct orthonormal frame, and we compute their differentials:  
\begin{align*}
d w_* & = n_* d\theta_* + \sin \theta_* p_* d\varphi_*, \\
dn_* & = - w_* d\theta_* + \cos \theta_* p_* d\varphi_*, \\
d p_* & =  - (\sin \theta_* w_* + \cos \theta_* n_*) d\varphi_*.
\end{align*}
Also, by \eqref{formul1}, 
\be
\label{numbe}
\begin{cases}
& e_1 = \cos \theta_* w_* - \sin \theta_* n_*,\\
& e_1 \wedge w_* = \sin \theta_* p_*, \, e_1 \wedge n_* = \cos \theta_* p_*, \, e_1 \wedge p_* = - (\sin \theta_* w_* + \cos \theta_* n_*),\\
& \partial_x w_*  = \partial_x \theta_* n_* + \sin \theta_* \partial_x \varphi_* p_* = \partial_x \theta_* \left( n_* + \frac{\gamma}{\sqrt{1-\gamma^2}} p_* \right), \\
& \partial_x n_* = - \partial_x \theta_* w_* + \cos \theta_* \partial_x \varphi_* p_* = - \partial_x \theta_* w_* - \gamma \cos \theta_* p_*, \\
& \partial_x p_*  = - \partial_x \varphi_* (\sin \theta_* w_* + \cos \theta_* n_*) = \gamma (\sin \theta_* w_* + \cos \theta_* n_*) .
\end{cases}
\ee
We also use the notation (see \eqref{rho*}):
\be
\label{dstar}
\begin{cases}
&\beta_* := \rho_0^*=w_* \cdot \delta E(w_*) =  (\partial_x \theta_*)^2 + (1-\gamma^2)\sin^2 \theta_*    = 2(1-\gamma^2) \sin^2 \theta_* >0,\\ 
&\delta E(w_*) = \beta_* w_*,
\end{cases}
\ee
by Theorem \ref{thm_DM}.

We will need the following operator 
\begin{align} \label{def:Lg}
L_\gamma  = - \partial_{xx}  + (1-\gamma^2)(\cos^2 \theta_*- \sin^2 \theta_*). 
\end{align}
Actually, $L_\gamma$ is a Schr\"odinger operator with classical potential ($V(x)=(1-\gamma^2)(\cos^2 \theta_*- \sin^2 \theta_*)$), and so its spectrum is well known. We summarise below the properties which will be relevant for the subsequent analysis following ideas in \cite[Proposition 2.10]{Wei86} (with $\sigma=1$, $N=1$), we also refer to \cite[Lemma 2.2]{chang2008spectra} and Lemma \ref{lem:coer00} in Appendix \ref{append:c}. 

\begin{lem} \label{lem:Lg} For $\gamma\in (-1,1)$, 
$L_\gamma$ is a self-adjoint operator on $L^2(\m R)$ with dense domain $H^2(\m R)$ and has $0$ as first (simple) eigenvalue with an eigenfunction $\sin(\theta_*)>0$:
\[ L_\gamma (\sin \theta_*) =0, \]
As a consequence, there exists $\lambda_0>0$ (small) such that for all $v \in H^1(\m R,\m R)$,
\begin{align} \label{coercivity} 
 0 \le \ds (L_\gamma v,v) \le \frac{1}{\lambda_0} \| v \|_{H^1}^2\quad \textrm{and} \quad (L_\gamma v, v) \ge 4\lambda_0 \| v \|_{H^1}^2 - \frac{1}{\lambda_0} \left( \int v \sin(\theta_*) dx \right)^2,
\end{align}
and for all $v \in H^2(\m R,\m R)$,
\begin{align} \label{coercivity2} 
\| L_\gamma v \|_{L^2}^2 \ge 4 \lambda_0 \| v \|_{H^2}^2 - \frac{1}{\lambda_0} \left( \int v \sin(\theta_*) dx \right)^2.
\end{align}
\end{lem}

\begin{proof}
We provide a short explanation of the results. By direct computations, 
\[ \partial_{xx} \theta_* = - \sqrt{1-\gamma^2} \cos \theta_* \partial_x \theta_* = (1-\gamma^2) \sin (2 \theta_*)/2, \]
so that
\[ \partial_{xx} (\partial_x \theta_*) = (1-\gamma^2) \cos(2 \theta_*) \partial_x \theta_* = (1-\gamma^2)(\cos^2 \theta_* - \sin^2 \theta_*) \partial_x \theta_* \]
Hence $\partial_x \theta_* \in \ker L_\gamma$ and so does $\sin \theta_*$ which is collinear to $\partial_x \theta_*$ by \eqref{formul1}. 
Now, note that $\sin \theta_*>0$, thus, Sturm-Liouville theory ensures that $L_\gamma$ is a nonnegative operator and the first eigenvalue, which is $0$, is simple, ie,  $\ker L_\gamma = \R \sin \theta_*$, see details in Lemma \ref{lem:coer00} in Appendix \ref{append:c} where also \eqref{coercivity} and \eqref{coercivity2} are proved.
\end{proof}

\begin{nb}
Notice that $L_\gamma$ has continuum spectrum $[1-\gamma^2,+\infty)$: indeed, the resolvent of $L_\gamma = -\partial_{xx} + (1-\gamma^2)(1-2\sin^2 \theta_*)$ is a compact perturbation of that of $-\partial_{xx} + (1-\gamma^2)$, so they share the same continuous spectrum. We refer for further details to \cite[Chapters 13 and 14]{HS96}, for example.
\end{nb}

The heart of the stability result is the following statement, which gives (at leading order) the relative size of the terms appearing in the energy dissipation identity \eqref{eq:en_dissip2}. It is an expansion at the stationary level.

\begin{prop}[Expansion of the energy] \label{prop:en_expansion} Let $w_*\in \{w_*^\pm\}$. 
There exist $\delta_3>0$ and $C_3 >0$ such that the following holds. Let $m:=w_* + \eta: \m R \to \m S^2$ be such that 
\[ \| \eta \|_{\q H^1} \le \delta_3. \]
We decompose $\eta$ in the $(w_*,n_*,p_*)$ basis (pointwise in $x$):
\[  \eta = \mu w_* + \nu n_* + \rho p_* \quad \text{where} \quad \mu: = \eta \cdot w_*, \quad \nu: = \eta \cdot n_*, \quad \rho : = \eta \cdot p_*. \]
Then $\mu, \nu, \rho \in H^1$ with
\[ \| \mu \|_{H^1} \le C_3 \| (\nu,\rho) \|_{H^1}^2, \quad   \frac{1}{C_3} \| \eta \|_{H^1} \le \| (\nu,\rho) \|_{H^1} \le C_3 \| \eta \|_{H^1}, \]
\begin{align} \label{E_coer}
&\left| E(m) - \big[ E(w_*) + \frac{1}{2} \left( (L_\gamma \nu,\nu) + (L_\gamma \rho,\rho) \right) \big] \right| \le C_3 \| (\rho,\nu) \|_{H^1}^3,\\
&\left| \int  (m \wedge e_1) \cdot (m \wedge \delta E(m) dx \right|  \le C_3  \| (\nu,\rho) \|_{H^1}^2. \label{E_Efield}
\end{align}
with $L_\gamma$ defined in \eqref{def:Lg}. 

If furthermore $\eta \in \q H^2$ (not necessarily small in $\q H^2$), then $\eta, \mu, \nu, \rho \in H^2$, 
\begin{align}
\nonumber
\| (\nu,\rho) \|_{H^2}&  \le C_3 \| \eta \|_{H^2}, \\
\left| \int \left( |\delta E(m)|^2- |m \cdot \delta E(m)|^2 \right) dx - \left[ \| L_\gamma \nu \|_{L^2}^2 + \| L_\gamma \rho \|_{L^2}^2 \right] \right| & \le C_3 \| (\nu,\rho) \|_{H^2}^2  \| (\nu,\rho) \|_{H^1},  \label{E_diss}
\end{align}
\end{prop}

\begin{proof}
We will use the notation $g = O^\ell_k(f)$ introduced in Definition \ref{def:O}.  In view of the expansion of the energy, we will expand to cubic order in $\eta$ (and so the computations are substantially more involved than in Proposition \ref{lem:decomp}). Recall that  $w_*$, $n_*$, $p_*$ are smooth functions of $\theta_*$ and $\f_*$ and together with their derivatives at any order, they belong to $W^{2, \infty}(\R)$, so they are admissible as coefficient functions in the polynomials $O^\ell_k$. This allows to differentiate the symbol  $O^\ell_k$ as in Claim \ref{cl:Okl_1} (iv).

By Lemma \ref{lem:H_H1_2} in Appendix \ref{append}, we denote $\delta_*$ and $C_*$ the constants given there for $w_*$ (which satisfies the nondegeneracy condition). Taking $ \delta_3 \in (0, \delta_*)$ small enough, we can assume that $\| \eta \|_{H^1} \le C_* \delta_3$ is as small as needed. Moreover, by a density argument, we can assume that $\eta$ is smooth. \label{pag} Indeed, by convolution, there is a sequence of smooth maps 
$\tilde \eta^n:\R\to \R^3$ such that $\tilde \eta^n\to \eta$ in $H^1$. By the embedding $H^1\subset L^\infty$, we deduce that $|w_*+\tilde \eta^n|\to |w_*+\eta|=1$ uniformly in $\R$. Therefore, for $n$ large, the smooth maps $\eta_n= \frac{w_*+\tilde \eta^n}{|w_*+\tilde \eta^n|}-w_*$ converge to $\eta$ in $H^1$ and $|w_*+\eta^n|=1$ in $\R$. Throughout the proof below, we will assume that $\eta$ is smooth, and the general case will follow using the above density argument.

\bigskip

\emph{Step 1: $\mu$ is quadratic in $(\nu,\rho)$.} The relations $\mu = \eta \cdot w_*$, $\nu = \eta \cdot n_*$, $\rho = \eta \cdot p_*$ yield that
\[ \mu, \nu, \rho = O^1_0(\eta). \]
Differentiating repetitively (and using Claim \ref{cl:Okl_1} (iv)), we get that for any $k \in \m N$,
\[ \partial_x^k \mu, \partial_x^k \nu, \partial_x^k \rho = O_k^1(\eta). \]
Using Claim \ref{cl:Okl_2}, if $\eta \in H^k$ for some $k \ge 1$, then
\begin{align}  \label{est:coord_eta_Hk}
\| \mu \|_{H^k} +  \| \nu \|_{H^k} +  \| \rho \|_{H^k} \lesssim \| \eta \|_{H^k}.
\end{align}
In particular, for $k=1$,
\[ \| \mu \|_{H^1} +  \| \nu \|_{H^1} +  \| \rho \|_{H^1} \lesssim \| \eta \|_{H^1}, \]
so that these quantities are small if $\| \eta \|_{H^1} $ is small. On the other side, as $\eta = \mu w_* + \nu n_* + \rho p_* = O^1_0(\mu,\nu,\rho)$, we have similarly that for all $k \in \m N$,
\[ \partial_x^k \eta = O_k^1(\mu,\nu,\rho). \]
As a consequence, by Claim \ref{cl:Okl_2}, there holds for $k\ge 1$:
\begin{align} \label{est:eta_coord_Hk}
\| \eta \|_{H^k} \lesssim \| \mu \|_{H^k} +  \| \nu \|_{H^k} +  \| \rho \|_{H^k}.
\end{align}
Let us now observe that  $\mu$ is quadratic in $\eta$, implying  that all terms containing $\mu$ are of order at least $2$ in $\eta$. Indeed, the constraint $|w_* + \eta|^2 =1$ writes
\[ \mu = - |\eta|^2/2 =  O_0^2(\eta). \]
Thus, 
\begin{align}
\label{eq:mu_eta4}
\mu = - \frac{1}{2} (\mu^2 + \nu^2 + \rho^2) = - \frac{1}{2} (\nu^2 + \rho^2) + O^4_0(\eta).
\end{align}
Hence, using again Claim \ref{cl:Okl_1} (iv), for all $k \in \m N$,
\[ \partial_x^k \mu = O_k^2(\eta), \]
and so, Claim \ref{cl:Okl_2} (with $k=1$ or $2$) yields 
\[ \| \mu \|_{H^1} \lesssim \| \eta \|_{H^1}^2, \quad \| \mu \|_{H^2} \lesssim \| \eta \|_{H^1} \| \eta \|_{H^2}. \]
As $\| \eta \|_{H^1}$ is small, we infer from \eqref{est:coord_eta_Hk} and \eqref{est:eta_coord_Hk} that for $k=1$ or $2$,
\begin{align}
\label{est:nu_rho_eta}
\| (\nu , \rho) \|_{H^k} \lesssim \| \eta \|_{H^k} \lesssim \| (\nu, \rho) \|_{H^k}.
\end{align}
In particular, we  also have
\[ \| \mu \|_{H^1} \lesssim \| (\nu,\rho) \|_{H^1}^2, \quad \| \mu \|_{H^2} \lesssim \| (\nu,\rho) \|_{H^1} \| (\nu,\rho) \|_{H^2}. \]

\bigskip

\emph{Step 2. Expansion of $\delta E(\eta)$ and related linear terms in $\eta$}. This is a preliminary step where we develop $\delta E(\eta)$ and other linear terms which appear in the integrals when expanding $E(m)$. 
By Claim \ref{cl:Okl_1} combined with \eqref{formul1} and \eqref{numbe}, we compute:
\begin{align*}
\partial_x \eta & = (\partial_x \mu - \partial_x \theta_* \nu + \gamma \sin \theta_* \rho) w_*  + (\partial_x \theta_* \mu + \partial_x \nu + \gamma \cos \theta_* \rho) n_* \\
& \quad + \left( \frac{\gamma}{\sqrt{1-\gamma^2}} \partial_x \theta_* \mu - \gamma \cos \theta_* \nu + \partial_x \rho \right) p_* \\
& = O_1^2(\eta)  
+ ( - \partial_x \theta_* \nu + \gamma \sin \theta_* \rho) w_* + ( \partial_x \nu + \gamma \cos \theta_* \rho) n_* + \left( - \gamma \cos \theta_* \nu + \partial_x \rho \right) p_*, \\
e_1 \wedge \eta & = ( \sin \theta_* \mu + \cos \theta_* \nu ) p_* - \sin \theta_* \rho w_* - \cos \theta_* \rho n_* \\
& =  O_0^2(\eta) - \sin \theta_* \rho w_* - \cos \theta_* \rho n_*+ \cos \theta_* \nu p_*, \\
\eta_1 & = \cos \theta_* \mu - \sin \theta_* \nu = O_0^2(\eta) - \sin \theta_* \nu, \\
\eta_1 e_1 & = O_0^2(\eta) - \sin \theta_*  \cos \theta_* \nu w_* + \sin^2 \theta_* \nu n_*, \\
\eta-\eta_1 e_1& = O_0^2(\eta)+\sin \theta_*  \cos \theta_* \nu w_*+ \cos^2 \theta_* \nu n_*+\rho p_*,\\
e_1 \wedge \partial_x \eta & =  O_1^2(\eta) + ( - \partial_x \theta_* \nu + \gamma \sin \theta_* \rho) \sin \theta_* p_* + ( \partial_x \nu + \gamma \cos \theta_* \rho) \cos \theta_* p_* \\
& \quad - \left( - \gamma \cos \theta_* \nu + \partial_x \rho \right) (\sin \theta_* w_* + \cos \theta_* n_*) \\
& =  O_1^2(\eta) + ( \gamma \cos \theta_* \sin \theta_* \nu - \sin \theta_* \partial_x \rho) w_* \\
& \quad + ( \gamma \cos^2 \theta_* \nu - \cos \theta_* \partial_x \rho) n_* + ( \partial_x ( \cos \theta_* \nu) + \gamma \rho ) p_*, \\
\partial_{xx} \eta & = O_2^2(\eta) + ( - \partial_{xx} \theta_* \nu - \partial_x \theta_* \partial_x \nu + \gamma \cos \theta_* \partial_x \theta_* \rho + \gamma \sin \theta_* \partial_x \rho) w_* \\
& \quad + \partial_x \theta_* ( - \partial_x \theta_* \nu + \gamma \sin \theta_* \rho) \left( n_* + \frac{\gamma}{\sqrt{1-\gamma^2}} p_* \right) \\
&  \quad +  ( \partial_{xx} \nu - \gamma \sin \theta_* \partial_x \theta_* \rho + \gamma \cos \theta_* \partial_x \rho) n_* \\
& \quad -  ( \partial_x \nu + \gamma \cos \theta_* \rho) (\partial_x \theta_* w_* + \gamma \cos \theta_* p_*)  \\
& \quad + ( \gamma \sin \theta_* \partial_x \theta_* \nu - \gamma \cos \theta_* \partial_x \nu + \partial_{xx} \rho ) p_* \\
& \quad + \gamma \left( - \gamma \cos \theta_* \nu + \partial_x \rho \right)  (\sin \theta_* w_* + \cos \theta_* n_*) \\
& = O_2^2(\eta) + w_* \left( - \partial_{xx} \theta_* \nu - 2\partial_x \theta_* \partial_x \nu - \gamma^2 \sin \theta_* \cos \theta_* \nu + 2 \gamma \sin \theta_* \partial_x \rho\right) \\ 
& \quad  +  \big(\partial_{xx} \nu -  [(\partial_x \theta_*)^2 + \gamma^2 \cos^2 \theta_* ] \nu + 2 \gamma \cos \theta_* \partial_x \rho \big) n_* \\
& \quad  +  \big(-  2 \gamma \partial_x(\cos \theta_* \nu) - \gamma^2 \rho + \partial_{xx} \rho \big) p_*.
\end{align*}
We obviously have
\begin{equation} \label{eq:dE(e)_1}
 \delta E(\eta) = -  \partial_{xx} \eta - 2 \gamma e_1 \wedge \partial_x \eta + \eta-\eta_1 e_1 = O_2^1(\eta). \end{equation}
Using \eqref{crit_theta} and \eqref{formul1}, we have an expansion at order 2 for this quantity
\begin{align}
\label{TT} \delta E(\eta) & = O_2^2(\eta) + 2 (\partial_{xx} \theta_* \nu + \partial_x \theta_* \partial_x \nu) w_* +  \left(- \partial_{xx} \nu+ (1 - \gamma^2) \nu \right)  n_* + (- \partial_{xx} \rho+ (1 - \gamma^2) \rho) p_*.
\end{align}
Hence, we have the following expansions for scalar and wedge products linear in $\eta$:
\begin{align*}
w_* \cdot \eta & = \mu=  - \frac{1}{2} ( \nu^2 + \rho^2) + O_0^4(\eta),\\
w_* \cdot \delta E(\eta) & = 2 \partial_{xx} \theta_* \nu + 2 \partial_x \theta_* \partial_x \nu + O_2^2(\eta)=O_1^1(\eta)+O_2^2(\eta), \\
 w_* \wedge \delta E(\eta) &  = O_2^2(\eta) - (- \partial_{xx} \rho+ (1 - \gamma^2) \rho) n_*+ \left(- \partial_{xx} \nu+ (1 - \gamma^2) \nu \right)  p_*, \\
\eta \cdot \delta E(w_*) & = \beta_* \eta \cdot w_* = - \frac{1}{2} \beta_* |\eta|^2,
\end{align*}
where $\beta_*$ is given in \eqref{dstar}.

\bigskip

\emph{Step 3. Expansion of quadratic terms in $\eta$}. Using Claim \ref{cl:Okl_1} together with \eqref{eq:mu_eta4} and \eqref{TT}, we derive the following expansions for the quadratic terms in $\eta$:
\begin{align*}
|\eta|^2 & = O_0^4(\eta) + \nu^2 + \rho^2, \quad  \\
(w_* \cdot \delta E(\eta))^2 & =  O^3_3(\eta) + O_4^4(\eta) + \left( 2 \partial_{xx} \theta_* \nu + 2 \partial_x \theta_* \partial_x \nu \right)^2
\end{align*}
Using \eqref{eq:dE(e)_1}, we have
\[ \eta \cdot \delta E(\eta) = O_2^2(\eta). \]
With the help of \eqref{TT}, we also have
\begin{align*}
\eta \cdot \delta E(\eta) & = O_2^3(\eta) - \nu \partial_{xx} \nu + (1 - \gamma^2) \nu^2 - \rho \partial_{xx} \rho + (1 - \gamma^2) \rho^2, \\
|\delta E(\eta)|^2 & = O_4^3(\eta) + O_4^4(\eta) + \left(2 \partial_{xx} \theta_* \nu + 2\partial_{x} \theta_* \partial_x \nu \right)^2 + \left(- \partial_{xx} \nu + (1 - \gamma^2) \nu \right)^2 \\& + (- \partial_{xx} \rho + (1 - \gamma^2)  \rho)^2.
\end{align*}

\bigskip

\emph{Step 4. Expansion of $E(m)$. Proof of \eqref{E_coer}.} The energy $E$ consists only in quadratic terms: therefore, $E(\eta)=\frac12 \int \eta \delta E(\eta) dx$ and we compute using Claim \ref{cl:O_IPP}:
\begin{align*}
\MoveEqLeft  E(w_* + \eta) - E(w_*) =\frac{d}{ds}\big|_{s=0} E(w_* + s\eta)+\frac12 \frac{d^2}{ds^2}\big|_{s=0} E(w_* + s\eta)\\
  &= \int \eta \cdot \delta E(w_*) dx + \frac{1}{2} \int \eta \cdot \delta E(\eta) dx \\
& = O(\| \eta \|_{H^1}^3) - \frac{1}{2} \int \beta_* |\eta|^2 dx + \frac{1}{2} \int \left\{ (-\partial_{xx} \nu + (1 - \gamma^2)  \nu) \nu + (-\partial_{xx} \rho +  (1 - \gamma^2) \rho) \rho \right\} dx \\
& = O(\| \eta \|_{H^1}^3) + \frac{1}{2} \int \left\{  (\partial_x \nu)^2 + (1-\gamma^2 - \beta_*) \nu^2 +  (\partial_x \rho)^2 + (1-\gamma^2 - \beta_*)\rho^2 \right\} dx,
\end{align*}
where we integrated by parts and used $\nu(\pm \infty)=\rho(\pm \infty)=0$ (as $\nu, \rho \in H^1$) and 
\[ \int \beta_* |\mu|^2 dx=\int O_0^4(\eta)\, dx=O(\| \eta \|_{H^1}^4) = O(\| \eta \|_{H^1}^3), \]
as $\| \eta \|_{H^1}$ is small. Now
\begin{align}
\label{eq:gamma_beta}
1- (\gamma^2 + \beta_*) = 1- \gamma^2 -2(1-\gamma^2) \sin^2 \theta_* = (1-\gamma^2) (\cos^2 \theta_* - \sin^2 \theta_*),
\end{align}
and finally, by \eqref{def:Lg},
\[ E(w_* + \eta) - E(w_*) = \frac{1}{2} ( (L_\gamma\nu, \nu) + (L_\gamma \rho, \rho)) + O(\| \eta \|_{H^1}^3), \]
yielding \eqref{E_coer}.

\bigskip

\emph{Step 5. Expansion of the dissipation term. Proof of \eqref{E_diss}.}
As $m=w_*+\eta$,
\begin{align*}
\int |\delta E(m)|^2 dx - \int |m \cdot \delta E(m)|^2 dx & = \int \left(|\delta E(w_* + \eta)|^2 -  |\delta E(w_*)|^2 \right) dx \\
& \qquad - \int \left( |(w_*+\eta) \cdot \delta E(w_*+\eta)|^2 -  |w_* \cdot \delta E(w_*)|^2 \right)   dx
\end{align*}
because $|\delta E( w_*)| = \beta_* = |w_* \cdot \delta E(w_*)|$ by \eqref{dstar}. As $\delta E(v)$ is linear in $v$, the integrand in the first integral is
\begin{align} |\delta E(w_* + \eta)|^2 -  |\delta E(w_*)|^2  = 2 \delta E(w_*) \cdot \delta E(\eta) + |\delta E(\eta)|^2
\label{expansion_diss_1} = 2 \beta_* w_* \cdot \delta E(\eta) + |\delta E(\eta)|^2.
\end{align}
The integrand in the second integral is of fourth order in $\eta$: we expand it up to order 2. First,
\[ (w_*+\eta) \cdot \delta E(w_*+\eta) = w_* \cdot \delta E(w_*) + \eta \cdot \delta E(w_*) + w_* \cdot \delta E(\eta) + \eta \cdot \delta E(\eta). \]
Therefore, as $\eta \cdot w_* = \mu = O^2_0(\eta)$, by \eqref{dstar} and Step 3, 
\begin{align}
\nonumber \MoveEqLeft |(w_*+\eta) \cdot \delta E(w_*+\eta)|^2 - |w_* \cdot \delta E(w_*)|^2 \\
\nonumber & =  2 \left( \eta \cdot \delta E(w_*) + w_* \cdot \delta E(\eta) + \eta \cdot \delta E(\eta) \right) (w_* \cdot \delta E(w_*)) \\
\nonumber & \qquad + \left( \eta \cdot \delta E(w_*) + w_* \cdot \delta E(\eta) + \eta \cdot \delta E(\eta) \right)^2 \\
\nonumber & = 2 \beta_* \left( \beta_* \mu + w_* \cdot \delta E(\eta) +  \eta \cdot \delta E(\eta) \right) + \left( \beta_* \mu + w_* \cdot \delta E(\eta) + O^2_2(\eta) \right)^2 
\\
\label{expansion_diss_2} & = 2 \beta_*^2 \mu + 2 \beta_* w_* \cdot \delta E(\eta) + 2 \beta_*  \eta \cdot \delta E(\eta) +  |w_* \cdot \delta E(\eta)|^2 +  O^3_3(\eta) + O^4_4(\eta).
\end{align}
Summing up \eqref{expansion_diss_1} and \eqref{expansion_diss_2}, we see a cancellation of the linear term $2\beta_* w_* \cdot \delta E(\eta)$. We integrate, use Claim \ref{cl:O_IPP} (and that $\| \eta \|_{H^1} \le 1$) to bound the terms of order at least $3$ in $\eta$:
\begin{align*}
\MoveEqLeft \int |\delta E(m)|^2 dx - \int |m \cdot \delta E(m)|^2 dx =  \int |\delta E(\eta)|^2 dx \\
& \qquad  - 2 \int  \beta_*^2 \mu dx -  2 \int \beta_*  \eta \cdot \delta E(\eta) dx - \int  |w_* \cdot \delta E(\eta)|^2 dx + O(\| \eta \|_{H^1} \| \eta \|_{H^2}^2).
\end{align*}

We now use the expansions obtained in Step 3 and \eqref{eq:mu_eta4}: observe a partial cancellation in the first and last integral of the above right hand side, so that we have
\begin{align*}
\MoveEqLeft \int |\delta E(m)|^2 dx - \int |m \cdot \delta E(m)|^2 dx  \\
& = \int (- \partial_{xx} \nu + (1 - \gamma^2)  \nu)^2 dx + (- \partial_{xx} \rho + (1 - \gamma^2)  \rho)^2 dx + \int \beta_*^2 (\nu^2 + \rho^2) dx \\
& \quad - 2 \int \beta_* \left( - \nu \partial_{xx} \nu + (1 - \gamma^2) \nu^2 - \rho \partial_{xx} \rho + (1 - \gamma^2) \rho^2 \right) dx + O(\| \eta \|_{H^1} \| \eta \|_{H^2}^2) \\
& =  \int (- \partial_{xx} \nu + (1 - \gamma^2)  \nu - \beta_* \nu)^2 dx + (- \partial_{xx} \rho + (1 - \gamma^2)  \rho - \beta_* \rho)^2 dx + O(\| \eta \|_{H^1} \| \eta \|_{H^2}^2).
\end{align*}
In view of \eqref{eq:gamma_beta} and \eqref{def:Lg}, we conclude that
\[  \int |\delta E(m)|^2 dx - \int |m \cdot \delta E(m)|^2 dx = \| L_\gamma \nu \|_{L^2}^2 +  \| L_\gamma \rho \|_{L^2}^2 + O(\| \eta \|_{H^1} \| \eta \|_{H^2}^2). \]
Using finally \eqref{est:nu_rho_eta}, this establishes \eqref{E_diss}.

\bigskip

\emph{Step 6. Bound on the forcing term. Proof of \eqref{E_Efield}.}
Expanding $m=w_*+\eta$ in the forcing term
\begin{align}
\label{field_full}
(m \wedge e_1) \cdot (m \wedge \delta E(m)),
\end{align}
we show that the terms of order $0$ and $1$ in $\eta$ vanish when integrating. Indeed, the term of order $0$ in $\eta$ is $(w_* \wedge e_1) \cdot (w_* \wedge \delta E(w_*))  =0$ pointwise, so does not contribute to the integral. The term of order $1$ in $\eta$ is
\begin{align*} 
\MoveEqLeft (\eta \wedge e_1)  \cdot (w_* \wedge \delta E(w_*)) + (w_* \wedge e_1) \cdot \left(\eta \wedge \delta E(w_*) +w_* \wedge \delta E(\eta) \right)\\
&=0 - \sin \theta_* p_* \cdot \left[ (\eta \wedge \beta_* w_*) + (w_* \wedge \delta E(\eta)) \right] \\
& = \beta_* \sin \theta_* \nu - \sin \theta_*  \left(- \partial_{xx} \nu+ (1 - \gamma^2) \nu \right) + O_2^2(\eta) \\
& =  - \sin \theta_* L_\gamma \nu + O_2^2(\eta),
\end{align*}
 where we used $w_* \wedge e_1 = -\sin(\theta_*) p_*$, the expansion of $w_* \wedge \delta E(\eta)$ in Step 2 and \eqref{eq:gamma_beta}.
 
Recalling Lemma \ref{lem:Lg}, there hold
 \[ \int \sin \theta_* L_\gamma \nu  dx = 0. \]
Therefore, the contribution of the terms of order $1$ in $\eta$ is
\[ \int O_2^2(\eta) dx = O(\| \eta \|_{H^1}^2). \]
The quadratic term in $\eta$ of \eqref{field_full} is 
\begin{align*}
(\eta \wedge e_1)  \cdot [\eta \wedge \delta E(w_*) + w_* \wedge \delta E(\eta)] + (w_* \wedge e_1) \cdot (\eta \wedge \delta E(\eta)).
\end{align*}
As $\delta E(\eta) = O^1_2(\eta)$, the quadratic term in $\eta$ is clearly $O_2^2(\eta)$, and so, using Claim \ref{cl:O_IPP},  its contribution after integration is $O(\| \eta \|_{H^1}^2)$.

Finally, the cubic term in $\eta$ is $(\eta \wedge e_1)  \cdot (\eta \wedge \delta E(\eta)) = O^3_2(\eta)$ and again, due to Claim \ref{cl:O_IPP}, its contribution after integration is $O(\| \eta \|_{H^1}^3)$. As $\| \eta \|_{H^1}$ is small, we proved that
\[ \int (m \wedge e_1) \cdot (m \wedge \delta E(m)) dx = O(\| \eta \|_{H^1}^2),\]
 which proves \eqref{E_Efield}.
\end{proof}

\subsection{Proof of Theorem \ref{th:stab} when \texorpdfstring{$m_0 \in \q H^2$}{m0 dans qH2}} \label{sec:4.3}

Our goal is to define $\delta_0 >0$ of the statement of Theorem \ref{th:stab}.

\bigskip

Denote $\delta_*>0$ and $C_*>0$ the constants given by Lemma \ref{lem:H_H1_2} in Appendix \ref{append} for the magnetisation $w_*$: we can assume that $\delta_0 \le \delta_*$ so that
\begin{align} \label{est:m0-w*}
\| m_0 - w_* \|_{H^1} \le C_* \| m_0 - w_* \|_{\q H^1} \le C_* \delta_0.
\end{align}
By Theorem \ref{th:lwp}, there exists a unique solution $m \in \q C([0, T_+), \q H^2)$ to \eqref{ll} with initial data $m_0$. In order to prove Theorem \ref{th:stab}, we introduce a bootstrap argument, and for this, we need an extra large parameter $M \ge 2C_*$ (to be fixed later).
Define 
\[T_0 = T_0(m_0):=\sup  \left\{ T \in (0,T_+(m_0)) : \forall t \in [0,T], \quad \inf_{\zeta \in G} \| m(t) - \zeta. w_* \|_{H^1} < M \delta_0 \right\}, \]
where $T_+(m_0)$ is defined in Theorem \ref{th:lwp}.
By continuity of the \eqref{ll} flow in Theorem \ref{th:lwp}, in view of \eqref{est:m0-w*}, $T_0$ is well defined, i.e., $T_0 >0$. Also, reducing further $\delta_0$ if needed, we can assume that $M \delta_0 < \delta_2$ with $\delta_2$ defined in Proposition \ref{lem:decomp}. 

Our goal is to prove that $T_0 = T_+(m_0) = +\infty$. For that, let $T < T_0$. We work on the interval $[0,T]$.

\bigskip

\emph{Step 1. Main bootstrap estimates}.  Due to the definition of $T_0$ and $M\delta_0 \le \delta_2$, the decomposition in Proposition \ref{lem:decomp} applies to $m$ on $[0,T]$, and there exist $g: [0,T] \to G$ Lipschitz continuous and $\e \in \q C([0,T], H^2)$ with $\partial_t \e \in L^\infty([0,T], L^2)$, and there hold $m = g.(w_*+\e)$,
\begin{gather*}
\forall t \in [0,T], \quad |\dot g(t) - \dot g_*(t)| \le C_2 (1+|h(t)|) \| \e(t) \|_{H^1}, \quad  \| \e(t) \|_{H^1} \le C_2 M \delta_0,
\end{gather*}
and
\begin{align} \label{eq:ortho_e}
\forall t \in [0,T], \quad \int \e(t) \cdot \partial_x w_*(t) dx = \int \e(t) \cdot (e_1 \wedge w_*(t)) dx =0.
\end{align}
Also note that ( due to Lemma \ref{lem:mod})
\begin{align}
\label{est:e0}
\| \e(0) \|_{H^1} \le C_1 \| m_0 - w_* \|_{H^1} \le C_* C_1 \delta_0.
\end{align}
We decompose $\e$ on the basis $(w_*, n_*, p_*)$
\[ \e = \mu w_* + \nu n_* + \rho p_*, \]
and provided that $C_2 M\delta_0\le \delta_3$, we are in a position to apply Proposition \ref{prop:en_expansion} to $(-g(t)).m$: we get on $[0,T]$,
\begin{align}
\label{est:e_nu_rho} \| \e \|_{H^1} \le C_3  \| (\nu,\rho) \|_{H^1} \le  C_3^2 \| \e \|_{H^1} & \le C_3^2 C_2 M \delta_0, \\
\nonumber |E(m) - E(w_*) - \frac{1}{2} \left((L_\gamma \nu, \nu) + (L_\gamma \rho, \rho) \right) | & \le C_3 \| (\nu,\rho) \|_{H^1}^3, \\
\nonumber \left| \frac{d}{dt} E(m) + \alpha ( \| L_\gamma \nu \|_{L^2}^2 + \| L_\gamma \rho \|_{L^2}^2) \right| & \le C_3 \alpha  \| (\nu,\rho) \|_{H^2}^2  \| (\nu,\rho) \|_{H^1} + C_3 \alpha |h|  \| (\nu,\rho) \|_{H^1}^2,
\end{align}
where we used \eqref{eq:en_dissip2} for the last line. Now, the orthogonality conditions \eqref{eq:ortho_e} (together with \eqref{numbe}) write
\begin{align*}
0 & = \int \e \cdot \partial_x w_* dx = \sqrt{1-\gamma^2} \int \e \cdot \left( -\sin \theta_* n_* 
 - \frac{\gamma}{\sqrt{1-\gamma^2}}
 \sin \theta_* p_* \right) dx \quad \text{yielding} \\
0 & = \int \left( \nu \sin \theta_* + \frac{\gamma}{\sqrt{1-\gamma^2}} \rho \sin \theta_* \right) dx, \quad \text{and} \\
0 & = \int \e \cdot (e_1 \wedge w_*) dx = \int \sin \theta_*  \e \cdot p_* dx = \int \rho \sin \theta_* dx.
\end{align*}
We deduce
\[ \int \nu \sin \theta_* dx = \int \rho \sin \theta_* dx =0. \]
As a consequence, using Lemma \ref{lem:Lg}, \eqref{est:e_nu_rho} yields
\begin{align*}
2 \lambda_0 \| (\nu ,\rho) \|_{H^1}^2  & \le E(m) - E(w_*) + C_3 \| (\nu ,\rho) \|_{H^1}^3, \\
E(m) - E(w_*) & \le \frac1{2\lambda_0}  \| (\nu ,\rho) \|_{H^1}^2+ C_3 \| (\nu ,\rho) \|_{H^1}^3, \\
\frac{d}{dt} (E(m) - E(w_*)) + 4 \lambda_0 \alpha \| (\nu ,\rho) \|_{H^2}^2 & \le C_3 \alpha  \| (\nu,\rho) \|_{H^2}^2  \| (\nu,\rho) \|_{H^1} + C_3 \alpha |h|  \| (\nu,\rho) \|_{H^1}^2.
\end{align*}

\bigskip

\emph{Step 2. $T_0 = T_+(m_0)$.}
Recall that $\|h\|_{L^\infty}\le \delta_0$.
Given $M$, choose $\delta_0>0$ such that 
\begin{enumerate}
\item $C_3 \delta_0 \le \lambda_0$, so that by \eqref{est:e_nu_rho},  $C_3 |h|  \| (\nu,\rho) \|_{H^1}^2 \le \lambda_0 \| (\nu ,\rho) \|_{H^1}^2$ on $[0,T]$,
\item $C_2 C_3^2 M\delta_0 \le \lambda_0$ so that $C_3 \| (\nu,\rho) \|_{H^1} \le \lambda_0$ on $[0,T]$.
\end{enumerate}
This justifies all the expansions and bounds above (up to choosing $M$), and (letting $\lambda_1>0$ such that $\frac1{\lambda_1} = \frac{1}{2\lambda_0^2} +1$), we obtain that for all $t \in [0,T]$:
\begin{align}
  \| (\nu ,\rho) \|_{H^1}^2 & \le \frac{1}{\lambda_0} (E(m) - E(w_*)) \le \frac{1}{\lambda_1}  \| (\nu ,\rho) \|_{H^1}^2, \label{est:gr_1} \\
\frac{d}{dt} (E(m) - E(w_*)) & \le - 2 \lambda_0 \alpha \| (\nu ,\rho) \|_{H^2}^2 \le - 2\lambda_0 \alpha  \| (\nu ,\rho) \|_{H^1}^2. \label{est:gr_2}
\end{align}

\medskip
{
We use the following Gronwall type bound: if two functions $p,q:[0,T]\to \R$ satisfy 
\[ 0\le p\le \frac1{\lambda_0}  q \le \frac1{\lambda_1}  p \quad \text{and} \quad \dot{q}\le -\lambda_2 p \quad \text{on } [0,T], \]
for some constants $\lambda_0, \lambda_1, \lambda_2>0$, then 
\begin{align} \label{gronwall}
 \forall t \in [0,T], \quad q(t) \le \frac{\lambda_0 p(0)}{\lambda_1} \exp(-\frac {\lambda_1\lambda_2}{\lambda_0}t) \quad \text{and} \quad p(t) \le  \frac{p(0)}{\lambda_1} \exp(-\frac {\lambda_1\lambda_2}{\lambda_0}t).
\end{align}

\medskip

We use the above fact for $p:= \| (\nu ,\rho) \|_{H^1}^2$, $q:=E(m) - E(w_*)$ and $\lambda_2=2\lambda_0 \alpha$. By Lemma \ref{lem:mod} and \eqref{est:e_nu_rho}, we know that 
\[ p(0)\le C_3^2 \|\e(0)\|^2_{H^1}\le C_3^2 C_1^2 \|m_0-w_*\|^2_{H^1}\le C_3^2 C_1^2 C_*^2 \delta_0^2. \]
 Thus, \eqref{gronwall} together with  \eqref{est:e_nu_rho} yield
\be
\label{nou2}
\| \e(t) \|_{H^1}  \le C_3 \| (\nu ,\rho)(t) \|_{H^1} \le \frac{C_*C_1 C_3^2}{\sqrt{\lambda_1}}\delta_0 \exp(-\sigma t), \quad \forall t\in [0,T]\ee
where $\sigma:=\lambda_1 \alpha$.
We now choose $\ds M = 2 \max\left( C_*,  \frac{C_*C_1 C_3^2}{\sqrt{\lambda_1}} \right)$. Hence $\| \e(t) \|_{H^1} \le M\delta_0/2$ for all $t \in [0,T]$, and this bound is independent of $T < T_0$. 
Thus,  $M\delta_0/2\ge \| \e(t) \|_{H^1}=\|g(t). \e(t) \|_{H^1}\ge \inf_{\zeta\in G} \|m(t)-\zeta.w_*\|_{H^1}$, a continuity argument implies that $T_0=T_+(m_0)$. 

\bigskip

\emph{Step 3. $T_+(m_0)=+\infty$.}
For that, we have the uniform bound on $m$ in $\q H^1$:
\[ \forall t \in [0,T_+(m_0)), \quad \| m(t) \|_{\q H^1}=\| (-g(t)).m(t) \|_{\q H^1} \le \| w_* \|_{\q H^1} + \| \e(t) \|_{H^1} \le \| w_* \|_{\q H^1} + M \delta_0 \]
The blow up criterion of Theorem \ref{th:lwp} gives that $m$ is globally defined for positive times: $T_+(m_0)=+\infty$. In a nutshell,
\begin{align} \label{est:e_1st}
\forall t \ge 0, \quad \| \e(t) \|_{H^1} \le M \delta_0 \quad \text{and} \quad | \dot g (t) - \dot g_*(t) | \le C_2(1+ |h|) \| \e (t) \|_{H^1}.
\end{align}

\bigskip

\emph{Step 4. Exponential convergence of $g$ and $m$ in $H^1$.} By \eqref{nou2}, we have that
\[ \forall t \ge 0, \quad \|m(t)-g(t).w_*\|_{H^1}=\| \e(t) \|_{H^1} \le C \exp({-\sigma t}) \| m_0 - w_* \|_{\q H^1}, \]
where $\sigma = \alpha \lambda_1>0$ and $C >0$ is a large constant. The $\dot g$ part of \eqref{eq:decay} is then straightforward from \eqref{est:e_1st}. 

To prove \eqref{est:decay2}, note that \eqref{eq:decay} implies that the function $g-g_*$ has a limit $g_\infty\in G$ at $+\infty$ (as its derivative is integrable on $\R_+$) and therefore, for every $t\ge 0$,
\[ |g(t)-g_*(t)-g_\infty|\le \int_t^{\infty} | \dot g (s) - \dot g_*(s) |\, ds\le \tilde C \exp({-\sigma t}) \| m_0 - w_* \|_{\q H^1}. \]
By Lemma \ref{lem:mod}, we know that $|g(0)|\le C_1\| m_0 - w_* \|_{\q H^1}$ yielding  $|g_\infty|\le (\tilde C+C_1)\| m_0 - w_* \|_{\q H^1}$. Combined with \eqref{eq:decay}  and \eqref{here-a}, we conclude for every $t\ge 0$:
\begin{align*}
\|m(t)-(g_\infty+g_*(t)).w_*\|_{H^1}&\le \|m(t)-g(t).w_*\|_{H^1}+\|(g(t)-g_*(t)-g_\infty).w_*-w_*\|_{H^1}\\
&\le \hat C
\exp({-\sigma t}) \| m_0 - w_* \|_{\q H^1}.\end{align*}

\subsection{Proof of Theorem \ref{th:stab} when \texorpdfstring{$m_0 \in \q H^1$}{m0 dans qH1}}\label{sec:44}

We use a limiting argument. As $w_*\in \q H^2$ (the case treated above), we may assume $m_0\neq w_*$. By \eqref{def:e0}, the density argument at page \pageref{pag} yields a (smooth) sequence $m^n_{0} \in \q H^2$ such that $m^n_{0} \to m_0$ in $\q H^1$. Up to dropping the first terms, we can assume that for all $n \in \m N$,  $\| m^n_{0} - w_* \|_{H^1} \le 2 \| m(0) - w_* \|_{H^1} \le \delta_2$. Denote $m^n(t)\in \q C([0, T_+), \q H^2)$ the solution to \eqref{ll} with initial data $m^n_{0}$: due to the previous case of initial data $m^n_{0} \in \q H^2$ satisfying \eqref{def:e0}, we know that $m^n $ admits the decomposition
\begin{align} \label{eq:decomp_mn}
m^n = g^n.(w_* + \e^n),
\end{align}
and for some $g^n_{\infty} \in G$ uniformly bounded and  some universal constants $C,\sigma>0$, we have for all $t \ge 0$:
\begin{gather} \label{est:app_mn}
 \| \e^n(t) \|_{H^1} \le C e^{-\sigma t} \| m^n_{0} - w_* \|_{H^1}, \quad   | \dot g^n - \dot g_* | \le C \| \e^n(t) \|_{H^1}, \quad |g_\infty^n| \le C \| m_0 - w_* \|_{\q H^1}, \\
 |g^n(t) - (g^n_{\infty} + g_*(t))| \le 2 C e^{-\sigma t} \| m_{0} - w_* \|_{H^1}.  \label{est:app_mn2}
\end{gather}
As the flow of \eqref{ll} is continuous in $\q H^1$, we see that $m^n(t) \to m(t)$ in $\q H^1$ for all $t \in [0,T_+(m_0))$; by \eqref{eq:decomp_mn} and \eqref{est:app_mn}, $m^n$ is uniformly bounded in $\q H^1$ (in $n$ and $t$), so $m$ is too, and Theorem \ref{th:lwp} yields $T_+(m_0) = +\infty$. Then, from the bounds \eqref{est:app_mn}, $(g^n-g_*)_n$ is bounded in $W^{1,\infty}$. Using Ascoli's theorem and the uniform decay \eqref{est:app_mn2}, we infer that up to a subsequence, $ g^n - g\to 0$ in $ \q C_b([0,+\infty))$, $ \dot g^n- \dot g \to 0$ weakly-$*$ in $L^{\infty}([0,+\infty))$ and $g_\infty^n\to g_\infty$ in $G$ with $|g_\infty|\le C \| m_0 - w_* \|_{\q H^1}$.

As a consequence, $\e^n(t) = (-g^n(t)).m^n(t) -w_*$ converges in $\q C_b([0,+\infty), H^1)$: denote $\e$ its limit, by letting $n \to +\infty$ in the decomposition \eqref{eq:decomp_mn}, there holds $m(t) = g(t).(w_* + \e(t))$ for every $t\ge 0$. 
Finally, taking limits in \eqref{est:app_mn} and \eqref{est:app_mn2}, we deduce \eqref{eq:decay} and \eqref{est:decay2}: the proof of Theorem \ref{th:stab} is complete.

\appendix

\section{Toolbox}
\label{sec:tool}

 We compute for $a,b,c\in \m R^3$, $g= (y,\phi) \in G$, $m, \tilde m:\m R\to \m S^2$, $m=(m_1, m_2, m_3)$: 
 \[ a \wedge( b \wedge c) = (a\cdot c) b - (a \cdot b)c, \quad
e_1\wedge a=(a\cdot e_2) e_3-(a\cdot e_3) e_2,
\]
\[ R_\phi a\wedge R_\phi b=R_\phi (a\wedge b), \quad
 \partial_\phi (R_\phi a)=e_1\wedge (R_\phi a)=R_\phi (e_1\wedge a), \, \,\]
\[ g.m\wedge g.\tilde m=g.(m\wedge \tilde m),\, \, \partial_y (g.m) = -g. \partial_x m, \, \,  \partial_\phi (g.m) = e_1 \wedge g.m=g.(e_1 \wedge m),\] 
\[ \partial_x m \cdot (e_1\wedge m)=m_2\partial_x m_3-m_3\partial_x m_2, \quad |e_1\wedge m|^2=1-m_1^2.\]

\section{Some properties of the sets \texorpdfstring{$\q H^1$}{qH1} and \texorpdfstring{$H^1(\m R, \m S^2)$}{H1(R,S2)}}

\label{append}
We recall that for $m = (m_1,m_2,m_3) \in \m R^3$, we denote $m' = (m_2,m_3) \in \m R^2$ the last two coordinates.

For $s \ge 1$, we defined for every $m:\m R\to \m R^3$, 
\[ \| m \|_{\q H^s} : = \| m_2 \|_{L^2} + \| m_3 \|_{L^2} + \| m \|_{\dot H^s}, \]
where $\dot H^s$ was defined in \eqref{def:H^s}. We start by proving the following interpolation inequality\footnote{One can prove that $\q H^s(\m R, \m S^2)\subset \dot H^1(\m R, \m S^2)$ for every $s\ge 1$, but for our purposes,  in the following, we restrict to the case $s\in [1,2]$.} $\q H^s(\m R, \m S^2)\subset \dot H^1(\m R, \m S^2)$ for $s\in [1,2]$. We also prove the behaviour at $\pm \infty$ of maps in $\q H^s(\m R, \m S^2)$.

\begin{lem} \label{lem:Hs_H1}
Let $s \in [1,2]$. There exists $C>0$ such that for any $m: \m R \to \m S^2$ with $\| m \|_{\q H^s} <+\infty$, then $\partial_x m \in L^2$ and
\begin{equation} \label{est:H1_Hs}
 \| \partial_x m \|_{L^2} \le C \| m \|_{\q H^s} (1+  \| m \|_{\q H^s}^2 );
 \end{equation}
in particular, $E(m)<\infty$, and $m$ admits limits belonging to $\{ (\pm 1,0,0) \}$ as $x \to \pm\infty$.
\end{lem}

\begin{proof} We divide the proof in two steps:

\medskip

\nd {\it Step 1:  $s=1$}. In this case, the desired inequality follows by the definition of $\q H^1$. Note that 
\[ |\partial_x m\cdot (e_1\wedge m)|\le |\partial_x m'|\cdot |m'|\le \frac12(|\partial_x m'|^2+ |m'|^2). \]
We infer that $E(m)\le C  \| m \|^2_{\q H^1}$.

\medskip

\nd {\it Step 2:  $s \in(1,2]$.} First observe that $m_2 \in H^s=L^2\cap \dot H^s$ and by interpolation, we have
\[ \| \partial_x m_2 \|_{L^2} \lesssim \| m_2 \|_{L^2} + \| m_2 \|_{\dot H^s}  \lesssim \| m \|_{\q H^s}, \]
and similarly $m_3 \in H^s$ and $\| \partial_x m_3 \|_{L^2} \lesssim \| m \|_{\q H^s}$.
In particular, $|m_2|, |m_3| \in H^1$; as $m$ takes its values in $\m S^2$, this means that $\sqrt{1-m_1^2} = |(m_2,m_3)| \in H^1$ and
\be
\label{12}
 \| \sqrt{1-m_1^2}  \|_{H^1} \lesssim \| (m_2,m_3) \|_{H^1} \lesssim \| m \|_{\q H^s}. 
 \ee
In particular,  the functions $m_2, m_3, |m_1|$ are continuous and have the following limits $m_2(\pm \infty)=m_3(\pm \infty)=0$ and $|m_1|(\pm \infty)=1$.

We now consider 
\[ A : = \{ x \in \m R: |m_1(x)| < \frac{1}{2} \} \]
which is an open bounded set. First, we estimate the length $|A|$ of $A$:
\be 
\label{estA}
\| m \|_{\q H^s}^2 \ge \int_{\R} (1-m_1^2) \, dx\ge \int_A \frac{3}{4}\, dx \ge \frac{3}{4} |A|. \ee
Next, we estimate the $L^2$ norm of $\partial_x m_1$. On the set $A^c = \R \setminus A$, one has
\[ |\partial_x m_1| \le 2 \frac{|m_1 \partial_x m_1|}{\sqrt{1-m_1^2}} = 2 \big| \partial_x \sqrt{1-m_1^2}\big|\quad \textrm{a.e. in } A^c.\]
Therefore, by \eqref{12},
\[ \int_{A^c} |\partial_x m_1|^2 \, dx\le 4 \| \partial_x \sqrt{1-m_1^2} \|_{L^2}^2 \lesssim \| m \|_{\q H^s}^2. \]
It remains to estimate the $L^2$ norm of $\partial_x m_1$ on $A$. 
As $A$ is open and bounded, we write
\[ A = \bigcup_{k\in \q K} I_k, \quad I_k = (a_k^-, a_k^+), \quad |m_1(a_k^\pm)|=\frac12,\]
where $\{I_k\}_{k\in \q K}$ is a (at most) countable family of disjoint open bounded intervals.
We have
\begin{align*}
\int_A |\partial_x m_1|^2 \, dx& = \sum_{k\in \q K} \int_{I_k} |\partial_x m_1|^2\, dx \\
& \le 2 \sum_{k\in \q K}  \left(\int_{I_k} \big|\partial_x m_1 - \fint_{I_k} \partial_x m_1  \big|^2 \, dx+ |I_k| \big( \fint_{I_k} \partial_x m_1 \big)^2 \right) = 2 (J_1 + J_2),
\end{align*}
where $ \fint_{I_k}$ denotes the average on the bounded interval $I_k$.

\nd {\it Estimating $J_2$}. Let $\q J = \big\{ k \in \q K\, : \, m_1(a_k^-) \ne m_1(a_k^+) \big\}$. Then 
$|m_1(a_k^-) - m_1(a_k^+)|=1$ if $k \in \q J$ and $m_1(a_k^-) - m_1(a_k^+) =0$ if $k \notin \q J$. Thus,
\begin{align*}
J_2 = \sum_{k\in \q K} \big|I_k\big|  \left( \fint_{I_k} \partial_x m_1 \right)^2 = \sum_{k \in \q J} \frac{1}{|I_k|}. 
\end{align*}
To estimate the length $|I_k|$ of $I_k$ for $k\in \q J$, as $m_1(a_k^-)$ and $m_1(a_k^+)$ have different sign, 
the continuity of $m_1$ implies the existence of $b_k \in (a_k^-, a_k^+)$ such that $m_1(b_k) =0$. Denoting by $f = \sqrt{1-m_1^2} \in H^1$ (by \eqref{12}), the Cauchy-Schwarz inequality implies
\[ 1- \frac{\sqrt{3}}{2} = | f(a_k^-) - f(b_k)|\le \int_{a_k^-}^{b_k} |\partial_x f| \, dx\le |b_k - a_k^-|^{1/2} \left( \int_{a_k^-}^{b_k} |\partial_x f|^2 \, dx\right)^{1/2}. \]
Arguing similarly on $[b_k,a_k^+]$, we get after squaring, for $c = 7/2 - 2\sqrt 3 >0$:
\[ c  \le (|b_k - a_k^-| + |b_k - a_k^+|)  \int_{I_k} |\partial_x  f|^2\, dx = |I_k|   \int_{I_k} |\partial_x f|^2\, dx . \]
Therefore, by \eqref{12},
\[ J_2=\sum_{k \in \q J} \frac{1}{|I_k|} \le \frac{1}{c} \sum_{k \in \q J}   \int_{I_k} |\partial_x f|^2 \, dx\lesssim \int_{\R}  |\partial_x f|^2 \, dx\lesssim \| m \|_{\q H^s}^2. \]

\nd {\it Estimating $J_1$}. We consider separately the cases $s \in (1,2)$ and $s=2$. 

\nd {\it Case 1. Assume that $s \in (1,2)$.} We have by Jensen's inequality:
\begin{align*}
J_1 & = \sum_{k\in \q K} \int_{I_k}  \left|\partial_x m_1 - \fint_{I_k} \partial_x m_1  \right|^2 \, dx\\
& \le \sum_{k\in \q K} \int_{I_k} \int_{I_k} \frac{|\partial_x m_1(x) - \partial_x m_1(y)  |^2}{| I_k|} dydx \\
& \le \sum_{k\in \q K} \int_{I_k} \int_{I_k} \frac{|\partial_x m_1(x) - \partial_x m_1(y)  |^2}{| x-y|^{1+2(s-1)}} | I_k |^{2(s-1)} dydx \\
& \le  |A|^{2(s-1)}   \iint_{\m R^2}  \frac{|\partial_x m_1(x) - \partial_x m_1(y)  |^2}{| x-y|^{1+2(s-1)}}   dxdy \\
& \lesssim |A|^{2(s-1)} \| \partial_x m_1 \|^2_{\dot H^{s-1}} \lesssim  |A|^{2(s-1)}  \| m \|_{\q H^s}^2 \lesssim \| m \|_{\q H^s}^{2+4(s-1)},
\end{align*}
where we used \eqref{estA}. As $2+4(s-1) \le 6$, this completes the case when $s <2$.

\medskip

\nd {\it Case 2. Assume that $s=2$.} As before, we have by Jensen's inequality
\begin{align*}
J_1 & =  \sum_{k\in \q K} \int_{I_k}  \left| \fint_{I_k} \big(\partial_x m_1(x) - \partial_x m_1(y)\big) dy \right|^2 dx \\
&\le  \sum_{k\in \q K} \int_{I_k}   \fint_{I_k} \left| \int_x^y \partial_{xx} m_1(z) dz \right|^2 dydx \\
& \le \sum_{k\in \q K} \int_{I_k}   \fint_{I_k}  |x-y| \left( \int_x^y |\partial_{xx} m_1(z)|^2 dz \right) dydx \\
&\le \sum_{k\in \q K}  |I_k|^2 \int_{I_k} |\partial_{xx} m_1(z)|^2 \, dz
\le |A|^2 \int_A |\partial_{xx} m_1|^2 \, dz \lesssim  \| m \|_{\q H^s}^6,
\end{align*}
where we used \eqref{estA}. The bound \eqref{est:H1_Hs} follows.

Finally, $m_1$ is continuous (as $\partial_x m_1 \in L^2$) and we saw that $|m_1| \to 1$ as $x \to +\infty$. As $\{ \pm 1 \}$ is discrete, 
we infer that $m_1 \to a^+$ as $x \to +\infty$ with $a^+ \in \{ \pm 1 \}$. The same argument shows that  $m_1 \to a^-$ as $x \to -\infty$ with $a^- \in \{ \pm 1 \}$.
\end{proof}

{In the next lemma, we prove the reverse inequality with respect to Lemma \ref{lem:Hs_H1} relating $E(m)$ and $\|m\|_{\q H^1}$. In other words, the seminorm $\|\cdot\|_{\q H^1}$ fully characterises the set of finite energy configurations $E(m)<\infty$. 

\begin{lem}
\label{lem:infinit}
If $\gamma^2<1$ then there exists a constant $C_\gamma>0$ such that 
\[ E(m)\ge C_\gamma \|m\|^2_{\q H^1} \quad \textrm{for every } m: \m R \to \m S^2. \]
\end{lem}

\begin{proof}
Let $m=(m_1, m_2, m_3):\R\to \m S^2$ and $m'=(m_2, m_3)$. As $|\gamma|<1$, we can choose $a,b\in (-1,1)$ such that $ab=\gamma$. Then 
\[ a^2|\partial_x m'|^2+2\gamma \partial_x m\cdot (e_1 \wedge m)+b^2 |m'|^2=(a\partial_x m_{2}-b m_{3})^2+(a\partial_x m_{3}+b m_{2})^2. \]
Therefore, 
\[ E(m)\ge \frac12 \int_\R (\partial_x m_1)^2+(1-a^2) |\partial_x m'|^2 +(1-b^2) |m'|^2\, dx\ge C_\gamma   \|m\|^2_{\q H^1}. \]
(If one chooses $a^2=b^2=|\gamma|$, then $C_\gamma$ can be taken equal to $\frac{1-|\gamma|}2$). 
\end{proof}

In the following, we study more closely the norms $\| \cdot \|_{\q H^1}$ and $\| \cdot \|_{H^1}$ for maps with target $\m S^2$. First, note that Lemma \ref{lem:infinit} implies that two maps $m, \tilde m\in \q H^1(\m R, \m S^2)$ may have different limits at $\pm \infty$, so $m-\tilde m$ might not belong to $L^2$ (e.g., $\tilde m=-m$). However, if $m$ and $\tilde m$ have the same limits at $\pm \infty$, then we show that $m-\tilde m\in L^2$. Next we prove an important stability property for non-constant maps: if $m$ and $\tilde m$ are close to each other in $\q H^1$ and they are not constant, then they have the same limits at $\pm \infty$ and $m-\tilde m$ is small in $L^2$.

\begin{prop} \label{prop:conv_H_H1}
1) Let $m, \tilde m\in \q H^1(\m R, \m S^2)$ with the same limits at infinity, i.e., $m(\pm \infty) = \tilde m(\pm \infty)$. Then $m-\tilde m \in H^1$.

2) Let $m\in \q H^1(\m R, \m S^2)$ such that $m_1$ is not a constant function. Then for all $\e>0$, there exists $\delta = \delta(m)>0$ such that if $\tilde m\in \q H^1(\m R, \m S^2)$ satisfies $\| m - \tilde m \|_{\q H^1} \le \delta$, then
\begin{equation} \label{H_and_H1}
\| m - \tilde m \|_{H^1} \le \e; 
\end{equation}
in particular, $m(\pm \infty) = \tilde m(\pm \infty)$.
\end{prop}

\begin{nb} \label{nb:m1_constant} 
The case of a constant component $m_1$ is different. For example, observe that if $m=e_1$, then $\tilde m=-e_1$ has the property that $\|m-\tilde m\|_{\q H^1}=0$, but $m-\tilde m\notin L^2$. 
\end{nb}

\begin{proof} We divide the proof in several steps:

\nd {\it Step 1. We prove 1).} Consider $u = m - \tilde m$. Since $m, \tilde m\in \q H^1(\m R, \m S^2)$, then $\| u \|_{\q H^1} < +\infty$, so that $u_2,u_3 \in H^1$ and $\partial_x u_1 \in L^2$. It remains to prove that $u_1\in L^2$. As $m(\pm \infty) = \tilde m(\pm \infty)$, we deduce that $u_1=m_1-\tilde m_1$ (which is continuous) tends to 0 at $\pm \infty$. By Lemma \ref{lem:Hs_H1}, we have 
$(m_2,m_3), (\tilde m_2, \tilde m_3)\to 0$ as $x\to \pm \infty$; combined with $m(\pm \infty) = \tilde m(\pm \infty)$, we deduce $\lim_{x\to \pm \infty} |m_1 + \tilde m_1|=2|m(\pm \infty)|=2$. 
Thus, there exists $R>0$ such that 
\be
\label{56} |m_1(x) - \tilde m_1(x)| \le 1/2 \quad \text{and} \quad |m_1(x) + \tilde m_1(x) | \ge 1  \quad \textrm{ if } \, |x| \ge R. \ee
We estimate the $L^2$ norm of $u_1$ in the set $\big\{|x| \ge R\big \}$ and then in the set  $\big\{|x| \le R\big \}$. \eqref{56} implies 
\[
\forall |x| \ge R, \quad u_1(x)^2  \le \frac{1}{2} |m_1(x) - \tilde m_1(x) | 
\le \frac{1}{2} |m_1(x)^2 - \tilde m_1(x)^2|,
\]
yielding
 \begin{align*} 
 \int_{\m R \setminus [-R,R]} u_1(x)^2 dx &\le \frac{1}{2}  \int_{\m R \setminus [-R,R]} |m_1(x)^2 - \tilde m_1(x)^2|  dx \le \frac 12 \int_{\R} |m_2^2 +m_3^2 -  \tilde m_2^2 -\tilde m_3^2|\, dx < \infty 
 \end{align*}
 because $m, \tilde m$ are $\m S^2$-valued maps and $(m_2,m_3), (\tilde m_2, \tilde m_3)\in L^2$. 
 The estimate of the $L^2$ norm of $u_1$ in the set $\{|x| \le R \}$ is easy:
 as $|u_1| \le |m_1| + |\tilde m_1| \le 2$, it follows
\[ \int_{-R}^R u_1(x)^2 dx \le 8 R < +\infty. \]
Hence $u_1 \in L^2(\m R)$. 

\medskip

\nd {\it Step 2. We prove 2).} First in the following claim, we prove smallness of the $L^\infty$ norm of $u = m - \tilde m$ provided that $\| m - \tilde m \|_{\q H^1}$ is small.

\begin{claim} Let $m\in \q H^1(\R, \m S^2)$ with $m_1$ non-constant and
let $u^n$ be a sequence such that $m + u^n: \m R \to \m S^2$ and $\| u^n \|_{\q H^1} \to 0$ as $n\to \infty$. Then $\| u^n \|_{L^\infty} \to 0$ as $n\to \infty$.
\end{claim}

\begin{proof}[Proof of Claim]
The assumption $\| u^n \|_{\q H^1} \to 0$ as $n\to \infty$ implies $\| u^n_2 \|_{H^1}, \| u^n_3 \|_{H^1} \to 0$; thus, the embedding $H^1\hookrightarrow L^\infty$ yields $u^n_2$, $u^n_3 \to 0$ in $L^\infty$ as $n\to \infty$.  It remains to show that $u_1^n \to 0$ in $L^\infty$. For this, observe that the constraints $m, m + u^n \in \m S^2$ yield
\[ 2m \cdot u^n + |u^n|^2 =0 \quad \textrm{in} \quad \R. \]
Combined with $u^n_2,u^n_3 \to 0$ in $L^\infty$ as $n\to \infty$, since $|m|=1$ in $\R$, we deduce
\be
\label{eq:m.v_n}
 2m_1 u^n_1 + (u^n_1)^2 \to 0 \quad \text{in } L^\infty\quad \textrm{as $n\to \infty$}. \ee
Recall that  $\| u^n_1 \|_{L^\infty}\le \| u^n \|_{L^\infty} \le 2$. Up to a subsequence, we can assume that $u^n_1(0)\to c$ for some limit 
$c\in [-2,2]$. 

\smallskip

\emph{Step 1. We claim that $c=0$ and for any $R >0$, $u_1^n\to 0$ in $L^\infty([-R,R])$.} Indeed, it suffices to consider large $R$, and as $m_1$ is not constant in $\R$,  we can assume that $m_1$ is not constant in $[-R,R]$. For such $R$ and any $x \in [-R,R]$,
\[ |u^n_1(x)-c|\le |u^n_1(0)-c|+\sqrt{R}\|\partial_x u^n_1\|_{L^2(-R,R)}. \]
As $\partial_x u^n_1 \to 0$ in $L^2$, we deduce that $u_1^n\to c$ in $L^\infty([-R,R])$.
Passing to the limit in \eqref{eq:m.v_n}, it yields $2 m_1 c + c^2 =0$ in $[-R,R]$. If $c\neq 0$, then we would obtain  $m_1 = -c/2$ is constant  in $[-R,R]$ which contradicts our assumption. Thus, $c=0$, i.e., $u_1^n\to 0$ in $L^\infty([-R,R])$ as $n\to \infty$. 

\smallskip

\emph{Step 2. Convergence on all $\m R$ of $u^n_1 \to 0$ in $L^\infty$.} 

Assume for the sake of contradiction that there is $\delta \in (0, \frac12)$, a sequence of points $x_n \in \m R$ and a subsequence still denoted $u^n$ for simplicity of notations, such that $|u^n_1(x_n)| \ge \delta$ for every $n$.

Up to choosing a further subsequence, we can also assume $x_n\to x_\infty$ with $x_\infty \in [-\infty, \infty]$. First observe that $x_\infty \in \{\pm \infty \}$. Indeed,
if $x_\infty\in \R$, then choosing $R>0$ sufficiently large, we would have that $x_\infty \in (-R,R)$ and $u_1^n\to 0$ in $L^\infty([-R,R])$ which contradicts the fact that $|u^n_1(x_n)| \ge \delta>0$ and $x_n\to x_\infty \in (-R,R)$. Therefore, $x_0=\pm\infty$: we will assume $x_0=+\infty$ (the other case follows similarly). 

By Lemma \ref{lem:infinit}, we know that $m_1$ has limits belonging to $\{\pm 1\}$ as $x\to \pm \infty$; fix a large $R$, such that $|m_1|\ge \frac12$ outside $(-R,R)$. As  $x_n \to +\infty$ and $u_1^n\to 0$ in $L^\infty([-R,R])$, there exists $N>0$ such that for every $n\ge N$, $x_n \ge 2R$ and $|u^n_1|\le \frac\delta2$ in $[-R,R]$. As $|u^n_1(x_n)| \ge \delta$ and $x_n\ge 2R$, by continuity of $u^n_1$, there exists $y_n> R$ such that 
$|u^n_1(y_n)|=\delta$ for every $n\ge N$. Using $|m_1(y_n)|\ge \frac 12$ (because $y_n\notin [-R,R]$), we deduce that for $n\ge N$:
\[ |2m_1(y_n) u^n_1(y_n) + (u^n_1)^2(y_n)| \ge 2 |m_1(y_n)| \delta-\delta^2\ge \delta(1-\delta)\ge \frac\delta 2. \]
This contradicts \eqref{eq:m.v_n}. 

This argument shows that every subsequence of $u^n_1$ has a further subsequence converging to the unique limit $0$ in $L^\infty$; as $L^\infty$ is a Hausdorff space, we conclude that $u^n_1 \to 0$ in $L^\infty$.
\end{proof}

\medskip

\nd {\it End of proof of point 2)}. We now prove smallness of the $H^1$ norm of $u$ as stated in \eqref{H_and_H1}. 
By Lemma \ref{lem:infinit}, there exists $R \ge 1$ such that $|m_1| \ge 1/2$ for every $|x| \ge R$. 
Let $\e\in (0,1)$. The above claim implies the existence of $\delta=\delta(m) >0$ such that if $m+u: \m R \to \m S^2$ and $\| u \|_{\q H^1} \le \delta$, then $\| u \|_{L^\infty} \le \frac\e{2\sqrt{R}}<\frac12$. As $m+u$ and $m$ are $\m S^2$-valued maps, we have
\[ 2m_1 u_1 =- u_1^2 + v, \quad v: = -2(m_2u_2 + m_3 u_3) - u_2^2 - u_3^2. \]
Squaring the above equality, as $|m_1| \ge 1/2$ outside $(-R,R)$ and $|u_1|\le |u| \le \frac12$, we get 
\[ u_1^2 \le (2m_1 u_1)^2 \le 2 (u_1^4 + v^2) \le \frac{1}{2} u_1^2 + 2 v^2, \quad \textrm{for } |x| \ge R\]
and so $u_1^2 \le 4 v^2$ outside $(-R,R)$. We now integrate on $|x| \ge R$; as $|m|=1$ and $|u|^2\le |u|\le \frac12$ in $\R$ yielding $|v|\le 3(|u_2|+|u_3|)$ in $\R$, we obtain for some universal constant $C>1$:
\[ \int_{|x| \ge R} u_1^2 \le 4 \| v \|_{L^2}^2 \lesssim \| (u_2,u_3) \|_{L^2}^2 \le (C-1) \| u \|_{\q H^1}^2. \]
Combined with
\[ \int_{-R}^R u_1^2 \, dx \le 2 R \| u_1 \|_{L^\infty}^2 \le \frac{\e^2}2, \]
we get
\[ \| u \|_{H^1}^2 = \| u_1 \|_{L^2}^2 + \| u \|_{\dot H^1}^2 \le \frac{\e^2}2 + C \| u \|_{\q H^1}^2 \le \frac{\e^2}2 + C \delta^2. \]
Up to lowering $\delta>0$ further, we can also assume that $C \delta^2 \le \e^2/2$, and so $\| u \|_{H^1} \le \e$.
\end{proof}

As mentioned in Remark \ref{nb:m1_constant}, the case of a constant first  coordinate should be considered separately, and for this, we have to take into account the following (mirror) symmetry in the first component}
\be
\label{mirror}
m^c := (-m_1,m_2,m_3).
\ee
The following result holds for continuous maps $m$ not necessarily belonging to $\q H^1$.  

\begin{lem} \label{lem:H_H1_1}
Let $m=(m_1, m_2, m_3): \m R \to \m S^2$ be a continuous map with $m_1=c$ constant in $\R$. 

1) If $c\ne 0$, then there exist  $\delta,C>0$ (depending on\footnote{One can choose $\delta = O(c^2)$ and $C = O(1/c)$ as $c\to 0$.} $c$) such that if $\tilde m: \m R \to  \m S^2$ satisfies $\| m - \tilde m \|_{\q H^1} \le \delta$, then 
\[  \min( \| m - \tilde m  \|_{H^1}, \| m^c - \tilde m \|_{H^1}) \le C \| m - \tilde m \|_{\q H^1}. \]

2) If $c=0$, then there exists a universal $C>0$ such that for every $\tilde m: \m R \to  \m S^2$,
\[ \| m - \tilde m \|_{H^1} \le C(\| m - \tilde m \|_{\q H^1}+\| (m - \tilde m)' \|^{1/2}_{L^1}), \]
where we recall the notation $(m - \tilde m)'=(m_2 - \tilde m_2, m_3 - \tilde m_3)$.
\end{lem}

\begin{proof}
Denote $u:= \tilde m-m$. As in the proof of Proposition \ref{prop:conv_H_H1}, since $m$ and $\tilde m$ are $\m S^2$-valued, we get
\[ -c^2\le 2c u_1 + u_1^2 = v, \quad v: = -2 (m_2u_2 + m_3 u_3) - u_2^2 - u_3^2, \quad \textrm{in } \R\]
so that
\[ u_1 \in \bigg\{ - c \pm \sqrt{c^2 + v} \bigg\} \quad \textrm{in } \R.\]

\nd {\it Case 1.  $c=0$}. In this case, $|(m_2,m_3)|=1$ (note that $m\notin \q H^1$) and $|u_1| = \sqrt{v}$; then there is a constant $C>0$ (independent of $m$) such that 
\[ \|u_1\|^2_{L^2}=\|v\|_{L^1}\le C( \| (u_2,u_3) \|^2_{L^2}+ \| (u_2,u_3) \|_{L^1}), \]
which leads to conclusion 2).

\medskip

\nd {\it Case 2.  $c\neq 0$}.
As $|m|=1$ and $|u|^2\le 2|u|\le 4$ in $\R$, we obtain that 
\be
\label{ine34}
\| v \|_{L^2} \lesssim \| (u_2,u_3) \|_{L^2} \le M \| u \|_{\q H^1}\ee
for some universal constant $M>0$ (not depending on $m$).
Also, as $\partial_x v=2c\partial_x u_1+ 2u_1 \partial_x u_1$ and $|u_1|\le 2$, we deduce that
$\| v \|_{\dot H^1}\lesssim \| u \|_{\q H^1}$. By the Sobolev embedding $H^1\hookrightarrow L^\infty$, up to possibly increase the above constant $M$, it follows 
\[ \| v \|_{L^\infty}\le M \| u \|_{\q H^1}. \]
Choosing $\delta \le c^2/(2M)$, if $\| u \|_{\q H^1} \le \delta$, then the above estimate yields $\| v \|_{L^\infty} \le c^2/2$, and in particular, $v+c^2\ge c^2/2>0$ in $\R$. 
As $u_1$ is a continuous function belonging to $ \bigg\{ - c \pm \sqrt{c^2 + v} \bigg\}$, then there exists a sign $\sigma \in \{ \pm 1 \}$ such that 
\[ u_1 = -c + \sigma \sqrt{c^2 +v} = |c| (\sigma - \sgn(c)) - \sigma |c| \left(1 - \sqrt{1+\frac v{c^2}} \right) \, \, \textrm{ in } \R. \]
Note that
\[ \left| 1 - \sqrt{1+\frac v{c^2}}  \right| \le \frac{|v|}{c^2} \, \, \textrm{ in } \R.\]

\nd {\it Subcase i).  $\sigma = \sgn(c)$}. In this case, $ |u_1| \le \left|\frac v c\right|$ in $\R$ and we get by \eqref{ine34}:
\[\|u_1 \|_{L^2} \le \frac{M}{c} \| u \|_{\q H^1}, \]
and so
\[  \| m-\tilde m \|_{H^1} = \| u \|_{H^1} \le \left( 1+ \frac{M}{c} \right)  \| u \|_{\q H^1}. \]

\nd {\it Subcase ii).  $\sigma = -\sgn(c)$}. In this case,
\[ |u_1 +2c | \le \left|\frac v c\right|, \]
and similarly, we conclude
\[  \| m^c-\tilde m \|_{H^1} =  \| (2c+u_1,u_2,u_3) \|_{H^1} \le \left( 1+ \frac{M}{c} \right)  \| u \|_{\q H^1}. \qedhere \]
\end{proof}

We go back to magnetisations $m$ with non constant first component. Actually, the smallness \eqref{H_and_H1} in Proposition \ref{prop:conv_H_H1} can be quantified to a linear bound under an additional nondegeneracy assumption on $m$ 
\begin{lem} \label{lem:H_H1_2}
Let $m\in \q H^1(\m R, \m S^2)$ be such that $m_1$ is non-constant and satisfies the nondegeneracy condition for some $c>0$:
\[ |m_1| + |\partial_x m_1| \ge c, \quad \textrm{a.e. in } \, \R.\]
Then there exist $\delta,C>0$ (depending on $m$\footnote{Tracking the constants in the proof shows that one can choose $\delta = \delta(c)$ given by Proposition \ref{prop:conv_H_H1} and $C= O(1+ \| m \|_{\q H^1}/c)$.}) such that if $\tilde m: \m R \to \m S^2$ satisfies $\| m - \tilde m \|_{\q H^1} \le \delta$, then 
\[ \| m - \tilde m \|_{H^1} \le C \| m - \tilde m \|_{\q H^1}. \]
\end{lem}

We emphasise that the nondegeneracy assumption is verified for $w_*$\footnote{By \eqref{formul1}, we have 
$|(w_*)_1|+|\partial_x (w_*)_1|\ge \cos^2 \theta_*+\sqrt{1-\gamma^2} \sin^2 \theta_*\ge \sqrt{1-\gamma^2}>0$ in $\R$.}, and this is our purpose in Theorem \ref{th:stab}.

\begin{proof}
Denote $u = \tilde m - m$ as before. If $u_1=0$, then we are done. Otherwise, we may assume that $\|u_1\|_{L^2}>0$. In the following, $M>0$ is a universal constant that can change from line to line. We fix $\e>0$ that is given later (it will be defined below on \eqref{choice_ep}). By 
Proposition \ref{prop:conv_H_H1}, there exists $\delta\in (0,1)$ so that $\|u\|_{\q H^1}\le \delta$ implies $\| u \|_{H^1} \le \e$.
By the constraint $|m|=|\tilde m|=1$, it follows
\[-m_1 u_1=\frac12 |u|^2 +m_2 u_2+m_3 u_3 \quad \textrm{ in } \, \R\]
yielding
\[ \|m_1 u_1\|_{L^2} \le \frac{1}{2} (\|u_1\|_{L^4}^2 + \|u_2\|_{L^4}^2+\|u_3\|_{L^4}^2)+ \|m_2u_2\|_{L^2} + 
\|m_3u_3\|_{L^2}. \]
Using the Gagliardo-Nirenberg inequality $\|f\|_{L^4}\lesssim \|\partial_x f\|^{1/4}_{L^2} \|f\|^{3/4}_{L^2}$ and the H\"older inequality $\|m_ju_j\|_{L^2}\le \|u_j\|_{L^4}\|m_j\|_{L^4}$ for $ j=2, 3$,  we deduce that
\begin{align*}
\|m_1 u_1\|_{L^2}&\le \frac12\|u_1\|_{L^\infty}\|u_1\|_{L^2}+M(\|m\|_{\q H^1}\|u\|_{\q H^1}+\|u\|^2_{\q H^1})\\
&\le \frac{1}{2} \e \|u_1\|_{L^2}+M(\|m\|_{\q H^1}+1)\|u\|_{\q H^1}, 
\end{align*}
where we used $\|u\|^2_{\q H^1}\le \|u\|_{\q H^1}\le \delta<1$ and the Sobolev embedding $\|u_1\|_{L^\infty} \le \| u \|_{H^1} \le \e$. Combined with the non-degeneracy condition and using again the Gagliardo-Nirenberg and H\"older inequality, we obtain
\begin{align*}
c \int_{\R} u_1^2 \, dx &\le \int_{\R} (|m_1| + |\partial_x m_1|) u_1^2 \, dx\\
& \le \|m_1 u_1\|_{L^2} \|u_1\|_{L^2}+\|\partial_x m_1\|_{L^2} \|u_1\|^2_{L^4}\\
&\le \frac{1}{2} \e \|u_1\|^2_{L^2}+M(\|m\|_{\q H^1}+1)\|u\|_{\q H^1} \|u_1\|_{L^2}\\
&\quad \quad +M \|m\|_{\q H^1}\|u\|^{1/2}_{\q H^1}  \|u_1\|^{3/2}_{L^2}.
\end{align*}
We choose
\be
\label{choice_ep}
\e:= c.
\ee
Taking $s> 0$ given\footnote{Recall that $ \|u_1\|_{L^2}>0$.} by $s^2= \|u_1\|_{L^2}(\le  \|u_1\|_{H^1}\le \e<\infty)$ and denoting $a:=M(\|m\|_{\q H^1}+1)>0$, after dividing by $s^2$, we obtain
\[ \frac c2 s^2-a\|u\|^{1/2}_{\q H^1}s-a\|u\|_{\q H^1}\le 0. \]
The discriminant of the above quadratic form in $s$ is positive of order $O(\|u\|_{\q H^1} (\| m \|_{\q H^1}^2 + c^2))$, so both roots are real numbers of order $\|u\|^{1/2}_{\q H^1}$, and in particular $s=\|u_1\|^{1/2}_{L^2}$ is of order $\|u\|^{1/2}_{\q H^1}$ which yields the conclusion.
\end{proof}

\begin{nb}
If  $\tilde m$ and $m$ are not close in $\q H^1$, then one can not expect any bound of the form 
\[ \| m - \tilde m \|_{H^1} \le F(\| m -\tilde m \|_{\q H^1}) \]
 for any function $F: \m R_+ \to \m R_+$ which takes finite values. To see this, let us give an example of a family of magnetisations $m$, $\tilde m$, such that $\| m -\tilde m \|_{\q H^1}$ remains bounded, but $\| m - \tilde m \|_{H^1}$ is unbounded.  For this, let $\theta$ be a smooth function in $\R$ such that $\theta(x) = 0$ for $x < -1$, $\theta(x) = \pi$ for $x > 1$ and $\theta$ is increasing on $[-1,1]$. Consider $\tilde m = (1,0,0)$ constant and $m$ such that $m_1$ has two transitions from $-1$ to $1$ (given by $\cos \theta$) separated at a distance of order $R>2$: for example, we can choose such $m_1$ to be 
\[ m_1(x) = \begin{cases} 
\cos \theta(x+R) \quad \text{for} \quad x \le 0, \\
\cos \theta(-x+R) \quad \text{for} \quad x \ge 0. 
\end{cases} \]
Let $m_2 = \sqrt{1-m_1^2} \in H^1$ and $m_3 =0$. Then $m \in \q H^1$ and one sees that $\| m - \tilde m \|_{\q H^1}$ is  constant for large $R$, but of course $\| m - \tilde m \|_{L^2}=O(\sqrt{R}) \to +\infty$ as $R \to +\infty$.
\end{nb}

\section{Coercivity of a Schr\"odinger operator} \label{append:c}

\begin{lem}\label{lem:coer00}
Let $L = - \Delta +V$ where $V \in L^\infty(\RR^N)$ has the property that there exist $R>0$ and $c >0$ such that $V(x) \ge c$ for every $|x| \ge R$.
Assume that there exists $\phi \in H^1(\RR^N)$ such that$L \phi = 0$ in the sense of distributions and $\phi >0$ in $\RR^N$. Then $\ker L=\RR \phi$ and there exists $\lambda >0$ (small) such that for every $v \in H^1(\RR^N)$, $0\le (Lv,v)_{H^{-1}, H^1} \le \frac1{\lambda}\|v\|^2_{H^1}$,
\be
\label{est:Lvv_H1}
(Lv,v)_{H^{-1}, H^1} \ge 4 \lambda \| v \|_{H^1}^2 - \frac{1}{\lambda} (v, \phi)^2_{L^2}
\ee
and for every $v \in H^2(\RR^N)$, 
\be
\label{last_one}
\|Lv\|^2_{L^2} \ge 4 \lambda \| v \|_{H^2}^2 - \frac{1}{\lambda} (v, \phi)^2_{L^2}.
\ee
\end{lem}

\begin{proof} Note that if $\phi\in H^1(\RR^N)$ is a solution of $L\phi=0$, as $V\in L^\infty(\RR^N)$, then $\phi\in H^2(\RR^N)$. Moreover, standard elliptic regularity implies that $\phi\in W^{2,p}_{loc}(\RR^N)$ for every $p<\infty$, in particular, $\phi\in C^1(\RR^N)$. Therefore, the condition $\phi>0$ in $\RR^N$ makes sense pointwise in $\RR^N$. Also, we will assume without loss of generality that $\| \phi \|_{L^2}=1$.

\bigskip
 
\nd \emph{Step 1.} For every $v\in C^\infty_c(\RR^N)$, set $w:=\frac v \phi\in C_c^1(\RR^N)$ with compact support. Integrating by parts, we obtain:
\begin{align*}
\int_{\RR^N} Lv\cdot v \, dx &= \int_{\RR^N} \bigg(|\nabla v|^2 +V(x) v^2\bigg) \, dx=\int_{\RR^N} \bigg(|\nabla(\phi w)|^2 +V(x) \phi^2 w^2\bigg)\, dx\\
&=\int_{\RR^N} \bigg( \phi^2 |\nabla w|^2+w^2 |\nabla\phi|^2 +\frac 1 2  \nabla (\phi^2) \cdot \nabla (w^2) +V(x) \phi^2 w^2\bigg)\, dx\\
&{=} \int_{\RR^N} \bigg(\phi^2 |\nabla w|^2 +w^2 |\nabla\phi|^2 +V(x) \phi^2 w^2\bigg)\, dx -\frac 1 2 (\Delta (\phi^2), w^2) \\ 
&=\int_{\RR^N} \phi^2 | \nabla w|^2 \, dx+ (L\phi, w^2 \phi)=\int_{\RR^N} \phi^2 | \nabla w|^2 \, dx \ge 0,
\end{align*}
because $L\phi=0$ and $\Delta(\phi^2)=2\phi\Delta \phi+2|\nabla \phi|^2\in H^{-1}$ on any open bounded set. 
By density, for every $v\in H^1(\RR^N)$, there exists a sequence $v_n\in C^\infty_c(\RR^N)$ such that $v_n\to v$ and $\nabla v_n\to \nabla v$ in $L^2$ and  a.e. in $\RR^N$. In particular, $\nabla\big(\frac{v_n}{\phi} \big)\to \nabla\big(\frac{v}{\phi} \big)$   a.e. in $\RR^N$. Since $V\in L^\infty(\RR^N)$, it follows by Fatou's lemma:
\[ (Lv,v)=\lim_n (L v_n,v_n)=\liminf_n \int_{\RR^N} \phi^2 \big| \nabla\big(\frac{v_n}{\phi} \big)\big|^2 \, dx\ge \int_{\RR^N} \phi^2 \big| \nabla\big(\frac{v}{\phi} \big)\big|^2 \, dx\ge 0. \]
As a consequence, if $v\in H^1$ belongs to $\ker L$, since $\phi>0$ in $\RR^N$, then $v\in \RR \phi$. Also, as $V\in L^\infty$, we conclude that $0\le (Lv,v) \le \max\big(1, \|V\|_{L^\infty}\big)\|v\|^2_{H^1}$ for every $v\in H^1$.

\bigskip

\nd \emph{Step 2.}  Let 
\[ a = \inf \{ (Lv,v) : v \in H^1, \| v \|_{L^2} =1, (v ,\phi)_{L^2} =0 \}. \]
By Step 1, we know that $a\ge 0$. The aim is to prove that $a > 0$. 
For that, we consider a minimising sequence $(v_n)_n$ in $H^1$. As  for $v \in H^1$,
\[ (L|v|,|v|) = \int_{\RR^N} \bigg(|\nabla v|^2 +V(x) v^2\bigg) \, dx=(Lv,v), \] we can furthermore assume that $v_n \ge 0$ for all $n$. As  for some $C>0$, $V \ge -C$ in $\RR^N$, we infer that for any $v \in H^1$,
\[ \| v \|_{H^1}^2 \le (Lv,v) + (C+1) \|v \|_{L^2}^2, \]
so that $(v_n)_n$ is bounded in $H^1$. Up to extracting a subsequence, we can assume that for some $\underline v \in H^1$, $v_n \weak \underline  v$ in $H^1$,  $v_n \to \underline v$ in $L^2$ locally on compact sets and a.e. in $\RR^N$ and $(Lv_n, v_n)\to a$ as $n\to \infty$. In particular, $\underline v \ge 0$. Now as $(v_n,\phi)_{L^2} =0$, since $v_n \weak \underline  v$ in $L^2$ and $\phi\in L^2$, we deduce $(\underline v,\phi)_{L^2}=0$, and as $\phi >0$ and $\underline  v\ge 0$, we conclude $\underline  v=0$ a.e. in $\RR^N$. 

Now we argue by contradiction and assume that $a=0$, that is $(Lv_n,v_n) \to 0$. As $V+(C+c)  {\m 1_{|x| \le R}} \ge c$, we get from strong $L^2$ convergence in $B(0,R)$ of $v_n$ to $\underline v$, we see that
\begin{align*}
0 \le  \int_{m R} \left( |\nabla v_n|^2 + c |v_n|^2 \right) dx & \le \int_{\m R}  |\nabla v_n|^2 + (  V + (C+c) {\m 1_{|x| \le R}}) |v_n|^2  dx  \\
& \le (Lv_n,v_n) + (C+c) \int_{|x| \le R} |v_n|^2 dx \to 0.
\end{align*}
This implies that $\| v_n \|_{H^1} \to 0$, and the convergence $v_n \to \underline v =0$ is strong in $H^1$. But $\| v_n \|_{L^2} =1$, so this strong convergence also implies $\| \underline v \|_{L^2} =1$, a contradiction. Hence $a >0$.

\bigskip

\nd \emph{Step 3.} We now prove inequality \eqref{est:Lvv_H1}. By Step 2 and homogeneity, we get that if $v \in H^1$ and $(v,\phi)_{L^2}=0$, then
\[ (Lv,v) \ge a \| v \|_{L^2}^2. \]
Then, still assuming that $(v,\phi)=0$, since $V \ge -C$ in $\RR^N$, there holds for $b:=\frac{a}{a+C+1}\in (0,1)$:
\begin{align}
\nonumber (Lv,v) & = b \int_{\RR^N} |\nabla v|^2 + (1-b) (Lv,v) + b \int_{\RR^N} V v^2 \\
\label{n12}& \ge b \int_{\RR^N} |\nabla v|^2 + \int_{\RR^N} v^2 (a(1-b) + b V) \ge b \| v \|_{H^1}^2,
\end{align}
because $a(1-b) + b V \ge a(1-b) -b C = b$.

\bigskip

Let $v \in H^1$ (no longer assuming the orthogonality condition). We define the $L^2$ orthogonal decomposition $w:=v- (v, \phi)_{L^2} \phi$, so that $(w,\phi)_{L^2} =0$. 
As $L\phi=0$ and $L$ is self-adjoint, we compute using \eqref{n12}:
\begin{align*}
(Lv,v) &=(Lw,w)+2(v,\phi) (L\phi, w)+(v,\phi)^2 (L\phi,\phi)\\
& = (Lw,w) \ge b \| w \|_{H^1}^2 = b \| v -   (v, \phi) \phi \|_{H^1}^2 \\
& \ge b (\| v \|_{H^1} - |(v,\phi)|\| \phi \|_{H^1})^2 \ge b \left( \frac{1}{2} \| v \|_{H^1}^2 - |(v,\phi)|^2\| \phi \|_{H^1}^2 \right).
\end{align*}
where we used $(x-y)^2 \ge \frac{1}{2} x^2 - y^2$. The desired inequality follows for $\lambda:= \min\big(\frac{1}{b\|\phi\|^2_{H^1}}, \frac{b}8\big)$.

\bigskip

\nd \emph{Step 4.} We finally prove estimate \eqref{last_one}, with a similar strategy as \eqref{est:Lvv_H1}. By Step 2, we know that for every $v\in H^2(\RR^N)$ with $(v, \phi)_{L^2}=0$, 
\[ \|Lv\|_{L^2} \|v\|_{L^2}\ge (Lv,v)\ge a\|v\|_{L^2}^2 \]
yielding $\|Lv\|_{L^2}\ge a \|v\|_{L^2}$ with $a>0$. 

Let $\tilde b = \frac{a}{1+2a+2\|V \|_{L^\infty}} >0$, then, as in Step 3, we compute for $v\in H^2(\RR^N)$ such that $(v, \phi)_{L^2}=0$:
\begin{align*}
\|Lv\|^2_{L^2} & = 2 \tilde b \| \Delta v \|_{L^2}^2 + (1-2\tilde b) \| Lv \|_{L^2}^2 + 4 \tilde b (\Delta v,Vv) + 2 \tilde b \| V v \|_{L^2}^2 \\
& \ge \tilde b \| \Delta v \|_{L^2}^2 + a (1-2\tilde b) \| v \|_{L^2}^2 - 2 \tilde b \| V v \|_{L^2}^2 \\
& \ge \tilde b \| \Delta v \|_{L^2}^2 + (a (1-2\tilde b) -2\| V \|_{L^\infty} \tilde b)  \| v \|_{L^2}^2 \\
& \ge \tilde b (\| \Delta v \|_{L^2}^2 + \| v \|_{L^2}^2).
\end{align*}
because $(a (1-2\tilde b) -2\| V \|_{L^\infty} \tilde b) = \tilde b$. As $\| \Delta v \|_{L^2} + \| v \|_{L^2}$ controls the $\| v \|_{H^2}$ norm, we can conclude that for some $b'>0$,
\[ \| Lv \|_{L^2} \ge b' \| v \|_{H^2}. \]
For general $v \in H^2$ (without assuming that $(v, \phi)_{L^2}=0$), we set $w:=v- (v, \phi) \phi$. Then $w\in H^2$ (as $\phi\in H^2$) with $(w, \phi)_{L^2}=0$ and since $Lv=Lw$, the desired inequality \eqref{last_one} follows as in Step 3.
\end{proof}

\section{BV curves to \texorpdfstring{$\m S^2$}{S2}}

We claimed in the sketch of proof of Theorem \ref{th:lwp} that any $\m S^2$ valued map in $H^1([-A,A])$ can not have an image which is dense in any cap. Performing a change a chart, it suffices prove that a map in $H^1([0,1], [0,1]^2)$ does not have a dense image (in $[0,1]^2$). We actually provide a short quantitative proof in the slightly more general setting of $BV$ maps.

\begin{lem} \label{lem:bv_S2}
Let $m \in BV([0,1],[0,1]^2)$. Then  $m([0,1])$ is not dense in $[0,1]^2$.

More precisely, there exists a square $Q$ of length $\frac{1}{5(\| m \|_{BV}+1)}$ such that $Q \cap m([0,1]) = \varnothing$.
\end{lem}

\begin{proof}
Let $n \ge 2$ be an integer to be fixed later, and consider a partition of $[0,1[^2$ into $n^2$ squares $(Q_j)_{1 \le j \le n^2}$ of length $1/n$, of the form $[\alpha /n, (\alpha+1) /n[ \times [\beta /n,  (\beta+1) /n[$ for $0 \le \alpha, \beta \le n-1$.

For all $1 \le j \le n^2$, denote $Q_j'$ the smaller square of length $1/(3n)$ which has the same center as $Q_j$: if $Q_j = [\alpha /n, (\alpha+1) /n[ \times [\beta /n,  (\beta+1) /n[$ then $Q_j' = [(\alpha+1/3) /n, (\alpha+2/3) /n[ \times [(\beta+1/3) /n,  (\beta+2/3) /n[$.

Assume that for all $1 \le j \le n^2$, $Q_j' \cap m([0,1]) \ne \varnothing$, and consider $y_j \in [0,1]$ such that $m(y_j) \in Q_j' \cap m([0,1])$, and $(x_j)_j$ is the reordering of the $(y_j)_j$ that is $0 \le x_1 < x_2 < \cdots < x_{n^2} \le 1$.

By construction, for any $1 \le  j <k \le n^2$, the distance between $Q_j'$ and $Q_k'$ is at least $2/(3n)$ so that $d(m(x_j), m(x_{j+1})) \ge 2/(3n)$. Therefore, by definition of the $BV$ norm,
\[ \| m \|_{BV}  \ge \sum_{j=1}^{n^2-1} d(m(x_{j}), m(x_{j+1})) \ge \sum_{j=1}^{n^2-1} \frac{2}{3n} \ge \frac{2}{3} \left( n - \frac{1}{n} \right) \ge \frac{2n-1}{3}. \]
(Recall $n \ge 2$). As a consequence, $\| m \|_{BV} \ge 1$, and for any integer $n > (3 \| m \|_{BV}+ 1)/2$, there exists $j$ such that the $Q_j'$ such that $Q_j' \cap m([0,1]) = \varnothing$. The conclusion follows by choosing the integer $n_0$ in the interval $((3 \| m \|_{BV}+ 1)/2, 3 (\| m \|_{BV}+ 1)/2]$ (notice that $n_0 \ge 3$). Then  $Q_j'$ has length 
\[ \frac{1}{3n_0} \ge \frac{2}{9(\| m \|_{BV}+ 1)} \ge \frac{1}{5(\| m \|_{BV}+1)}. \qedhere \]
\end{proof}

\nocite{GRST16}
\nocite{KMV19}

\bibliography{biblio_llg}

\begin{thebibliography}{10}

\bibitem{AS92}
F.~Alouges and A.~Soyeur.
\newblock On global weak solutions for {L}andau-{L}ifshitz equations.
\newblock {\em Nonlinear Analysis, Theory, Methods \& Applications},
  18:1071--1084, 1992.

\bibitem{BIK07}
I.~Bejenaru, A.~D. Ionescu, and C.~E. Kenig.
\newblock Global existence and uniqueness of {S}chr\"{o}dinger maps in
  dimensions {$d\geq 4$}.
\newblock {\em Adv. Math.}, 215(1):263--291, 2007.

\bibitem{BIKT11}
I.~Bejenaru, A.~D. Ionescu, C.~E. Kenig, and D.~Tataru.
\newblock Global {S}chr\"{o}dinger maps in dimensions {$d\geq 2$}: small data
  in the critical {S}obolev spaces.
\newblock {\em Ann. of Math. (2)}, 173(3):1443--1506, 2011.

\bibitem{Car10}
Gilles Carbou.
\newblock Stability of static walls for a three-dimensional model of
  ferromagnetic material.
\newblock {\em J. Math. Pures Appl. (9)}, 93(2):183--203, 2010.

\bibitem{Car14}
Gilles Carbou.
\newblock Metastability of wall configurations in ferromagnetic nanowires.
\newblock {\em SIAM J. Math. Anal.}, 46(1):45--95, 2014.

\bibitem{CL06a}
Gilles Carbou and St\'{e}phane Labb\'{e}.
\newblock Stability for static walls in ferromagnetic nanowires.
\newblock {\em Discrete Contin. Dyn. Syst. Ser. B}, 6(2):273--290, 2006.

\bibitem{CL06b}
Gilles Carbou and St\'{e}phane Labb\'{e}.
\newblock Stabilization of walls for nano-wires of finite length.
\newblock {\em ESAIM Control Optim. Calc. Var.}, 18(1):1--21, 2012.

\bibitem{CLT08}
Gilles Carbou, St\'{e}phane Labb\'{e}, and Emmanuel Tr\'{e}lat.
\newblock Control of travelling walls in a ferromagnetic nanowire.
\newblock {\em Discrete Contin. Dyn. Syst. Ser. S}, 1(1):51--59, 2008.

\bibitem{chang2008spectra}
Shu-Ming Chang, Stephen Gustafson, Kenji Nakanishi, and Tai-Peng Tsai.
\newblock Spectra of linearized operators for nls solitary waves.
\newblock {\em SIAM Journal on Mathematical Analysis}, 39(4):1070--1111, 2008.

\bibitem{DIO}
Lukas D\"{o}ring, Radu Ignat, and Felix Otto.
\newblock A reduced model for domain walls in soft ferromagnetic films at the
  cross-over from symmetric to asymmetric wall types.
\newblock {\em J. Eur. Math. Soc. (JEMS)}, 16(7):1377--1422, 2014.

\bibitem{GGRS11}
Yan Gou, Arseni Goussev, J.~M. Robbins, and Valeriy Slastikov.
\newblock Stability of precessing domain walls in ferromagnetic nanowires.
\newblock {\em Phys. Rev. B}, 84:104445, Sep 2011.

\bibitem{GRS10}
Arseni Goussev, J.~M. Robbins, and Valeriy Slastikov.
\newblock Domain-wall motion in ferromagnetic nanowires driven by arbitrary
  time-dependent fields: An exact result.
\newblock {\em Phys. Rev. Lett.}, 104:147202, Apr 2010.

\bibitem{GRST16}
Arseni Goussev, J.~M. Robbins, Valeriy Slastikov, and Oleg~A. Tretiakov.
\newblock {D}zyaloshinskii-{M}oriya domain walls in magnetic nanotubes.
\newblock {\em Phys. Rev. B}, 93:054418, Feb 2016.

\bibitem{GdL17}
Susana Guti\'{e}rrez and Andr\'{e} de~Laire.
\newblock The {C}auchy problem for the {L}andau-{L}ifshitz-{G}ilbert equation
  in {BMO} and self-similar solutions.
\newblock {\em Nonlinearity}, 32(7):2522--2563, 2019.

\bibitem{HS96}
P.~D. Hislop and I.~M. Sigal.
\newblock {\em Introduction to Spectral Theory: With Applications to
  Schrödinger Operators}.
\newblock Applied Mathematical Sciences 113. Springer-Verlag New York, 1
  edition, 1996.

\bibitem{Ign}
Radu Ignat.
\newblock A {$\Gamma$}-convergence result for {N}\'{e}el walls in
  micromagnetics.
\newblock {\em Calc. Var. Partial Differential Equations}, 36(2):285--316,
  2009.

\bibitem{IgUniq}
Radu Ignat.
\newblock Uniqueness result for a weighted pendulum equation modeling domain
  walls in notched ferromagnetic nanowires.
\newblock {\em arXiv:2112.13358, 2021}, 2021.

\bibitem{IgnMer}
Radu Ignat and Beno\^{\i}t Merlet.
\newblock Lower bound for the energy of {B}loch walls in micromagnetics.
\newblock {\em Arch. Ration. Mech. Anal.}, 199(2):369--406, 2011.

\bibitem{IM_ARMA}
Radu Ignat and Roger Moser.
\newblock Interaction energy of domain walls in a nonlocal {G}inzburg-{L}andau
  type model from micromagnetics.
\newblock {\em Arch. Ration. Mech. Anal.}, 221(1):419--485, 2016.

\bibitem{IM_JDE}
Radu Ignat and Roger Moser.
\newblock N\'{e}el walls with prescribed winding number and how a nonlocal term
  can change the energy landscape.
\newblock {\em J. Differential Equations}, 263(9):5846--5901, 2017.

\bibitem{IM_Adv}
Radu Ignat and Roger Moser.
\newblock Energy minimisers of prescribed winding number in an
  {$\mathbb{S}^1$}-valued nonlocal {A}llen-{C}ahn type model.
\newblock {\em Adv. Math.}, 357:106819, 45, 2019.

\bibitem{Jiz11}
Rida Jizzini.
\newblock Optimal stability criterion for a wall in a ferromagnetic wire in a
  magnetic field.
\newblock {\em J. Differential Equations}, 250(8):3349--3361, 2011.

\bibitem{KMV19}
Stavros Komineas, Christof Melcher, and Stephanos Venakides.
\newblock Traveling domain walls in chiral ferromagnets.
\newblock {\em Nonlinearity}, 32(7):2392--2412, 2019.

\bibitem{KK_these}
Katharina K\"uhn.
\newblock Reversal modes in magnetic nanowires.
\newblock {\em Ph.D. Thesis, Max-Planck-Institut f\"ur Mathematik, Germany},
  2007.

\bibitem{Kuh09}
Katharina K\"uhn.
\newblock Moving domain walls in magnetic nanowires.
\newblock {\em Annales de l'I.H.P. Analyse non lin\'eaire}, 26(4):1345--1360,
  2009.

\bibitem{LN84}
N.~Lakshmanan and K.~Nakamura.
\newblock Landau-lifshitz equation of ferromagnetism: Exact treatment of the
  gilbert damping.
\newblock {\em Physical Review Letters}, 53(26):2497--2499, 1984.

\bibitem{MM01a}
Yvan Martel and Frank Merle.
\newblock Asymptotic stability of solitons for subcritical generalized {K}d{V}
  equations.
\newblock {\em Arch. Ration. Mech. Anal.}, 157(3):219--254, 2001.

\bibitem{SW74}
N.~L. Schryer and L.~R. Walker.
\newblock The motion of 180° domain walls in uniform dc magnetic fields.
\newblock {\em Journal of Applied Physics}, 45(12):5406--5421, 1974.

\bibitem{SlaSon}
Valeriy~V. Slastikov and Charles Sonnenberg.
\newblock Reduced models for ferromagnetic nanowires.
\newblock {\em IMA J. Appl. Math.}, 77(2):220--235, 2012.

\bibitem{Tak11}
Keisuke Takasao.
\newblock {Stability of travelling wave solutions for the Landau-Lifshitz
  equation}.
\newblock {\em Hiroshima Mathematical Journal}, 41(3):367 -- 388, 2011.

\bibitem{Tsu08}
Masayoshi Tsutsumi.
\newblock On the cauchy problem for the noncompact
  {L}andau-{L}ifshitz-{G}ilbert equation.
\newblock {\em J. Math. Ana Appl.}, 344(1):157--174, 2008.

\bibitem{Wei86}
Michael~I. Weinstein.
\newblock Modulational stability of ground states of nonlinear
  {S}chr\"{o}dinger equations.
\newblock {\em SIAM J. Math. Anal.}, 16(3):472--491, 1985.

\end{thebibliography}
\bibliographystyle{plain}

\bigskip
\bigskip

\normalsize

\begin{center}

{\scshape Rapha\"el C\^ote}\\
{\footnotesize
Universit\'e de Strasbourg\\
CNRS, IRMA UMR 7501\\
F-67000 Strasbourg, France\\
\email{cote@math.unistra.fr}
}

\bigskip

{\scshape Radu Ignat}\\
{\footnotesize
Institut de Math\'ematiques de Toulouse \& Institut Universitaire de France \\
UMR 5219, Universit\'e de Toulouse, CNRS, UPS, IMT \\
F-31062 Toulouse Cedex 9, France\\
\email{radu.ignat@math.univ-toulouse.fr}
}

\end{center}

\end{document}